\documentclass[reqno]{amsart}            
\usepackage{amsthm,amssymb,amsfonts,graphicx,cite,color}
\usepackage[bookmarks,bookmarksnumbered,bookmarksopen,colorlinks,backref,linkcolor=blue,citecolor=red]{hyperref}%
\usepackage{xcolor}
\usepackage{graphicx}
\usepackage{mathrsfs}

\newtheorem{theorem}{Theorem}[section]
\newtheorem{definition}{Definition}[section]
\newtheorem{lemma}{Lemma}[section]
\newtheorem{corollary}{Corollary}[section]
\newtheorem{remark}{Remark}[section]
\newtheorem{proposition}{Proposition}[section]
\numberwithin{equation}{section}

\begin{document}
\title[Spectral problem in a thin graph-like junction
with concentrated mass in the node]
{Asymptotic approximations for eigenvalues and eigenfunctions  of a spectral problem in a thin graph-like junction
with a concentrated mass in the node}
\author[T.A.~Mel'nyk]{Taras A. Mel'nyk
}
\address{\hskip-12pt  Faculty of Mathematics and Mechanics, Department of Mathematical Physics\\
Taras Shevchenko National University of Kyiv\\
Volodymyrska str. 64,\ 01601 Kyiv,  \ Ukraine
}
\email{melnyk@imath.kiev.ua}

\begin{abstract}
A spectral problem  is considered  in a thin $3D$ graph-like junction that consists of three  thin curvilinear cylinders that are joined through a domain (node) of the diameter $\mathcal{O}(\varepsilon),$ where $\varepsilon$ is a small parameter.
A concentrated mass with the density $\varepsilon^{-\alpha}$ $(\alpha \ge 0)$ is located in the node.
The asymptotic behaviour of  the eigenvalues and eigenfunctions
is studied as $\varepsilon \to 0,$ i.e. when the thin junction is shrunk into a graph.

There are five qualitatively different cases in the asymptotic behaviour $(\varepsilon \to 0)$ of the eigenelements  depending on the
 value of the parameter $\alpha.$ In this paper three cases are considered, namely, $\alpha =0,$ \   $\alpha \in (0, 1),$ and $\alpha =1.$

Using multiscale analysis, asymptotic approximations for eigenvalues and eigenfunctions are constructed and justified
with a predetermined accuracy with respect to the degree of $\varepsilon.$ For irrational  $\alpha \in (0, 1),$ a new kind of asymptotic expansions is introduced. These approximations show how to account the influence of local geometric inhomogeneity of the node and the concentrated mass in the corresponding limit spectral problems on the graph for different values of the parameter $\alpha.$
\end{abstract}

\keywords{perturbed spectral problem, thin graph-like junction, concentrated mass, asymptotic behaviour
\\
\hspace*{9pt} {\it MOS subject classification:} \
35B25, 47A75, 35B40, 35C20,   74K30
}

\maketitle
\tableofcontents

\section{Introduction}\label{Sect1}

In recent years, there has been an increasing interest in the study of spectral problems in thin graph-like structures, since such problems have various  applications in many fields, e.g., in quantum physics,  mathematical biology, chemistry and many others (see, e.g. \cite{Post-2012}). The main task is to study the possibility of approximating the spectra of different operators by the spectra of appropriate operators on the corresponding graph.
The convergence of  spectra for the Laplacians with different boundary conditions (Neumann, Dirichlet and Robin) at various levels of generality was proved in many papers. In order not to repeat myself, I refer readers to the works \cite{Kuchment2002,Post-2012}, where
fairly complete reviews on this topic has been presented. Here I mention several new papers that have appeared recently. Interesting multifarious transmission conditions were obtained in the limit passage for spectral problems on thin periodic honeycomb lattice
\cite{Naz-Ruo-Uus-2016}. In \cite{Antonio_Perez-2019} the authors studied the asymptotic behavior the high frequencies for the Laplace operator in a thin T-like shaped structure and gave a characterization of the asymptotic form of the
eigenfunctions originating these vibrations. About the study of  boundary-value problems in thin graph-like junctions  in other contexts, we refer to last new papers  \cite{Mel_Klev_AA-2019,Bun_Gau_Leo_2019,GauPanPia2016,P-P-Stokes-1-2015,P-P-Stokes-2-2015} and references there.

The book by O. Post \cite[Introduction]{Post-2012} raised a number of  questions.  Here are some of them.  If the diameter of the  nodal domain  in a thin graph-like junction  is  of the order $\mathcal{O}(\varepsilon)$, the node is not seen in the limit.
\begin{itemize}
  \item
How then to determine the influence of the  nodal domain  on the asymptotic behavior of the spectrum in this case?
  \item
Does the one-dimensional model yield correct information about  the original physical system?
 \end{itemize}
In the present paper these questions are  answered (see the conclusion in  Remark~\ref{remark-answers}).

The main goal of this paper is the study of the impact of the heavy  nodal domain in a thin graph-like junction on the asymptotic behaviour of the eigenvalues and eigenfunctions when this thin graph-like junction is shrunk into a graph.   To my knowledge, this is the first  article in the literature concerning the asymptotics of eigenvibrations of a thin graph-like junction with a  concentrated mass in the node. For this, the asymptotic approach developed in \cite{Mel_Klev_AA-2016,Mel_Klev_M2AS-2018,Mel_Klev_AA-2019} is used.
This approach makes it possible to take into account various factors (e.g. variable thickness of thin curvilinear cylinders,  inhomogeneous nonlinear boundary conditions, geometric characteristics of nodal domains and physical processes inside, etc.) in statements of  problems on graphs. In addition, it gives the better estimate for the difference between the solution of the starting  problem and the solution of the corresponding limit problem (see \cite{Mel_Klev_AA-2016}).

Qualitatively different cases in the asymptotic behaviour of the eigenelements  are discovered depending on the parameter $\alpha$ characterizing the value of the concentrated mass.
The main novelty is the construction of complete asymptotic expansions for the eigenvalues and asymptotic approximations for eigenfunctions with a predetermined accuracy. For irrational  $\alpha \in (0, 1),$ a new kind of asymptotic expansions is introduced.

\subsection{Statement of the problem}
The model thin star-shaped junction  $\Omega_\varepsilon$  consists of three thin curvilinear cylinders
$$
\Omega_\varepsilon^{(i)} =
  \Bigg\{
  x=(x_1, x_2, x_3)\in\Bbb{R}^3 \ : \
  \varepsilon \ell_0<x_i<\ell_i, \quad
  \sum\limits_{j=1}^3 (1-\delta_{ij})x_j^2<\varepsilon^2 h_i^2(x_i)
  \Bigg\}, \quad i=1,2,3,
$$
that are joined through a  nodal domain $\Omega_\varepsilon^{(0)}$ (referred in the sequel "node").
Here $\varepsilon$ is a small parameter; $\ell_0\in(0, \frac13), \ \ell_i\geq1, \ i=1,2,3;$
 the positive function $h_i$ belongs to the space $C^1 ([0, \ell_i])$
and it is equal to some constants in  neighborhoods of  $x_i=0$ and $x_i=\ell_i;$
the symbol $\delta_{ij}$ is the Kroneker delta, i.e.,
$\delta_{ii} = 1$ and $\delta_{ij} = 0$ if $i \neq j.$
The node $\Omega_\varepsilon^{(0)}$ (see Fig.~\ref{f2}) is formed by the homothetic transformation with coefficient $\varepsilon$ from a bounded domain
$\Xi^{(0)}\subset \Bbb R^3$,  i.e.,
$
\Omega_\varepsilon^{(0)} = \varepsilon\, \Xi^{(0)}.
$
In addition, we assume that its boundary contains the disks $\Upsilon_\varepsilon^{(i)} (\varepsilon\ell_0), \ i=1,2,3,$ where
$$
\Upsilon_\varepsilon^{(i)} (s) :=
\Big\{
 x\in\Bbb{R}^3 \, : \ x_i= s, \quad
 \sum\limits_{j=1}^3 (1-\delta_{ij})x_j^2<\varepsilon^2 h_i^2(s)
\Big\}.
$$
\begin{figure}[htbp]
\begin{center}
\includegraphics[width=6cm]{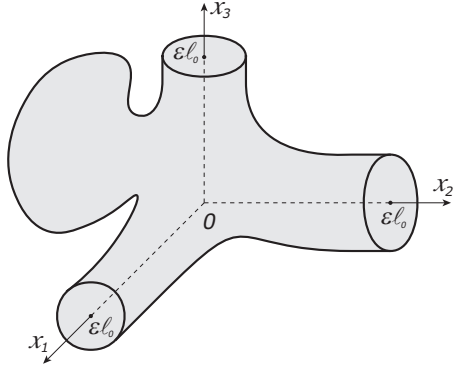}
\end{center}
\caption{The node $\Omega_\varepsilon^{(0)}$}\label{f2}
\end{figure}

Thus, the model thin graph-like junction  $\Omega_\varepsilon$  (see Fig.~\ref{f3})
is   the interior of the union
$
\bigcup_{i=0}^{3}\overline{\Omega_\varepsilon^{(i)}}
$
and we assume that it has the Lipschitz boundary.

\begin{remark}
  Thin graph-like junctions with arbitrary  number  and arbitrary  orientation of the thin cylinders can be also considered (see \cite{Klev_2019}, where a parabolic problem in a thin graph-like junction was studied). But in order to avoid technical and huge calculations, the thin junction $\Omega_\varepsilon,$ in which the cylinders are located along the coordinate axes,  is proposed.
\end{remark}

\begin{figure}[htbp]
\begin{center}
\includegraphics[width=8cm]{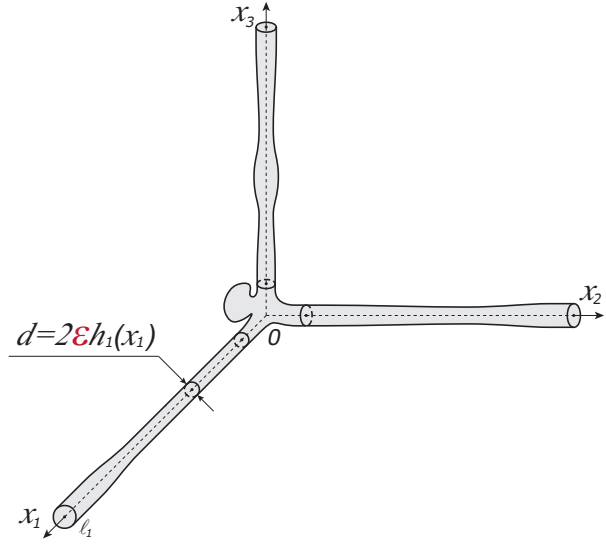}
\end{center}
\caption{The model thin graph-like  junction $\Omega_\varepsilon$}\label{f3}
\end{figure}

In $\Omega_\varepsilon$ we consider the following spectral problem:
\begin{equation}\label{1.1}
\left\{
\begin{array}{rcll}
   -\Delta_x  u^\varepsilon(x)& = & \lambda(\varepsilon) \,
 \rho_{\varepsilon}(x)\,  u^\varepsilon(x), &
                         \quad  x \in \Omega_{\varepsilon};
\\[2mm]
u^\varepsilon(x)& = & 0, & \quad   x \in {\Upsilon_{\varepsilon}^{(i)}(\ell_i)},  \ \  i=1,2,3;
\\[2mm]
 \partial_{\nu} u^\varepsilon(x) & = & 0 , & \quad x \in
\partial\Omega_\varepsilon \setminus \Big(\Upsilon_{\varepsilon}^{(1)}(\ell_1) \cup \Upsilon_{\varepsilon}^{(2)}(\ell_2)
\cup \Upsilon_{\varepsilon}^{(3)}(\ell_3)\Big);
\\[2mm]
\left[u^\varepsilon \right]_{|_{x_i=\varepsilon \ell_0}} &= & \left[\partial_{x_i}u^\varepsilon
\right]_{|_{x_i=\varepsilon \ell_0}}=0 , & \quad  x \in \Upsilon_\varepsilon^{(i)} (\varepsilon\ell_0), \ \ i=1,2,3.
\end{array}
\right.
\end{equation}
Here $\partial_{\nu}=\partial /\partial\nu$ is the outward normal
derivative; \ the brackets denote the jump of the enclosed
quantities;  \ the density
\begin{equation*}
\rho_{\varepsilon}(x)=
\left\{%
\begin{array}{ll}
    1, & x\in \Omega_{\varepsilon} \setminus \Omega_\varepsilon^{(0)},
\\[2mm]
\varepsilon^{-\alpha}  \varrho_0\big(\tfrac{x}{\varepsilon}\big), & x\in \Omega_\varepsilon^{(0)},
\end{array}%
\right.
\end{equation*}
where the parameter $\alpha\in [0, +\infty)$ (if
$\alpha > 0,$ then the concentrated mass is presented in the node  $\Omega_\varepsilon^{(0)}),$
 $\varrho_0$ is a smooth function in $\overline{\Xi^{(0)}}$ and such that
$$
\varrho_0(\xi) =
\left\{%
\begin{array}{ll}
    0, & \ \xi=(\xi_1, \xi_2, \xi_3) \in \Bbb R^3_{\xi}\setminus \Xi^{(0)},
\\[2mm]
0 < c_0 \le \varrho_0\big(\xi \big) \le c_1, & \  \xi \in \Xi^{(0)}.
\end{array}%
\right.
$$

Recall that  $u^\varepsilon \in\mathcal{H}_{\varepsilon} \setminus \{0\}$ is  an eigenfunction corresponding to the eigenvalue
$\lambda(\varepsilon)$ of  problem \eqref{1.1} if
\begin{equation}\label{identity}
  \langle u^\varepsilon , v \rangle_\varepsilon = \lambda(\varepsilon) \,
\left( u^\varepsilon , v \right)_{\mathcal{ V}_\varepsilon}
\qquad \forall \,   v \in \mathcal{H}_{\varepsilon},
\end{equation}
where $\mathcal{H}_{\varepsilon}$ is the Sobolev space
$$
\left\{ u\in H^{1}(\Omega_\varepsilon): \quad u|_{{\Upsilon_{\varepsilon}^{(i)}(\ell_i)}}=0 \ \ \text{ in sense of the trace},  \ \  i=1,2,3, \right\}
$$
with the scalar product
\begin{equation}\label{1.4}
\langle u, v \rangle_\varepsilon
 := \int_{\Omega_\varepsilon} \nabla u\cdot\nabla v\, dx \quad \forall \,  u, v \in \mathcal{H}_{\varepsilon},
\end{equation}
 and  $\mathcal{ V}_{\varepsilon}$ is the space $L^{2}(\Omega_\varepsilon)$
with the scalar product
$$
\left( u, v \right)_{\mathcal{ V}_\varepsilon}:= \sum_{i=1}^{3} \int_{\Omega^{(i)}_\varepsilon}  u \,  v \, dx + \varepsilon^{-\alpha}  \int_{\Omega^{(0)}_\varepsilon}  \varrho_0\big(\tfrac{x}{\varepsilon}\big)\, u \,  v \, dx.
$$

Let us define an operator $ A_{\varepsilon}: \mathcal{H}_{\varepsilon}\mapsto
\mathcal{H}_{\varepsilon}$  by the following equality
\begin{equation}\label{operator1}
\langle  A_{\varepsilon} u , v \rangle_\varepsilon = \left( u, v \right)_{\mathcal{ V}_\varepsilon}
 \quad \forall \ u, v \in \mathcal{ H}_{\varepsilon},
\end{equation}
It is easy to see that  operator $A_\varepsilon$ is self-adjoint, positive, and  compact due to  compactness of the imbedding $H^1(\Omega_\varepsilon) \subset L^2(\Omega_\varepsilon).$
Thus, the problem (\ref{1.1}) is equivalent to the spectral problem
$A_{\varepsilon} u = \lambda^{-1}(\varepsilon) \, u$ in $\mathcal{ H}_{\varepsilon}.$

Therefore,   for each fixed value of $\varepsilon$ there is a
sequence of eigenvalues
\begin{equation}
0  < \lambda_1(\varepsilon) < \lambda_2(\varepsilon) \le \ldots
\le \lambda_n(\varepsilon) \le \dots \to + \infty \quad ({\rm as}
\quad
  n \to +\infty) \label{1.2}
\end{equation}
of problem \eqref{1.1}, where each eigenvalue is counted as many times as its multiplicity.
The corresponding eigenfunctions $ \{u_n^\varepsilon\}_{n \in \Bbb N},$ which belong to $\mathcal{H}_{\varepsilon},$ can be
orthonormalized as follows
\begin{equation}\label{normalized}
\langle u_n^\varepsilon , u_m^\varepsilon \rangle_\varepsilon  = \varepsilon^2
  \delta_{n,m} , \quad \{n,\, m\} \subset \Bbb N.
\end{equation}
In addition, since the eigenfunctions are  defined up to the sign "$\pm$", we can choose them more precisely, namely
\begin{equation}\label{normalized+}
 \int_{\Upsilon^{(i)}_\varepsilon(\ell_1)}
  \big(\partial_{x_1} u_n^\varepsilon\big)(\ell_1,x_2, x_3)\, dx_2dx_3 > 0 .
\end{equation}

\medskip

{\bf The  aim} is to study the asymptotic behavior of the eigenvalues
$\{\lambda_n(\varepsilon)\}_{n\in \Bbb N}$ and the eigenfunctions
$\{ u_n^\varepsilon\}_{n\in \Bbb N}$ as $ \varepsilon \to
0,$ i.e., when the thin junction $\Omega_\varepsilon$ is  shrunk into the graph
$$
\mathcal{I} := \overline{I_1}\cup \overline{I_2}\cup \overline{I_3},
$$
where ${I}_i:=\{x=(x_1, x_2, x_3) \in \Bbb R^3: \ x_i\in (0, \ell_i), \ \overline{x}_i=(0,0)\}, \ i\in \{1, 2, 3\},$ and
$$
\overline{x}_i =
\left\{\begin{array}{lr}
(x_2, x_3), & i=1, \\
(x_1, x_3), & i=2, \\
(x_1, x_2), & i=3.
\end{array}\right.
$$
It should be noted that the limit process is accompanied by a
concentrated mass in the node. In fact, there are two kinds of perturbations in the problem (\ref{1.1}), namely, the
domain perturbation and the density perturbation. The influence of both those factors on the asymptotic behavior of the eigenvalues and eigenfunctions will be studied as well.

\subsection{Comments to the problem}

Vibrating systems with a concentration of masses have been studied for a long time.
It was experimentally established  that the concentrated  mass on a small set of
diameter $\mathcal{O}(\varepsilon)$ in a vibration system leads to the big reduction of the main frequency and to the
localization of corresponding eigenvibration near the concentrated  mass. A new approach in this field was
proposed  by E. S\'anchez-Palencia in the paper \cite{S-P}, where
the effect of local vibrations was mathematically described. After
this paper, many articles appeared dealing with different spectral problems
with concentrated masses. The reader can find widely presented bibliography on this topic  in  \cite{OYS,M3AS1993,CRAS1999,M3AS,Perez_Review,IZV,MMO,Perez-Nazarov, Mel_Che_2014}.
As follows from results obtained in those papers, the asymptotic behaviour of the eigenvalues and eigenfunctions  essentially
depends on the parameter $\alpha.$

Therefore, it is natural to expect that the concentrated mass in  the node  provokes crucial changes in the whole vibrating process in the thin graph-like junction $\Omega_\varepsilon,$ in particular, as we will see later,  it  rejects the traditional Kirchhoff transmission conditions at the vertex   of the graph $\mathcal{I}$ in the limit (as $\varepsilon \to 0)$ for some values of $\alpha.$

Since spectral problems in thin graph-liked junctions naturally  associated with spectral problems on graphs, we note some papers devoted to the study of spectral problems with concentrated masses on graphs.
The first work in this direction was the paper \cite{Oleinik_1988} by Oleinik, in which a spectral problem
$$
\frac{d^2 u^\varepsilon(x_1)}{dx^2_1}  +  \lambda(\varepsilon) \, \rho_{\varepsilon}(x_1)\,  u^\varepsilon(x_1) =0
                         \quad  \text{in} \ (a, b), \qquad u^\varepsilon(a) =u^\varepsilon(b) = 0,
$$
 where $a <0$ and $b>0,$ was considered with a concentrated mass on the interval $(-\varepsilon, \varepsilon).$
Five  qualitatively different cases in the asymptotic behaviour of the eigenvalues and eigenfunctions were  discovered, namely, $\alpha < 1,$ $\alpha = 1,$ $\alpha \in (1, 2),$ $\alpha =2,$ and $\alpha > 2.$ In each case the convergence of the eigenvalues and eigenfunctions was proved as $\varepsilon \to 0.$ Asymptotic expansions for the eigenvalues and eigenfunctions of the same problem were constructed in \cite{Gol_Naz_Ole_Sob_1989} for the integer values of $\alpha$ and $\alpha= \frac{3}{2}.$ The convergence theorems for eigenvalues and eigenfunctions of the fourth order differential equation on the interval $(a, b)$ with a concentrated mass on $(-\varepsilon, \varepsilon)$ were proved in \cite{Gol_1992}. Asymptotic behavior of the spectrum of the Sturm-Liouville problem  on metric  graphs  with perturbed density in small neighborhoods of vertices was studied in \cite{Gol_Gra_07} for
$\alpha < 1,$ $\alpha = 1,$ and $\alpha \in (1, 2).$

Similar situation is observed for the problem \eqref{1.1}.
In the present paper, only the cases $\alpha = 0,$ $\alpha \in (0, 1)$ and $\alpha = 1$ are considered.

In order to partially satisfy readers' desire to find out answers to the questions that are asked at the beginning of this section, I note  that this requires finding the second terms of the asymptotics (the full answers in Remark~\ref{remark-answers}). For example,  two-term asymptotics for the eigenvalue $\lambda_n(\varepsilon)$ in the case $\alpha = 0$ is
\begin{equation*}
\lambda_n(\varepsilon) = \Lambda_n + \varepsilon \mu_{1, n} + \mathcal{O}(\varepsilon^2)
\quad \text{as} \quad \varepsilon \to 0,
\end{equation*}
where
\begin{equation}\label{second_term}
 \mu_{1,n}   =  \left(\Lambda_n \, W_{n}^{(1)}(0)\right)^2  \left(  \sum\limits_{i=1}^3  \ell_0 h_i^2 (0)
     -  \frac{1}{\pi} \int\limits_{\Xi^{(0)}} \varrho_0(\xi)  \, d\xi \right)
 + \left(\dfrac{\boldsymbol{\delta_{1,n}^{(2)}}\, h^2_2(0)}{\ell_2}
 + \dfrac{ \boldsymbol{\delta_{1,n}^{(3)}} \, h^2_3(0)}{\ell_3}\right) \,   \Lambda_n W_{n}^{(1)}(0) .
\end{equation}
Here, $\Lambda_n$ is the $n$-th eigenvalue of the limit spectral problem \eqref{main} and
$\{W^{(i)}_n\}_{i=1}^3$ is the corresponding eigenfunction. From \eqref{second_term} it is possible to see the influence of the node and other characteristics of the thin junction $\Omega_\varepsilon$ on the asymptotic behavior of the spectrum of the problem \eqref{1.1} even for $\alpha = 0.$

The concentrated mass begins to impact the first terms of the asymptotic expansions if  $\alpha=1.$  This influence appears through the second Kirchhoff transmission condition in the limit spectral problem, namely,
$$
\sum_{i=1}^3  h_i^2 (0) \dfrac{d W^{(i)}}{dx_i} (0)  =
    -    \, \displaystyle{ \Lambda \, W^{(1)}(0) \, \frac{1}{\pi} \, \int_{\Xi^{(0)}} \varrho_0(\xi) \, d\xi}.
$$
It should be noted here that the spectral parameter $\Lambda$ is also present in the differential equations of the limit problem.

In contrast to asymptotic expansions constructed in \cite{Gol_Naz_Ole_Sob_1989},
the asymptotic approximations for the eigenfunctions of the problem \eqref{1.1}
consist of two parts, namely, the regular part of the asymptotics  located inside of each thin cylinder and the inner part presented in a neighborhood of the node (additional boundary-layer parts are needed in some cases  (see Sec.~\ref{Sec_3})).
Terms of the inner part of the asymptotics are special solutions of boundary-value problems in an unbounded domain with different outlets at infinity. It turns out they have polynomial growth at infinity and directly impact the regular asymptotics
through the matching conditions.  As a result, we derive the limit problem  $(\varepsilon =0)$ on the graph and the corresponding coupling conditions in the vertex, as well as the problems for the following terms of the regular asymptotics.

Another significant feature of this work is a new type of asymptotic expansions
in the case  when $\alpha$ is an irrational number from the interval  $(0, 1).$
For instance,  the following asymptotic ansatz
$$
   \sum\limits_{k=0}^{+\infty} \, \sum\limits_{p=0}^{+\infty} \, \varepsilon^{k - p\alpha } \, \mu_{k - p\alpha }
$$
is suggested for an eigenvalue $\lambda(\varepsilon)$ of the problem \eqref{1.1}.
This is a non-standard asymptotic expansion, since the better accuracy of the approximation we want, the more terms between integer powers of $\varepsilon$ must be determined, and with each step more and more (for more detail see Sec.~\ref{Sec_4}).
In classical asymptotic analysis, in order to construct the asymptotic approximation for a certain function, we successively calculate the coefficients (the next coefficient is determined through the previous ones) relative to a given scale (see e.g. \cite{Boyd,Miller}).

\smallskip

The rest of this paper is organized as follows. In Sec.~\ref{Sec_2}, some inequalities used in the other sections are proved.
Asymptotic approximations for the eigenvalues and eigenfunctions  of the problem \eqref{1.1} for $\alpha=0,$ $\alpha \in (0, 1),$ and $\alpha=1$ are constructed and justified in Sections~\ref{Sec_3}, \ref{Sec_4}, and \ref{alfa=1}, respectively.

\section{Additional useful inequalities and statements}\label{Sec_2}

Taking the uniform Dirichlet conditions on $\Upsilon_\varepsilon^{(1)}(\ell_1)$ into account  and using the linear extension operator $\mathcal{P}_{\varepsilon}$ (see \cite[Lemma 2.1]{ZAA99}) from the space $H^1(\Omega^{(1)}_\varepsilon)$ into
$H^1(B^{(1)}_{\varepsilon}),$ where
$$
B^{(1)}_{\varepsilon} := \{x:  \  x_1 \in (\varepsilon \ell_0, \ell_1), \   x_2 \in (-\varepsilon R_0, \varepsilon R_0), \ x_3 \in (-\varepsilon R_0, \varepsilon R_0)\}
$$
is a thin beam containing the thin curvilinear cylinder $\Omega^{(1)}_\varepsilon$ with its closure, we deduce the inequality
\begin{equation}\label{est1}
\int_{\Omega_\varepsilon^{(1)}} v^2 \, dx \le \int_{B_\varepsilon^{(1)}} \big(\mathcal{ P}_{\varepsilon} v\big)^2 \, dx \le
\ell_1^2 \int_{B_\varepsilon^{(1)}} \Big(\partial_{x_1} \big(\mathcal{P}_{\varepsilon} v\big) \Big)^2 \, dx \le C_1
\int_{\Omega_\varepsilon^{(1)}} |\nabla v|^2 \, dx
\end{equation}
for any function $v\in \mathcal{H}_\varepsilon .$ The same inequalities hold for $\Omega^{(2)}_\varepsilon$ and $\Omega^{(3)}_\varepsilon.$

\begin{remark}
Here and in what follows all constants $\{C_j\}$ and $\{c_j\}$
in inequalities are independent of the parameters $\varepsilon$ and $\alpha.$
\end{remark}

By the same way as in   \cite[Proposition 3.1]{Mel_Klev_M2AS-2018}  we prove the estimate
\begin{equation}\label{est2}
\int_{\Omega^{(0)}_\varepsilon} v^2 \, dx \le C_2 \Bigg( \varepsilon^2 \int_{\Omega^{(0)}_\varepsilon} |\nabla_{x}v|^2 \, dx +
  \varepsilon \sum_{i=1}^{3}\int_{\Upsilon_\varepsilon^{(i)} (\varepsilon\ell_0)} v^2 \, d\overline{x}_i \Bigg) \quad \forall \, v \in H^1(\Omega^{(0)}_\varepsilon).
\end{equation}
Then, again using the extension operator $\mathcal{P}_{\varepsilon}$ and taking into account the uniform Dirichlet conditions on $\{\Upsilon_\varepsilon^{(i)}(\ell_i)\}_{i=1}^3,$
we get  from \eqref{est2}  that
\begin{equation}\label{est3}
\int_{\Omega^{(0)}_\varepsilon} v^2 \, dx \le C_3 \varepsilon  \int_{\Omega_\varepsilon} |\nabla_{x}v|^2 \, dx \quad \forall \, v \in \mathcal{H_\varepsilon}.
\end{equation}

It follows from \eqref{est1} and \eqref{est3} that
\begin{equation}\label{est4}
\int_{\Omega_\varepsilon} v^2 \, dx \le C_4   \int_{\Omega_\varepsilon} |\nabla_{x}v|^2 \, dx \quad \forall \, v \in \mathcal{H_\varepsilon}.
\end{equation}
Thus, the usual norm$ \| v \|_1=(\int_{\Omega_{\varepsilon}}(|\nabla v|^2 + v^2)\,dx)^{1/2}$ and a new norm $\|\cdot\|_{\varepsilon} $ that is generated by the scalar product \eqref{1.4}  are uniformly  equivalent, i.e., there exist positive constants  $c_1, \varepsilon_0,$ such that for all $\varepsilon \in (0 , \varepsilon_0)$ the following inequalities hold:
$$
   \| v \|_{\varepsilon} \le \| v \|_1 \le c_1\|v \|_{\varepsilon}\,  \qquad \forall \, v \in \mathcal{H_\varepsilon}.
$$

It is easy to prove (see for more detail \cite{M-MMAS-2008}) the inequality
\begin{gather}\label{ineq1}
      \varepsilon \int_{\Gamma_\varepsilon^{(i)}} v^2 \, d\sigma_x
 \leq C_5 \int_{\Omega_\varepsilon^{(i)}} |\nabla_{x}v|^2 \, dx
 \qquad \forall \, v\in \mathcal{H_\varepsilon},
 \end{gather}
where $\Gamma_\varepsilon^{(i)} := \partial \Omega_\varepsilon^{(i)} \setminus \big(\Upsilon_{\varepsilon}^{(i)}(\varepsilon \ell_0) \cup \Upsilon_{\varepsilon}^{(i)}(\ell_i)\big)$ is the lateral surface of the cylinder $\Omega_\varepsilon^{(i)},  \ i\in \{1, 2, 3\}.$

\begin{proposition}\label{prop1}
For  any fixed $n\in \Bbb N $ there exists a constant $C_n$  such that
for all $\varepsilon$  the estimate
\begin{equation}\label{t0.1}
\lambda_{n}(\varepsilon) \le C_n
\end{equation}
 holds. In addition,  there exists a constant $c_1,$ which depends neither on $\varepsilon$ nor $n,$ such that
 for all $n\in \Bbb N$ and all $\varepsilon$
\begin{equation}\label{lower-est}
  \lambda_{n}(\varepsilon) \ge c_1 \cdot \left\{
                                        \begin{array}{ll}
                                          1, & \hbox{if} \ \  \alpha \in [0, 1], \\
                                          \varepsilon^{\alpha-1}, & \hbox{if} \ \ \alpha >1.
                                        \end{array}
                                      \right.
\end{equation}

If $\alpha>2,$ then   any fixed $n\in \Bbb N $
 \begin{equation}\label{t0.1_alfa>2}
\lambda_{n}(\varepsilon) \le C_n \, \varepsilon^{\alpha-2}.
\end{equation}
\end{proposition}

 \begin{proof} {\bf 1.}
  Let $ \mathcal{L}_n(\widetilde{\phi}_1,\ldots,\widetilde{\phi}_{n})$ be a
$n$-dimensional subspace of $\mathcal{H}_{\varepsilon},$ which is spanned on $n$ linearly independent functions
$$
\widetilde{\phi}_{k}=\left\{
     \begin{array}{rl}
   \phi_{k}(x_1) ,& \quad x\in \Omega_\varepsilon^{(1)}\cap \left\{x: \ x_1\in (\tfrac{\ell_1}{2} , \ell_1)\right\},
      \\[2mm]
0, & \quad x\in \Omega_\varepsilon\setminus\left(
 \Omega_\varepsilon^{(1)}\cap \left\{x: \ x_1\in (\tfrac{\ell_1}{2} , \ell_1)
 \right\}\right),
    \end{array}
     \right.
$$
where $\phi_1,\ldots,\phi_{n}$ are the eigenfunctions of
the spectral problem
\begin{equation*}
\left\{\begin{array}{l}
- \, \dfrac{d^2\phi_k(x_1)}{dx_1^2 } = \mu_k \,\phi_k(x_1), \quad  x_1\in (\tfrac{\ell_1}{2} , \ell_1),
\\\\
\phi_k(\tfrac{\ell_1}{2})= 0, \quad  \phi_k(\ell_1)= 0,
\end{array}\right.
\end{equation*}
which are orthonormalized  in $L^2\big((\tfrac{\ell_1}{2} , \ell_1)\big).$
Here, $0< \mu_1 < \mu_2 < \ldots < \mu_k < \ldots  \to +\infty$ as $ k \to +\infty.$

By virtue of the minimax principle for eigenvalues, we have
\begin{equation}\label{t0.3}
\begin{split}
  \lambda_{n}(\varepsilon) & = \min_{E \in {\bf E}_n}\;
   \max_{ v \in E\setminus \{0\} } \;
       \frac{\int_{\Omega_{\varepsilon}} |\nabla v |^2\, dx}
       {\sum_{i=1}^{3}\int_{\Omega^{(i)}_{\varepsilon}} v^2\,dx +
 \varepsilon^{-\alpha}\int_{\Omega^{(0)}_{\varepsilon}} \varrho_0(\frac{x}{\varepsilon}) \, v^2\,dx}
\le
  \min_{E \in {\bf E}_n}\;
   \max_{ v \in E\setminus \{0\} } \;
       \frac{\int_{\Omega_{\varepsilon}} |\nabla v |^2\ dx}
       {\int_{\Omega^{(1)}_{\varepsilon}} v^2\,dx}
              \\[2mm]
    & \le \,   \max_{0\ne v \in \mathcal{L}_n}  \;
\frac{ \int_{\Omega^{(1)}_{\varepsilon}} | v' |^2\, dx}{\int_{\Omega^{(1)}_{\varepsilon}} v^2\,dx}
\, \le \,  \frac{\max_{x_1\in  [0, \ell_1]} h_1^2(x_1)}{\min_{x_1\in [0, \ell_1]} h_1^2(x_1)}\max_{0\ne v \in \mathcal{L}_n}  \
\frac{\int_{(\tfrac{\ell_1}{2} , \ell_1)} (v')^2\,dx_1}{\int_{(\tfrac{\ell_1}{2} , \ell_1)} v^2\,dx_1} \le c_1 \mu_n.
\end{split}
\end{equation}
Here ${\bf E_n}$ is a set of all subspaces of $\mathcal{H}_{\varepsilon}$ with dimension $n.$

Now let us  prove the lower estimate. With the help of \eqref{est1} and \eqref{est3} , we have
\begin{equation*}
      \lambda_{n}(\varepsilon) \ge \lambda_{1 }(\varepsilon)=
      \min_{ v \in \mathcal{H}_{\varepsilon}\setminus \{0\} } \;
       \frac{\int_{\Omega_{\varepsilon}} |\nabla v |^2\, dx}
       {\sum_{i=1}^{3}\int_{\Omega^{(i)}_{\varepsilon}} v^2\,dx +
\varepsilon^{-\alpha}\int_{\Omega^{(0)}_{\varepsilon}}\varrho_0(\frac{x}{\varepsilon}) \,  v^2\,dx}
\ge  \frac{1}
       { C_2 + \varepsilon^{1-\alpha}  C_4 },
   \end{equation*}
from where it follows \eqref{lower-est}.

\medskip
\noindent
{\bf 2.}
Now let $ \mathcal{L}_n(\widetilde{\psi}_1,\ldots,\widetilde{\psi}_{n})$ be a
$n$-dimensional subspace of $\mathcal{H}_{\varepsilon},$ which is spanned on $n$ linearly independent functions
$$
\widetilde{\psi}_{k}=\left\{
     \begin{array}{rl}
   0,& \quad x\in \Omega_\varepsilon^{(i)}, \quad i\in \{1, 2,3\},
      \\[2mm]
\Psi\big(\frac{x}{\varepsilon}\big), & \quad x\in \Omega_\varepsilon^{(0)},
    \end{array}
     \right.
$$
where $\Psi_1,\ldots,\Psi_{n}$ are the eigenfunctions of
the spectral problem
\begin{equation}\label{spectr_junc_probl}
 \left\{\begin{array}{rcll}
  -\Delta_{\xi}{\Psi_{k}}(\xi) & = &
   \eta_k \, \Psi(\xi),   & \xi\in\Xi^{(0)},
 \\[1mm]
   \partial_{\nu_\xi}{\Psi_{k}}(\xi) & = &
   0,      & \xi\in \partial \Xi^{(0)}\setminus \Big( \bigcup_{i=1}^3\Upsilon_\varepsilon^{(i)} (\varepsilon\ell_0)\Big),
 \\[1mm]
   \Psi_k(\xi)  & = & 0,   & \xi \in \Upsilon_\varepsilon^{(i)} (\varepsilon\ell_0),  \quad i=1,2,3.
 \end{array}\right.
\end{equation}
which are orthonormalized  in $L^2\big(\Xi^{(0)}\big).$
Here, $0< \eta_1 < \eta_2 \le\ldots \le  \eta_k \le \ldots  \to +\infty$ as $ k \to +\infty.$

Again with the help of the minimax principle for eigenvalues, we get
\begin{equation}\label{t0.3+alfa>2}
  \lambda_{n}(\varepsilon)  \le \,   \varepsilon^{\alpha-2} \max_{0\ne \psi \in \mathcal{L}_n}  \;
\frac{ \int_{\Omega^{(0)}_{\varepsilon}} |\nabla_\xi \psi(\xi) |^2\Big|_{\xi=\frac{x}{\varepsilon}} \, dx}{\int_{\Omega^{(0)}_{\varepsilon}}
\psi^2\big(\frac{x}{\varepsilon} \big)^2\,dx}
\, = \,   \eta_n \, \varepsilon^{\alpha-2}.
\end{equation}
\end{proof}

From inequalities obtained in Proposition~\ref{prop1} we can expect  the following cases  in the asymptotic behaviour of the eigenvalues and eigenfunctions  depending on the parameter $\alpha$: $\alpha =0,$ $\alpha \in (0, 1),$ $\alpha =1,$ $\alpha \in (1, 2),$ $\alpha =2,$ and $\alpha > 2.$ In this paper we will consider the first three.

\section{Asymptotic expansions in the case $\alpha =0.$}\label{Sec_3}

We start from the case $\alpha = 0.$ In this case there is no big difference between  the node density and the density of the cylinders;
they can be even equal. Using  the approach of \cite{Mel_Klev_AA-2016},
the following ansatzes are proposed  for the approximation of an eigenfunction $u^\varepsilon$ (the index $n$ is omitted) of  the problem~\eqref{1.1}:
\begin{enumerate}
  \item for each $ i \in \{1,2,3\}$  the regular ansatz
  \begin{equation}\label{regul}
\mathcal{U}_\infty^{(i)} := w_0^{(i)} (x_i) + \varepsilon \, w_1^{(i)} (x_i)
 +  \sum\limits_{k=2}^{+\infty} \varepsilon^{k}
    \left(
    u_k^{(i)} \left( x_i, \dfrac{\overline{x}_i}{\varepsilon} \right) +  w_k^{(i)} (x_i)  \right)
\end{equation}
is  located inside of each thin cylinder $\Omega^{(i)}_\varepsilon $
and their terms  depend both on the corresponding longitudinal variable $x_i$ and so-called
``fast variables'' $\dfrac{\overline{x}_i}{\varepsilon};$
\item
 the boundary-layer ansatz
\begin{equation}\label{prim+}
{\Pi}_\infty^{(i)}
 := \sum\limits_{k=0}^{+\infty}\varepsilon^{k}\Pi_k^{(i)}
    \left(\frac{\ell_i-x_i}{\varepsilon},\frac{\overline{x}_i}{\varepsilon}\right)
\end{equation}
is located in a neighborhood of the base $\Upsilon_{\varepsilon}^{(i)} (\ell_i)$ of the thin cylinder $\Omega^{(i)}_\varepsilon;$
 \item
and the inner  one
  \begin{equation}\label{junc}
  \mathcal{N}^{(\infty)} :=\sum\limits_{k=0}^{+\infty}\varepsilon^k N_k\left(\frac{x}{\varepsilon}\right)
\end{equation}
is located in a neighborhood of the node $\Omega^{(0)}_\varepsilon$.
\end{enumerate}
Asymptotic expansion for the corresponding eigenvalue $\lambda(\varepsilon)$ (the index $n$ is omitted)  is sought in the form
\begin{equation}\label{exp-EVl}
\mathcal{L}^{(\infty)} := \sum\limits_{k=0}^{+\infty}\varepsilon^k \mu_k .
\end{equation}

\subsection{Regular parts of the asymptotics}\label{regul_asymp}

Formally substituting the series \eqref{regul} and  \eqref{exp-EVl} into the differential equation  of the problem \eqref{1.1}
in each thin cylinder $\Omega_\varepsilon^{(i)},$  we obtain
$$
 -  \sum\limits_{k=2}^{+\infty} \varepsilon^{k}
    \frac{d^2{u}_{k}^{(i)}}{d{x_i}^2} (x_i,\overline{\xi}_i)
 -  \sum\limits_{k=0}^{+\infty} \varepsilon^{k}
    \Delta_{\overline{\xi}_i}{u}_{k+2}^{(i)} (x_i,\overline{\xi}_i)
 -  \sum\limits_{k=0}^{+\infty} \varepsilon^{k}
    \frac{d^2 w_k^{(i)}}{d{x_i}^2} (x_i)
  \approx
$$
\begin{equation}\label{rel_1}
\mu_0 w_k^{(i)}(x_i) + \varepsilon \big( \mu_0 w_1^{(i)}(x_i) + \mu_1 w_0^{(i)}(x_i)\big)
+
\sum\limits_{k=2}^{+\infty} \varepsilon^{k}\sum_{m=0}^{k} \mu_m \left( u_{k-m}^{(i)} \left( x_i, \dfrac{\overline{x}_i}{\varepsilon} \right)
 +  w_{k-m}^{(i)} (x_i)\right),
\end{equation}
where $\xi_i = \dfrac{{x}_i}{\varepsilon}, \ \overline{\xi}_i = \dfrac{\overline{x}_i}{\varepsilon}.$

It is easy to calculate the outer unit normal to $\Gamma^{(i)}_\varepsilon:$
$$
{\boldsymbol{\nu}}^{(i)} (x_i, \ \overline{\xi}_i)
 =  \dfrac{1}{\sqrt{1 + \varepsilon^2 |h_i^\prime (x_i)|^2 \,}}
    \big( -\varepsilon h_i^\prime (x_i), \ \overline{\nu}_i (\overline{\xi}_i) \big)
 =
\left\{\begin{array}{lr}
    \dfrac{ \big(
   -\varepsilon h_1^\prime (x_1), \
    {\nu}_2^{(1)} (\overline{\xi}_1), \
    {\nu}_3^{(1)} (\overline{\xi}_1)
    \big) }{\sqrt{1 + \varepsilon^2 |h_1^\prime (x_1)|^2 \,}},
    \ & i=1,
\\
    \dfrac{ \big(
    {\nu}_1^{(2)} (\overline{\xi}_2), \
   -\varepsilon h_2^\prime (x_2), \
    {\nu}_3^{(2)} (\overline{\xi}_2)
    \big) }{\sqrt{1 + \varepsilon^2 |h_2^\prime (x_2)|^2 \,}},
    \ & i=2,
\\
    \dfrac{ \big(
    {\nu}_1^{(3)} (\overline{\xi}_3), \
    {\nu}_2^{(3)} (\overline{\xi}_3), \
   -\varepsilon h_3^\prime (x_3)
    \big) }{\sqrt{1 + \varepsilon^2 |h_3^\prime (x_3)|^2 \,}},
    \ & i=3,
\end{array}\right.
$$
where $\overline{\nu}_i (\overline{\xi}_i)$ is the outward normal to  the disk
${\Upsilon_{i}(x_i)}:=\{ \overline{\xi}_i\in\Bbb{R}^2 \ : \ |\overline{\xi}_i|< h_i(x_i) \}.$
Taking the view of the unit normal into account and putting the sum \eqref{regul}
into  the Neumann condition on $\Gamma^{(i)}_\varepsilon,$  we get
\begin{equation}\label{eq-2}
   \sum\limits_{k=1}^{+\infty}\varepsilon^{k}
    \partial_{\overline{\nu}_i (\overline{\xi}_i)}{u}_{k+1}^{(i)} (x_i,\overline{\xi}_i)  - h_i^\prime(x_i)
    \sum\limits_{k=3}^{+\infty}\varepsilon^{k}
    \frac{\partial {u}_{k-1}^{(i)}}{\partial x_i} (x_i,\overline{\xi}_i)
 -   h_i^\prime(x_i)
    \sum\limits_{k=1}^{+\infty}\varepsilon^{k}
    \frac{d w_{k-1}^{(i)}}{d{x_i}} (x_i) \approx 0.
\end{equation}

Equating  coefficients of the same powers of $\varepsilon$ in \eqref{rel_1} and \eqref{eq-2}, we deduce recurrent relations of  boundary-value problems for  determining the coefficients in \eqref{regul}. The problem problem for  $u_2^{(i)}$ is as follows:
\begin{equation}\label{regul_probl_2}
\left\{\begin{array}{rcll}
-\Delta_{\overline{\xi}_i}{u}_{2}^{(i)} (x_i,\overline{\xi}_i)
 & = &
\dfrac{d^{\,2} w_0^{(i)}}{d{x_i}^2} (x_i) + \mu_0 w_0^{(i)}(x_i),
 & \ \ \overline{\xi}_i\in\Upsilon_i (x_i),
\\[2mm]
-\partial_{\nu_{\overline{\xi}_i}}{u}_{2}^{(i)}(x_i,\overline{\xi}_i)
 & = &
- \ h_i^\prime (x_i)\dfrac{d w_0^{(i)}}{d{x_i}} (x_i) ,
 & \ \ \overline{\xi}_i\in\partial\Upsilon_i(x_i),
\\[2mm]
\langle u_2^{(i)} (x_i,\cdot) \rangle_{\Upsilon_i (x_i)}
 & = &
0.
 &
\end{array}\right.
\end{equation}
Here, the variable $x_i$ is regarded as a parameter from the interval $ \ I_\varepsilon^{(i)} =\{x: \ x_i\in  (\varepsilon \ell_0, \ell_i), \ \overline{x_i}=(0,0)\},$
$$
\langle u(x_i,\cdot) \rangle_{\Upsilon_i(x_i)} :=  \int_{\Upsilon_i(x_i)}u (x_i,\overline{\xi}_i)\, d{\overline{\xi}_i}, \quad \ i=1,2,3.
$$

For each $i\in \{1, 2, 3\},$ the problem (\ref{regul_probl_2}) is the inhomogeneous Neumann problem for the Poisson equation in the disk $\Upsilon_i(x_i)$ with respect to the variables ${\overline{\xi}_i}\in\Upsilon_i(x_i).$  Writing down
the necessary and sufficient condition for the solvability of (\ref{regul_probl_2}), we get the following differential equation:
\begin{equation}\label{omega_probl_2}
 -  \dfrac{d}{d{x_i}}\left(h_i^2(x_i)\frac{dw_0^{(i)}}{d{x_i}}(x_i)\right)
 =  \mu_0\, h_i^2(x_i)\, w_0^{(i)}(x_i),
\quad x_i\in I_\varepsilon^{(i)}.
\end{equation}
Let us assume that there exists a number $\mu_0$ and a function $w_0^{(i)}\neq 0$ satisfying the  equation \eqref{omega_probl_2})
(the existence will be justified later). Then,  a solution to the problem (\ref{regul_probl_2}) exists and the third relation in (\ref{regul_probl_2}) supplies its  uniqueness.

To determine the coefficient $u_3^{(i)},$ we obtain the following problem:
\begin{equation}\label{regul_probl_3}
\left\{\begin{array}{rcll}
-\Delta_{\overline{\xi}_i}{u}_{3}^{(i)}(x_i,\overline{\xi}_i)
 & = &
\dfrac{d^{\,2}w_1^{(i)}}{d{x_i}^2}(x_i) + \mu_0\, w_1^{(i)}(x_i) + \mu_1\,  w_0^{(i)}(x_i),
 &\quad
\overline{\xi}_i\in\Upsilon_i(x_i),
 \\[2mm]
-\partial_{\nu_{\overline{\xi}_i}}{u}_{3}^{(i)}(x_i,\overline{\xi}_i)
 & = &
- \ h_i^\prime(x_i)\dfrac{dw_1^{(i)}}{d{x_i}}(x_i),
 &\quad
\overline{\xi}_i\in\partial\Upsilon_i(x_i),
 \\[2mm]
\langle u_3^{(i)}(x_i,\cdot) \rangle_{\Upsilon_i(x_i)}
 & = &
0,
 &
\end{array}\right.
\end{equation}
for each $ i\in\{1,2,3\}.$ Repeating the above reasoning, we find
\begin{equation}\label{omega_probl_3}
 -   \dfrac{d}{d{x_i}}\left(h_i^2(x_i)\frac{d\omega_1^{(i)}}{d{x_i}}(x_i)\right)
 =   \mu_0\, h_i^2(x_i)\, w_1^{(i)}(x_i) + \mu_1\, h_i^2(x_i)\,  w_0^{(i)}(x_i),
\quad x_i\in I_\varepsilon^{(i)}, \quad i=1,2,3.
\end{equation}

For the coefficients  $\{u_k^{(i)}:  \ k\geq 4, \ i=1, 2, 3\},$ we obtain the following boundary-value problems:
\begin{equation}\label{regul_probl_k}
\left\{\begin{array}{rcll}
-\Delta_{\overline{\xi}_i}{u}_{k}^{(i)}(x_i,\overline{\xi}_i)
 & = &
\dfrac{d^{\,2}w_{k-2}^{(i)}}{d{x_i}^2}(x_i) + \dfrac{\partial^{\,2}u_{k-2}^{(i)}}{\partial{x_i}^2}(x_i,\overline{\xi}_i) +
\mu_{k-3}\, w_1^{(i)}(x_i) + \mu_{k-2}\,  w_0^{(i)}(x_i)
 &
\\[3mm]
 &   &
+ \sum_{m=0}^{k-4} \mu_m \left(u_{k-2-m}^{(i)} \left( x_i, \dfrac{\overline{x}_i}{\varepsilon} \right) +  w_{k-2-m}^{(i)} (x_i)  \right),
 & \overline{\xi}_i\in\partial\Upsilon_i(x_i),
 \\[3mm]
-\partial_{\nu_{\overline{\xi}_i}}{u}_{k}^{(i)}(x_i,\overline{\xi}_i)
 & = &
- \ h_i^\prime(x_i) \left(\dfrac{d w_{k-2}^{(i)}}{d{x_i}}(x_i) + \dfrac{\partial{u}_{k-2}^{(i)}}{\partial{x_i}}(x_i,\overline{\xi}_i)\right)
 & \overline{\xi}_i\in\Upsilon_i(x_i),
 \\[2mm]
\langle u_k^{(i)}(x_i,\cdot) \rangle_{\Upsilon_i(x_i)} & = & 0. &
\end{array}\right.
\end{equation}
Writing down  the solvability condition for the problem (\ref{regul_probl_k}), we see  that $w_{k-2}^{(i)}$ should be  a solution to the ordinary differential equation
\begin{equation}\label{omega_probl_k}
 -   \dfrac{d}{d{x_i}}\left(h_i^2(x_i)\dfrac{d w_{k-2}^{(i)}}{d{x_i}}(x_i)\right)
 =  \mu_0 \, h_i^2(x_i) \,  w_{k-2}^{(i)} (x_i) +
\sum_{m=1}^{k-2} \mu_m \, h_i^2(x_i) \,  w_{k-2-m}^{(i)} (x_i),
    \quad x_i\in I_\varepsilon^{(i)}.
\end{equation}
Here $ k\geq4 , \ i=1, 2, 3.$

\begin{remark}
Boundary conditions for the differential equations (\ref{omega_probl_2}), (\ref{omega_probl_3}) and (\ref{omega_probl_k}) are a priori unknown. They will be determined in the process of constructing the formal asymptotics.

The solution to the problem (\ref{regul_probl_k}) is uniquely determined if we find all the coefficients
$u_2^{(i)},\dots,u_{k-1}^{(i)}$ and $w_0^{(i)},\ldots, w_{k-2}^{(i)}$ of the expansion (\ref{regul}) and $\mu_0, \mu_1, \ldots, \mu_{k-2}$ of the
expansion \eqref{exp-EVl}. We show the solvability of this recursive procedure in the unique way in  Subsection~\ref{justification}.
\end{remark}

 \subsection{Boundary-layer part of the asymptotics}

Here we construct the boundary-layer parts of the asymptotics, which  compensate residuals from the regular ones
on the bases of the thin cylinders $\Omega^{(i)}_\varepsilon, \ i=1, 2, 3.$

Substituting the series \eqref{prim+} into \eqref{1.1} and collecting coefficients with the same powers of $\varepsilon$, we get the following mixed boundary-value problems:
\begin{equation}\label{prim+probl}
 \left\{\begin{array}{rcll}
  -\Delta_{\xi_i^*, \overline{\xi}_i} \Pi_k^{(i)}(\xi_i^*,\overline{\xi}_i) & =
   & 0,
   & \xi_i^*\in(0,+\infty), \quad \overline{\xi}_i\in\Upsilon_i(\ell_i),
   \\[2mm]
  -\partial_{\nu_{\overline{\xi}_i}} \Pi_k^{(i)}(\xi_i^*,\overline{\xi}_i) & =
   & 0,
   & \xi_i^*\in(0,+\infty), \quad \overline{\xi}_i\in\partial\Upsilon_i(\ell_i),
   \\[2mm]
  \Pi_k^{(i)}(0,\overline{\xi}_i) & =
   & \Phi_k^{(i)}(\overline{\xi}_i),
   & \overline{\xi}_i\in\Upsilon_i(\ell_i),
   \\[2mm]
  \Pi_k^{(i)}(\xi_i^*,\overline{\xi}_i) & \to
   & 0,
   & \xi_i^*\to+\infty, \quad \overline{\xi}_i\in\Upsilon_i(\ell_i),
 \end{array}\right.
\end{equation}\\[1mm]
where $\displaystyle{ \xi_i^* = \frac{\ell_i-x_i}{\varepsilon}},$ $\displaystyle{\overline{\xi}_i = \frac{\overline{x}_i}{\varepsilon}};$
\ $\Phi_k^{(i)} = -w_{k}^{(i)}(\ell_i), \ k=0,1;$ \ $\Phi_k^{(i)}(\overline{\xi}_i)
 = -u_k^{(i)}(\ell_i,\overline{\xi}_i) -  w_{k}^{(i)}(\ell_i),$ $k\geq 2,$ $ k\in \Bbb N.$
Using the method of separation of variables, we determine the solution
\begin{equation}\label{view_solution}
\Pi_k^{(i)}(\xi_i^*,\overline{\xi}_i)
 =  a_{k,0}^{(i)}
 +  \sum\limits_{p=1}^{+\infty}a_{k,p}^{(i)} \,
    \Theta_p^{(i)}(\overline{\xi}_i) \,
    \exp({-\eta_p^{(i)}\xi_i^*})
\end{equation}
to the problem (\ref{prim+probl}) at a fixed index $k,$
where
$$
a_{k,p}^{(i)}
 =      \int_{\Upsilon_i(\ell_i)}
    \Phi_k^{(i)}(\overline{\xi}_i)\,
    \Theta_p^{(i)}(\overline{\xi}_i) \, d\overline{\xi}_i,
$$
$$
a_{k,0}^{(i)}
 =  \dfrac{1}{|\Upsilon_i(\ell_i)|}
    \int_{\Upsilon_i(\ell_i)}
    \Phi_k^{(i)}(\overline{\xi}_i) \, d\overline{\xi}_i
 = -\dfrac{1}{\pi h_i^2(\ell_i)}
    \int_{\Upsilon_i(\ell_i)}
    u_k^{(i)} (\ell_i,\overline{\xi}_i) \, d\overline{\xi}_i
 -  w_{k}^{(i)}(\ell_i)
 = - w_{k}^{(i)}(\ell_i).
$$
Here $\Theta_0^{(i)}\equiv1,$ $ \ \{\Theta_p^{(i)}\}_{p\in \Bbb N}$  are
orthonormalized  in $L^2(\Upsilon_i(\ell_i))$ eigenfunctions of the spectral problem
\begin{equation}\label{spectral_problem}
\left\{\begin{array}{lcll}
- \Delta_{\overline{\xi}_i} \Theta^{(i)}_p  & = &  (\eta^{(i)}_p)^2\Theta^{(i)} &
\mbox{in} \ \Upsilon_i(\ell_i),
\\
\partial_{\nu_{\overline{\xi}_i}} \Theta^{(i)}_p & = & 0 &
\mbox{on} \ \partial\Upsilon_i(\ell_i),
\end{array}\right.
\end{equation}
and $\eta_0^{(i)}=0 <\eta_1^{(i)} \le \eta_2^{(i)}\le\ldots\le\eta_p^{(i)}\le\ldots .$

From the fourth condition in (\ref{prim+probl}) it follows  that coefficient $a_{k,0}^{(i)}$
must be equal to $0.$ As a result, we arrive at  the following boundary conditions for the functions $\{w_{k}^{(i)}\}$:
\begin{equation}\label{bv_left}
w_{k}^{(i)}(\ell_i)=0, \quad k\in\Bbb{N}_0, \quad i=1,2,3.
\end{equation}

\begin{remark}\label{asymp-Pi}
Since $\Phi_0^{(i)}\equiv\Phi_1^{(i)}\equiv0$, we conclude that $\Pi_{0}^{(i)}\equiv\Pi_{1}^{(i)}\equiv0, \quad i=1,2,3.$
Moreover, from representation (\ref{view_solution}) with regard to \eqref{bv_left} it follows that
\begin{equation}\label{as_estimates}
\Pi_k^{(i)}(\xi_i^*,\overline{\xi}_i)
 = \mathcal{O}(\exp(- \eta_1^{(i)}\xi_i^*))
\quad \mbox{as} \quad \xi_i^*\to+\infty, \quad i=1,2,3.
\end{equation}
\end{remark}


 \subsection{Inner part of the asymptotics}\label{inner_asymp}
 To obtain conditions on  the functions $\{\omega_k^{(i)}\}, \ i=1,2,3$ at the point $0,$ we need to run the inner part~ (\ref{junc}) of the asymptotics in a neighborhood of the node $\Omega^{(0)}_\varepsilon$. For this purpose we pass to the variables $\xi=\frac{x}{\varepsilon}.$ Letting  $\varepsilon$ to $0,$ we see  that the domain $\Omega_\varepsilon$ is transformed into the unbounded domain $\Xi$ that  is the union of the domain~$\Xi^{(0)}$ and three semibounded cylinders
$$
\Xi^{(i)}
 =  \{ \xi=(\xi_1,\xi_2,\xi_3)\in\Bbb R^3 \ :
    \quad  \ell_0<\xi_i<+\infty,
    \quad |\overline{\xi}_i|<h_i(0) \},
\qquad i=1,2,3,
$$
i.e., $\Xi$ is the interior of the set $\bigcup_{i=0}^3\overline{\Xi^{(i)}}$ (see Fig.~\ref{Fig-5}).

\begin{figure}[htbp]
\begin{center}
\includegraphics[width=8cm]{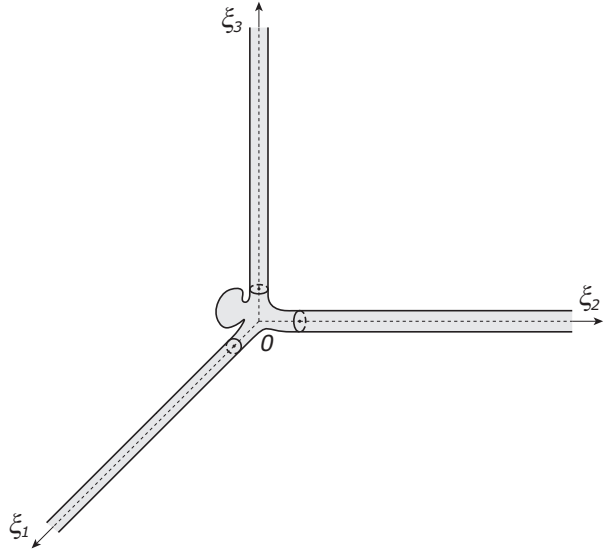}
\vskip - 10pt
\caption{Domain $\Xi$}\label{Fig-5}
\end{center}
\end{figure}

Substituting the series \eqref{junc} and \eqref{exp-EVl}  into the problem \eqref{1.1} and equating coefficients at the same powers of $\varepsilon$, we derive the following relations for $\{N_k\}:$
\begin{equation}\label{junc_probl_n}
 \left\{\begin{array}{rcll}
  -\Delta_{\xi}{N_k}(\xi) & = & \mathcal{R}(\xi) \sum_{m=0}^{k-2} \mu_m \, N_{k-2-m}(\xi),      &
   \quad \xi\in\Xi
\\[2mm]
   \partial_{{\boldsymbol \nu}_\xi}{N_k}(\xi) & = &
   0,                               &
   \quad \xi\in \partial \Xi
\\[2mm]
   N_k(\xi)                                                               & \sim &
   w^{(i)}_{k}(0) + \Psi^{(i)}_{k}(\xi),                                    &
   \quad \xi_i \to +\infty, \ \ \xi  \in \Xi^{(i)}, \quad i=1,2,3
   \\[2mm]
\left[N_k \right]_{|_{\xi_i= \ell_0}} &= & \left[\partial_{\xi_i}N_k
\right]_{|_{\xi_i= \ell_0}}=0 , & \quad  x \in \Upsilon_i(\ell_0), \ \ i=1,2,3,
 \end{array}\right.
\end{equation}
where
$$
\mathcal{R}(\xi) =
\left\{
  \begin{array}{ll}
    1, & \xi \in \Xi \setminus \Xi^{(0)},
\\
    \varrho_0(\xi), & \xi \in \Xi^{(0)}.
  \end{array}
\right.
$$
Relations in the third line of \eqref{junc_probl_n} appear by matching the regular and inner asymptotics in a neighborhood of the node, namely the asymptotics of the terms $\{N_k\}$ as $\xi_i \to +\infty$ have to coincide with the corresponding asymptotics of  terms of the regular expansions (\ref{regul}) as $x_i =\varepsilon \xi_i \to +0, \ i=1,2,3,$ respectively.
Expanding each term of the regular asymptotics in the Taylor series at the points $x_i=0, \ i=1,2,3,$
and collecting the coefficients of the same powers of $\varepsilon,$  we get
\begin{equation}\label{Psi_k}
\Psi_{0}^{(i)} \equiv 0, \qquad
\Psi_{k}^{(i)}(\xi_i)
 =   \sum\limits_{j=1}^{k} \dfrac{\xi_i^j}{j!}
     \dfrac{d^j w_{k-j}^{(i)}}{dx_i^j} (0),
\quad  i=1,2,3, \ \ k \in \Bbb N.
\end{equation}

\begin{remark}
Since the function $h_i$  is equal to some constant in a neighborhood of  $x_i=0,$
coefficients $\{u_{k}^{(i)}\}_{k\in \Bbb N}$ vanish in this neighborhood.

To simplify formulas, we will not prescribe the conjugation conditions  on  $\Upsilon_i(\ell_0), \  i=1,2,3,$ in the future.
\end{remark}

We look for a solution to the problem \eqref{junc_probl_n} with a fixed index  $k\in \Bbb N$  in the form
\begin{equation}\label{new-solution}
N_k(\xi) = \sum\limits_{i=1}^3 \Psi_{k}^{(i)}(\xi_i)\chi_i(\xi_i) + \widetilde{N}_k(\xi),
\end{equation}
where $ \chi_i \in C^{\infty}(\Bbb{R}_+),\ 0\leq \chi_i \leq1$ and
$$
\chi_i(\xi_i) =
\left\{\begin{array}{ll}
 0, & \ \  \xi_i \in [0,  1+\ell_0],
\\[2mm]
 1, &  \ \  \xi_i \geq 2+\ell_0,
\end{array}\right. \qquad i=1,2,3.
$$
Then $\widetilde{N}_k$ should be  a  solution to the problem
\begin{equation}\label{junc_probl_general}
 \left\{\begin{array}{rcll}
  -\Delta_{\xi}{\widetilde{N}_k}(\xi) & = &
   \widetilde{F}_k(\xi),                  &
   \quad \xi\in\Xi,
\\[2mm]
   \partial_{\boldsymbol{\nu}_\xi}{\widetilde{N}_k}(\xi) & = &
   0,                                           &
   \quad \xi\in \partial \Xi,
 \end{array}\right.
\end{equation}
and satisfy the conditions:
\begin{equation}\label{junc_probl_general+cond}
   \widetilde{N}_k(\xi)  \rightarrow  w^{(i)}_{k}(0)
   \quad \text{as} \quad \xi_i \to +\infty, \ \  \xi  \in \Xi^{(i)}, \quad i=1,2,3,
\end{equation}
where
$$
\widetilde{F}_1(\xi)
 =  \sum\limits_{i=1}^3
    \Big(
    \xi_i\dfrac{d w_0^{(i)}}{dx_i}(0) \chi_i^{\prime\prime}(\xi_i)
 + 2\dfrac{d w_0^{(i)}}{dx_i}(0) \chi_i^{\prime}(\xi_i)
    \Big),
$$
$$
\widetilde{F}_k(\xi)
 =  \sum\limits_{i=1}^3
    \Bigg[ \Big(\Psi_{k}^{(i)}(\xi_i)\chi^\prime_i(\xi_i)\Big)^\prime + \Big(\Psi_{k}^{(i)}(\xi_i)\Big)^\prime \chi^\prime_i(\xi_i)
+ \Big(\Psi_{k}^{(i)}(\xi_i)\Big)^{\prime\prime}  \chi_i(\xi_i)\Bigg]
$$
$$
 +  \mathcal{R}(\xi)\sum_{m=0}^{k-2} \mu_n \, N_{k-2-m}(\xi), \quad k \geq 2.
$$

The existence of a solution to the  problem (\ref{junc_probl_general}) in the corresponding energetic space can be obtained from general results about the
asymptotic behavior of solutions to elliptic problems in domains with different exits to infinity \cite{Ko-Ol,La-Pa,Na-Pla,Naz96}.
We will use approach proposed in \cite{Naz96, ZAA99}.

 Let $C^{\infty}_{0,\xi}(\overline{\Xi})$ be a space of functions infinitely differentiable in $\overline{\Xi}$ and finite with
respect to  $\xi$, i.e.,
$$
\forall \,v\in C^{\infty}_{0,\xi}(\overline{\Xi}) \quad \exists \,R>0 \quad \forall \, \xi\in\overline{\Xi} \quad \xi_i \geq R, \ \ i=1,2,3 \, : \quad v(\xi)=0.
$$
We now define a  space  $\mathcal{H} := \overline{\left( C^{\infty}_{0,\xi}(\overline{\Xi}), \ \| \cdot \|_\mathcal{H} \right)}$, where
$$
\| v \|_\mathcal{H}
 =  \sqrt{\int_\Xi|\nabla v(\xi)|^2 \, d\xi + \int_\Xi |v(\xi)|^2 |\rho(\xi)|^2 \, d\xi \, } ,
$$
and the weight function  $ \rho \in C^{\infty}(\Bbb{R}^3),\ 0\leq \rho \leq1$ and
$$
\rho (\xi) =
\left\{\begin{array}{ll}
1,            & \ \                    \xi \in \Xi^{(0)}, \\
|\xi_i|^{-1}, & \  \ \xi_i \geq \ell_0+1, \ \xi \in \Xi^{(i)}, \quad i=1,2,3.
\end{array}\right.
$$

A function $\widetilde{N}_k$ from the space $\mathcal{H}$ is called a weak solution to the problem (\ref{junc_probl_general}) if
\begin{equation}\label{integr}
    \int_{\Xi} \nabla \widetilde{N}_k \cdot \nabla v \, d\xi
 =  \int_{\Xi} \widetilde{F}_k \, v \, d\xi \quad \forall \, v\in\mathcal{H}.
\end{equation}

\begin{proposition}\label{tverd1}
Assume that  $\rho^{-1} \widetilde{F}_k\in L^2(\Xi).$  Then there exist a weak solution to the problem (\ref{junc_probl_general}) if and only if
\begin{equation}\label{solvability}
    \int_{\Xi} \widetilde{F}_k \, d\xi =  0.
\end{equation}
This solution is defined up to an additive constant.
The additive constant  can be chosen to guarantee  the existence and uniqueness of a weak solution to the problem (\ref{junc_probl_general}) with
the following differentiable asymptotics:
\begin{equation}\label{inner_asympt_general}
\widehat{N}_k(\xi)=\left\{
\begin{array}{rl}
    \mathcal{O}(\exp( - \gamma_1 \xi_1)) & \mbox{as} \ \ \xi_1\to+\infty, \ \ \xi  \in \Xi^{(1)},
\\[2mm]
    \boldsymbol{\delta_k^{(2)}}
 +  \mathcal{ O}(\exp( - \gamma_2 \xi_2)) & \mbox{as} \ \ \xi_2\to+\infty, \ \ \xi  \in \Xi^{(2)},
\\[2mm]
   \boldsymbol{ \delta_k^{(3)}}
 +  \mathcal{ O}(\exp( - \gamma_3 \xi_3)) & \mbox{as} \ \ \xi_3\to+\infty, \ \ \xi  \in \Xi^{(3)},
\end{array}
\right.
\end{equation}
where $\gamma_i, \ i=1,2,3 $ are positive constants.
\end{proposition}

The constants $\boldsymbol{\delta_k^{(2)}}$ and $\boldsymbol{\delta_k^{(3)}}$ in (\ref{inner_asympt_general})
are defined with the formulas
\begin{equation}\label{const_d_0}
\boldsymbol{\delta_k^{(i)}}
 =      \int_{\Xi} \mathfrak{N}_i \, \widetilde{F}_k(\xi) \, d\xi,
\quad i=2,3, \ \  k\in\Bbb{N}_0,
\end{equation}
where $\mathfrak{N}_2$ and $\mathfrak{N}_3$ are special solutions to
the corresponding homogeneous problem
\begin{equation}\label{hom_probl}
  -\Delta_{\xi}\mathfrak{N} = 0 \ \ \text{in} \ \ \Xi, \qquad
  \partial_\nu \mathfrak{N} = 0 \ \ \text{on} \ \ \partial \Xi,
\end{equation}
for the problem (\ref{junc_probl_general}).

\begin{proposition}[see \cite{Mel_Klev_M2AS-2018}]\label{tverd2}
The  problem (\ref{hom_probl}) has two linearly independent solutions $\mathfrak{N}_2$ and $\mathfrak{N}_3$ that do not belong to the space
$ \mathcal{ H}$ and they have the following differentiable asymptotics:
\begin{equation}\label{inner_asympt_hom_solution_1}
\mathfrak{N}_2(\xi) = \left\{
\begin{array}{rl}
     - \dfrac{\xi_1}{\pi h_1^2(0)}
 +  \mathcal{ O}(\exp( - \gamma_1 \xi_1)) & \mbox{as} \ \ \xi_1\to+\infty,
 \\[3mm]
    C_2^{(2)} + \dfrac{\xi_2}{\pi h_2^2(0)}
 +  \mathcal{ O}(\exp( - \gamma_2 \xi_2)) & \mbox{as} \ \ \xi_2\to+\infty,
 \\[3mm]
    C_2^{(3)} + \mathcal{ O}(\exp( - \gamma_3 \xi_3)) & \mbox{as} \ \ \xi_3\to+\infty,
\end{array}
\right.
\end{equation}
\begin{equation}\label{inner_asympt_hom_solution_2}
\mathfrak{N}_3(\xi) = \left\{
\begin{array}{rl}
     - \dfrac{\xi_1}{\pi h_1^2(0)}
 +     \mathcal{ O}(\exp( - \gamma_1 \xi_1)) & \mbox{as} \ \ \xi_1\to+\infty,
 \\[3mm]
   C_3^{(2)} +  \mathcal{ O}(\exp( - \gamma_2 \xi_2)) & \mbox{as} \ \ \xi_2\to+\infty,
 \\[3mm]
    C_3^{(3)} + \dfrac{\xi_3}{\pi h_3^2(0)}
 +   \mathcal{ O}(\exp( - \gamma_3 \xi_3)) & \mbox{as} \ \ \xi_3\to+\infty,
\end{array}
\right.
\end{equation}

Any other solution to the homogeneous problem, which has polynomial growth at infinity, can be presented as a linear combination
$ \alpha_1 + \alpha_2 \mathfrak{N}_2 + \alpha_3 \mathfrak{N}_3.$
\end{proposition}

\begin{remark}\label{remark_constant}
To obtain formulas (\ref{const_d_0}),  it is necessary to substitute
the functions $\widehat{N}_k, \mathfrak{N}_2$ and $\widehat{N}_k, \mathfrak{N}_3$ in  the second Green-Ostrogradsky formula
$$
\int_{\Xi_R} \big(\widehat{N} \, \Delta_\xi \mathfrak{N} - \mathfrak{N} \, \Delta_\xi \widehat{N}\big)\, d\xi = \int_{\partial \Xi_R}
\big(\widehat{N} \,\partial_{\nu_\xi} \mathfrak{N} - \mathfrak{N}\, \partial_{\nu_\xi}\widehat{N}\big)\, d\sigma_\xi
$$
respectively, and then pass to the limit as $R \to +\infty.$ Here $\Xi_R = \Xi \cap \{\xi : \ |\xi_i| < R, \ i=1, 2, 3\}.$
\end{remark}

\subsection{The limit spectral problem and problems for $\{\widetilde{w}_k\}$}\label{sub_3*4}
The problem (\ref{junc_probl_n})
if $k=0$ has the form
\begin{equation}\label{junc_probl_general-k=0}
 \left\{\begin{array}{rcll}
  -\Delta_{\xi}{{N}_0}(\xi) & = &
   0,                  &
   \quad \xi\in\Xi,
\\[2mm]
   \partial_{\boldsymbol{\nu}_\xi}{{N}_0}(\xi) & = &
   0,                                           &
   \quad \xi\in \partial \Xi,
\\[2mm]
   {N}_0(\xi)                                                   & \sim &
   w^{(i)}_{0}(0),                                                          &
   \quad \xi_i \to +\infty, \ \ \xi  \in \Xi^{(i)}, \quad i=1,2,3.
 \end{array}\right.
\end{equation}
It is ease to verify that $\boldsymbol{\delta_0^{(2)}}=\boldsymbol{\delta_0^{(3)}}=0$ and $\widehat{N}_0 \equiv 0.$
Thus, this problem has a solution in $\mathcal{H}$ if and only if
\begin{equation}\label{trans0}
w_0^{(1)} (0) = w_0^{(2)} (0) = w_0^{(3)} (0);
\end{equation}
in this case  \ $N_0 \equiv  \widetilde{N}_0 \equiv w_0^{(1)} (0).$

The solvability condition  (\ref{solvability}) for the problem (\ref{junc_probl_general}) with $k=1$ reads as follows:
\begin{equation}\label{transmisiont1}
  h_1^2 (0) \frac{d w_{0}^{(1)}}{dx_1} (0) +  h_2^2 (0) \frac{d w_{0}^{(2)}}{dx_2} (0) +  h_3^2 (0) \frac{d w_{0}^{(3)}}{dx_3} (0) =  0.
\end{equation}

Thus, for the functions $\{w_0^{(i)}\}_{i=1}^3$ that are the first  terms of the regular asymptotic expansions~(\ref{regul})
and for the first term $\mu_0$ of  the  expansion~\eqref{exp-EVl}, we obtain the spectral problem
\begin{equation}\label{main}
 \left\{\begin{array}{rclr}
  -  \dfrac{d}{d{x_i}}\left(h_i^2(x_i)\dfrac{d w_0^{(i)}}{d{x_i}}(x_i)\right) & = &
    \mu_0\, h_i^2(x_i)\, w_0^{(i)}(x_i), &                                x_i\in I_i, \quad i=1,2,3,
 \\[4mm]
    w_0^{(i)}(\ell_i) & = &
    0, &           i=1,2,3,
 \\[3mm]
    w_0^{(1)} (0) \ \, = \ \, w_0^{(2)} (0) & = &
    w_0^{(3)} (0), &
  \\[3mm]
    \sum\limits_{i=1}^3  h_i^2 (0) \dfrac{dw_{0}^{(i)}}{dx_i} (0) & = &
    0 &
 \end{array}\right.
\end{equation}
on the graph $\mathcal{I}= \overline{I_1}\cup \overline{I_2}\cup \overline{I_3}.$ The problem \eqref{main} is  called the {\it limit spectral problem} for (\ref{1.1}).

For functions
$$
\widetilde{\phi}(x)=\left\{
                                                  \begin{array}{ll}
                                                    \phi^{(1)}(x_1), &  \ \ x_1 \in \overline{I_1},\\
                                                    \phi^{(2)}(x_2), &  \ \ x_2 \in \overline{ I_2}, \\
                                                    \phi^{(3)}(x_3), &  \ \ x_3 \in \overline{I_3},
                                                  \end{array}
                                                \right.
$$
defined on the graph $\mathcal{I},$
we introduce the Sobolev space
\begin{equation}\label{space_H_0}
\mathcal{ H}_0 := \left\{ \widetilde{\phi}: \
\phi^{(i)} \in H^1(I_i), \ \ \phi^{(i)}(\ell_i) = 0, \ \ i=1, 2, 3, \ \ \hbox{and} \ \ \phi^{(1)}(0) = \phi^{(3)}(0) = \phi^{(3)}(0)
\right\}
\end{equation}
with the scalar product
$$
\langle \widetilde{\phi}, \widetilde{\psi}\rangle_0 := \sum_{i=1}^3  \int_{I_i}  h_i^2(x_i)\, \frac{d\phi^{(i)}}{d{x_i}} \, \frac{d\psi^{(i)}}{d{x_i}}\, dx_i,
\qquad \widetilde{\phi},\, \widetilde{\psi} \in \mathcal{ H}_0.
$$

\begin{definition}\label{def3_1}
A function $\widetilde{w}\in \mathcal{ H}_0 \setminus \{0\}$  is called  an eigenfunction  corresponding to the eigenvalue
$\mu$ of the  problem \eqref{main} if
\begin{equation}\label{limit_identity}
\langle \widetilde{w}, \widetilde{\phi} \rangle_0 = \mu \,
\left( \widetilde{w}, \widetilde{\phi} \right)_{\mathcal{ V}_0}
\qquad \forall \, \widetilde{\phi}  \in \mathcal{ H}_0,
  \end{equation}
where $\mathcal{ V}_{0}$ is the space $L^{2}(\mathcal{I})$
with the scalar product
$$
\left( \widetilde{u}, \widetilde{v} \right)_{\mathcal{ V}_0}:=
\sum_{i=1}^{3} \int_{I_i} h^2_i(x_i)\,  u^{(i)}(x_i) \,  v^{(i)}(x_i) \, dx_i .
$$
\end{definition}

Let us define an operator $ A_0: \mathcal{H}_{0}\mapsto
\mathcal{H}_{0}$  by the  equality
\begin{equation}\label{limit-operator1}
\langle A_{0} \widetilde{u} , \widetilde{v}\rangle_0 = \left( \widetilde{u}, \widetilde{v} \right)_{\mathcal{ V}_0}
 \quad \forall \ \widetilde{u}, \widetilde{v} \in \mathcal{ H}_{0}.
\end{equation}
Obviously,  the operator $A_0$ is self-adjoint, positive, and compact.
Then the problem (\ref{main}) is equivalent to the spectral problem
\ $A_{0} \widetilde{w} = \mu^{-1} \,\widetilde{w}$ \  in $\mathcal{ H}_{0}.$
In addition, it is easy to prove that each eigenvalue of the problem (\ref{main}) is simple.
Therefore, the eigenvalues of the problem (\ref{main})  form the sequence
\begin{equation}\label{EV_lim_prob}
0  <  \Lambda_1   < \Lambda_2 < \ldots < \Lambda_n  < \dots \to + \infty \quad {\rm as}
\quad   n \to +\infty.
\end{equation}
The corresponding eigenfunctions $ \{\widetilde{W}_n \}_{n \in \Bbb N} \subset  \mathcal{ H}_{0}$ can be uniquely
ortho\-nor\-ma\-liz\-ed as follows:
\begin{equation}\label{EF_lim_prob}
\left\langle \widetilde{W}_n, \widetilde{W}_m \right\rangle_{0}=  \delta_{n,m} , \quad n,\, m  \in \Bbb N;
\end{equation}
and
\begin{equation}\label{EF_addion_lim_prob}
\frac{d \, W^{(1)}_n(x_1)}{dx_1}\Big|_{x_1= \ell_1} > 0, \quad n\in \Bbb N.
\end{equation}

\medskip

Now let us write down the solvability condition  (\ref{solvability}) for the problem (\ref{junc_probl_general}) with any fixed $k \in \Bbb N, \ k \ge 2$.
First, note that thanks to \eqref{omega_probl_2}, \eqref{omega_probl_3} and \eqref{omega_probl_k}, we have
\begin{equation}\label{formula1}
  \Big(\Psi_{k}^{(i)}(\xi_i)\Big)^{\prime\prime} = - \sum_{j=0}^{k-2} \frac{\xi_i^j}{j!} \frac{d^j}{dx_i^j} \Bigg(\sum_{m=0}^{k-2-j}
\mu_m\, w^{(i)}_{k-2-j-m}(x_i)\Bigg)\Big|_{x_i =0+} = - \sum_{m=0}^{k-2} \mu_m  \, G^{(i)}_{k-2-m}(\xi_i),
\end{equation}
where
\begin{equation}\label{formula2}
G^{(i)}_k(\xi_i)  := w_{k}^{(i)}(0) + \Psi_k^{(i)}(\xi_i), \quad \xi\in\Xi^{(i)},
\quad i=1,2,3, \quad k\in \Bbb N_0.
\end{equation}
Taking into account  \eqref{formula1}, from (\ref{solvability})  we deduce  that
\begin{gather*}
0=   \sum\limits_{i=1}^3
    \Bigg[  \int_{\Upsilon_i(0)} \Big(\Psi_{k}^{(i)}(\xi_i)\Big)^{\prime}\Big|_{\xi_i = \ell_0} d\overline{\xi}_i +
\int_{\Xi^{(i)}} \sum_{m=0}^{k-2} \mu_m \, \Big( N_{k-2-m}(\xi) - G^{(i)}_{k-2-m}(\xi_i)\Big)  \,  d\xi \Bigg]
\\
+ \int_{\Xi^{(0)}} \varrho_0(\xi) \,  \sum_{m=0}^{k-2} \mu_m \, N_{k-2-m}(\xi) \, d\xi,
\end{gather*}
whence, using \eqref{formula1} again, we get
\begin{equation}\label{transmisiont1+}
    \sum\limits_{i=1}^3  h_i^2 (0) \frac{d w_{k-1}^{(i)}}{dx_i} (0)
 =  \boldsymbol{d_{k-1}^*},
\end{equation}
where
\begin{multline}\label{const_d_*}
\boldsymbol{d_k^*}
 =  \sum\limits_{i=1}^3 h_i^2 (0) \, \sum\limits_{j=1}^k \frac{\ell_0^j}{j!} \sum_{m=0}^{k-j } \mu_m \,
    \dfrac{d^{j-1} w_{k-j-m}^{(i)}}{d x_i^{j-1}} (0)
-  \frac{1}{\pi}  \sum_{m=0}^{k-1} \mu_m \, \sum\limits_{i=1}^3 \int_{\Xi^{(i)}}\Big( N_{k-1-m}(\xi) - G^{(i)}_{k-1-m}(\xi_i)\Big)  \,  d\xi
\\
     -  \frac{1}{\pi} \int_{\Xi^{(0)}} \varrho_0(\xi) \sum_{m=0}^{k-1} \mu_m \, N_{k-1-m}(\xi)  \, d\xi,
\quad k\in\Bbb N.
\end{multline}
In particular, since $N_0 \equiv G^{(1)}_0 \equiv G^{(2)}_0 \equiv G^{(3)}_0  \equiv w^{(1)}(0),$
\begin{equation}\label{const_d_*1}
\boldsymbol{d_1^*} = \mu_0 \, w^{(1)}_0(0)  \Bigg(\sum\limits_{i=1}^3 \ell_0 \, h_i^2 (0)
-  \frac{1}{\pi} \int_{\Xi^{(0)}} \varrho_0(\xi) \, d\xi \Bigg).
\end{equation}
Recall that $\boldsymbol{d_0^*}=0.$

Hence, if the functions $\{ w_{k-1}^{(i)}\}_{i=1}^3$ satisfy (\ref{transmisiont1+}), then there exist a weak solution $\widetilde{N}_k$ to  the problem~(\ref{junc_probl_general}). According to Proposition  \ref{tverd1}, it can be chosen in a unique way to guarantee the
asymptotics~(\ref{inner_asympt_general}).
 \ But, till now, we do not take into account the conditions (\ref{junc_probl_general+cond}).
To satisfy the first condition, we represent a weak solution to the problem (\ref{junc_probl_general}) in the following form:
$$
\widetilde{N}_k = w_k^{(1)}(0) + \widehat{N}_k.
$$
Taking into account the asymptotics (\ref{inner_asympt_general}), we have to put
\begin{equation}\label{trans1}
w_k^{(1)}(0) =  w_k^{(2)}(0) - \boldsymbol{\delta_k^{(2)}} = w_k^{(3)}(0) - \boldsymbol{\delta_k^{(3)}},  \quad k \in \Bbb N,
\end{equation}
where $\boldsymbol{\delta_k^{(2)}}$ and $\boldsymbol{\delta_k^{(3)}}$ are defined with the formulas \eqref{const_d_0}.
As a result, we get the solution to the problem~(\ref{junc_probl_n}) with the following asymptotics:
\begin{equation}\label{inner_asympt}
{N}_{k}(\xi)
 =  \omega_{k}^{(i)} (0) + \Psi_k^{(i)} (\xi)
 +  \mathcal{ O}(\exp(-\gamma_i\xi_i))
\quad \mbox{as} \ \ \xi_i\to+\infty,  \ \  \xi  \in \Xi^{(i)},  \quad i=1,2,3.
\end{equation}

\begin{remark}\label{rem_exp-decrease}
Thus, the functions $\{{N}_{k} - G^{(i)}_k\}_{k\in \Bbb N}$ are exponentially decreasing  as $\xi_i \to +\infty, \  \xi  \in \Xi^{(i)},$ $i=1,2,3.$ This means that the integrals over $\Xi^{(i)}$ in \eqref{const_d_*} exist.
\end{remark}

Relations (\ref{trans1}) and (\ref{transmisiont1+}) are the  first and second transmission conditions for the functions $\{w_k^{(i)}\}$ at $x=0.$
Thus, with regard to \eqref{omega_probl_k},  the coefficients  $\{w_k^{(1)}, \, w_k^{(2)}, \, w_k^{(3)} \}$ and $\mu_k$ can be determined from the problem
\begin{equation}\label{omega_probl*}
 \left\{\begin{array}{rclr}
  -  \dfrac{d}{d{x_i}}\left(h_i^2(x_i)\dfrac{d w_k^{(i)}}{d{x_i}}(x_i)\right) & = &
  \mu_0 \, h_i^2(x_i) \,  w_{k}^{(i)} (x_i) +
\displaystyle{\sum_{m=1}^{k}} \mu_m \, h_i^2(x_i) \,  w_{k-m}^{(i)} (x_i), &   x_i\in I_i, \  i=1,2,3,
 \\[3mm]
    w_k^{(i)}(\ell_i) & = &
    0, \quad i=1,2,3, &
 \\[3mm]
    w_k^{(1)} (0) \ \, = \ \, w_k^{(2)} (0) - \boldsymbol{\delta_k^{(2)}} & = &
                                   w_k^{(3)} (0) - \boldsymbol{\delta_k^{(3)}}, &
 \\[3mm]
    \sum\limits_{i=1}^3 h_i^2 (0) \,\dfrac{d w_{k}^{(i)}}{dx_i} (0) & = &
   \boldsymbol{d_k^*}.&
 \end{array}\right.
\end{equation}
To solve the problem (\ref{omega_probl*}),  we make the following substitutions:
\begin{equation}\label{sunstitutions}
\phi_k^{(1)}(x_1)=\omega_k^{(1)}(x_1),\qquad \phi_k^{(i)}(x_i) = \omega_k^{(i)}(x_i) - \boldsymbol{\delta_k^{(i)}} \, \frac{\ell_i - x_i}{\ell_i},
\quad i= 2, 3.
\end{equation}
As a result for the  functions $\{\phi_k^{(i)}\}_{i=1}^3,$ we get the problem
\begin{equation}\label{tilde-omega_probl*}
 \left\{\begin{array}{rclr}
  - \dfrac{d}{d{x_i}}\left(h_i^2(x_i)\dfrac{d\phi_k^{(i)}}{d{x_i}}(x_i)\right) - \mu_0 h_i^2(x_i)   \phi_{k}^{(i)} (x_i)  & = &
    \Phi_k^{(i)}(x_i; \mu_1,\ldots,\mu_{k-1}) + \mu_k  h_i^2(x_i)  w_{0}^{(i)} (x_i) , &     i\in I_i,
 \\[3mm]
    \phi_k^{(i)}(\ell_i) & = &
    0,  \quad \  i=1,2,3,&
 \\[3mm]
    \phi_k^{(1)} (0) \ \, = \ \, \phi_k^{(2)} (0)  & = &
                                   \phi_k^{(3)} (0), &
 \\[3mm]
    \sum\limits_{i=1}^3 h_i^2 (0) \,\dfrac{d\phi_{k}^{(i)}}{dx_i} (0) & = &
{\bf D_k}
   , &
 \end{array}\right.
\end{equation}
where
\begin{gather}
{\bf D_k} = \boldsymbol{d_k^*} + \dfrac{\boldsymbol{\delta_k^{(2)}} h^2_2(0)}{\ell_2}
 + \dfrac{\boldsymbol{\delta_k^{(3)}} h^2_3(0)}{\ell_3} \label{D_k}
\\
{\Phi}_k^{(1)}(x_1; \mu_1,\ldots,\mu_{k-1}) =  \displaystyle{\sum_{m=1}^{k-1}} \mu_m \, h_1^2(x_1) \,  w_{k-m}^{(1)} (x_1) \notag
 \\
  {\Phi}_k^{(i)}(x_i; \mu_1,\ldots,\mu_{k-1}) = \displaystyle{\sum_{m=1}^{k-1}} \mu_m \, h_i^2(x_i) \,  w_{k-m}^{(i)} (x_i) -
\dfrac{2 \boldsymbol{\delta_k^{(i)}}}{\ell_i} h_i(x_i) \, h_i'(x_i) +  \mu_0 \, h_i^2(x_i) \, \boldsymbol{\delta_k^{(i)}} \, \frac{\ell_i - x_i}{\ell_i}, \ \ i=2, 3.
\notag
 \end{gather}

\begin{definition}
A function $\widetilde{\phi}_k\in \mathcal{ H}_0$ is called a weak solution to the problem \eqref{tilde-omega_probl*} if
\begin{multline}\label{def_3_2}
  \langle \widetilde{\phi}_k - \mu_0 \,A_0 \widetilde{\phi}_k \, , \,\widetilde{\psi}\rangle_0 = \sum_{i=1}^3 \int_0^{\ell_i} \Phi_k^{(i)}(x_i; \mu_1,\ldots,\mu_{k-1}) \, \psi_i(x_i)\, dx_i - {\bf D_k} \,   \psi_1(0)
\\
+  \mu_k  \sum_{i=1}^3 \int_0^{\ell_i} h_i^2(x_i) \, w_{0}^{(i)} (x_i) \, \psi_i(x_i)\, dx_i  \quad \forall \, \widetilde{\psi} \in \mathcal{ H}_0.
\end{multline}
\end{definition}

  Since $\frac{1}{\mu_0}$ is an eigenvalue of the operator $A_0$ defined by \eqref{limit-operator1} and
the linear functional over the space $\mathcal{ H}_0$ in the right-hand side of the identity \eqref{def_3_2} is bounded,
we can apply the Fredholm theory   to find the unique weak solution $\widetilde{\phi}_k$ and the number $\mu_k.$
Taking into account that the spectrum  of $A_0$ is simple, we can state that the problem (\ref{tilde-omega_probl*})
has a weak solution if and only if
\begin{equation}\label{Fredholm}
  \sum_{i=1}^3 \int_0^{\ell_i} \Phi_k^{(i)}(x_i; \mu_1,\ldots,\mu_{k-1}) \, w_{0}^{(i)} (x_i) \, dx_i - {\bf D_k}\,   w_{0}^{(1)}(0)
+  \mu_k  \sum_{i=1}^3 \int_0^{\ell_i} h_i^2(x_i) \, \big(w_{0}^{(i)} (x_i)\big)^2 \, dx_i =0.
\end{equation}
From \eqref{Fredholm} we define the number $\mu_k.$ Thus, there exists a unique weak solution to the problem (\ref{tilde-omega_probl*})
such that
\begin{equation}\label{ortho_solution}
  \langle \widetilde{\phi}_k, \widetilde{w}_0 \rangle_0 =0.
\end{equation}
Then, with the help of  \eqref{sunstitutions} we find  $\{w_k^{(1)}, \, w_k^{(2)}, \, w_k^{(3)} \}.$
We will do this in detail in the next subsection.

\subsection{Scheme of the construction of  asymptotic expansions}\label{justification}

{\it Step 1.}
We know that the first term $\mu_0$ of  the  expansion~\eqref{exp-EVl} must be an eigenvalue of the  limit spectral problem (\ref{main})
and  $\{w_0^{(i)}\}_{i=1}^3,$ which are the first  terms of the regular ansatzes ~(\ref{regul}), form the corresponding eigenfunction $\widetilde{w}_0 \in \mathcal{H}_0.$  So let
$$
\mu_0 = \Lambda_n, \qquad  \widetilde{w}_0 = \widetilde{W}_n
$$
for some index $n\in \Bbb N,$ where $\Lambda_n$ is the $n$-th eigenvalue from the sequence  \eqref{EV_lim_prob} and
$\widetilde{W}_n$ is the corresponding eigenfunction satisfying  conditions \eqref{EF_lim_prob} and \eqref{EF_addion_lim_prob}.

Next we  determine the first term $N_0$ of the inner asymptotic expansion (\ref{junc}); it is a solution to the problem (\ref{junc_probl_general-k=0}) and $N_0=W_n^{(1)}(0).$

Then we can uniquely define the solution ${u}_{2,n}^{(i)}$ to the problem (\ref{regul_probl_2}) for each index $i=1,2,3,$
since the solvability condition for this problem is satisfied (see (\ref{omega_probl_2})).

Now with the help of formulas (\ref{view_solution}), we determine the first terms $\Pi_{2,n}^{(i)}, \ i=1,2,3$ of the
boundary-layer expansions (\ref{prim+}), as solutions to the problems  (\ref{prim+probl}) that can be rewritten as follows:
\begin{equation}\label{new_prim+probl_2}
 \left\{\begin{array}{rcll}
  -\Delta_{\xi_i^*, \overline{\xi}_i} \Pi_{2,n}^{(i)}(\xi_i^*,\overline{\xi}_i) & =
   & 0,
   & \xi_i^*\in(0,+\infty), \quad \overline{\xi}_i\in\Upsilon_i(\ell_i),
   \\[2mm]
  -\partial_{\nu_{\overline{\xi}_i}} \Pi_{2,n}^{(i)}(\xi_i^*,\overline{\xi}_i) & =
   & 0,
   & \xi_i^*\in(0,+\infty), \quad \overline{\xi}_i\in\partial\Upsilon_i(\ell_i),
   \\[2mm]
  \Pi_{2,n}^{(i)}(0,\overline{\xi}_i) & =
   & - u_{2,n}^{(i)} (\ell_i,\overline{\xi}_i) ,
   & \overline{\xi}_i\in\Upsilon_i(\ell_i),
   \\[2mm]
  \Pi_{2,n}^{(i)}(\xi_i^*,\overline{\xi}_i) & \to
   & 0,
   & \xi_i^*\to+\infty, \quad \overline{\xi}_i\in\Upsilon_i(\ell_i),
 \end{array}\right.
\end{equation}

{\it Step 2.}
The second terms $\{w_{1,n}^{(i)}\}_{i=1}^3$ of the regular asymptotics (\ref{regul}) are founded  from the problem (\ref{omega_probl*}) that looks now as follows:
\begin{equation}\label{omega_probl_1}
 \left\{\begin{array}{rclr}
  -  \dfrac{d}{d{x_i}}\left(h_i^2(x_i)\dfrac{d w_{1,n}^{(i)}}{d{x_i}}(x_i)\right) & = &
  \Lambda_n \, h_i^2(x_i) \,  w_{1,n}^{(i)} (x_i) +
 \mu_{1,n} \, h_i^2(x_i) \,  W_{n}^{(i)} (x_i), &   x_i\in I_i, \  i=1,2,3,
 \\[3mm]
    w_{1,n}^{(i)}(\ell_i) & = &
    0, \quad i=1,2,3, &
 \\[3mm]
    w_{1,n}^{(1)} (0) \ \, = \ \, w_{1,n}^{(2)} (0) - \boldsymbol{\delta_{1,n}^{(2)}} & = &
                                   w_{1,n}^{(3)} (0) - \boldsymbol{\delta_{1,n}^{(3)}}, &
 \\[3mm]
    \sum\limits_{i=1}^3 h_i^2 (0) \,\dfrac{d w_{1,n}^{(i)}}{dx_i} (0) & = &
   \boldsymbol{d_{1,n}^*}.&
 \end{array}\right.
\end{equation}
Here, the constants $\boldsymbol{\delta_{1,n}^{(2)}}$ and $\boldsymbol{\delta_{1,n}^{(3)}}$ are uniquely determined by the
formula  \eqref{const_d_0} with $k=1$ and
\begin{equation}\label{delta_1}
\boldsymbol{\delta_{1,n}^{(i)}}
 =  \sum\limits_{j=1}^3 \dfrac{d W_n^{(j)}}{dx_j}(0) \int_{\Xi^{(j)}} \mathfrak{N}_i(\xi)
        \Big( \xi_j \chi_j^{\prime\prime}(\xi_j) + 2 \chi_j^{\prime}(\xi_j) \Big) \, d\xi,
\quad i=2,3;
\end{equation}
\begin{equation}\label{d_1^*}
\boldsymbol{d_{1,n}^*}
 =  \Lambda_n \, W_{n}^{(1)}(0) \left(  \sum\limits_{i=1}^3  \ell_0 h_i^2 (0)
     - \frac{1}{\pi} \int_{\Xi^{(0)}} \varrho_0(\xi)  \, d\xi \right)
\end{equation}
(see \eqref{const_d_*1}).
We can reduce the problem \eqref{omega_probl_1} to the corresponding problem \eqref{tilde-omega_probl*} with $k=1$ and apply the
Fredholm theory. Taking into account \eqref{EF_lim_prob}, \eqref{D_k} and \eqref{d_1^*}, we get from \eqref{Fredholm}  that
 \begin{equation}\label{mu_1}
 \mu_{1,n}   =  \left(\Lambda_n \, W_{n}^{(1)}(0)\right)^2  \left(  \sum\limits_{i=1}^3  \ell_0 h_i^2 (0)
     -  \frac{1}{\pi} \int\limits_{\Xi^{(0)}} \varrho_0(\xi)  \, d\xi \right)
 + \left(\dfrac{\boldsymbol{\delta_{1,n}^{(2)}} h^2_2(0)}{\ell_2}
 + \dfrac{ \boldsymbol{\delta_{1,n}^{(3)}} h^2_3(0)}{\ell_3}\right) \,   \Lambda_n W_{n}^{(1)}(0) .
\end{equation}

Thus, there exists a unique weak solution to the problem (\ref{tilde-omega_probl*}) with $k=1$
such that
\begin{equation}\label{ortho_solution}
  \langle \widetilde{\phi}_{1,n}, \widetilde{W}_n \rangle_0 =0.
\end{equation}
Then, with the help of the substitutions \eqref{sunstitutions} we find the coefficients  $\{w_1^{(1)}, \ w_1^{(2)}, \ w_1^{(3)} \}.$

Having  $\{w_1^{(i)}\}_{i=1}^3,$ we can uniquely find the seconds terms of the regular asymptotics $\{{u}_{3,n}^{(i)}\}_{i=1}^3$ (series (\ref{regul})) from
the problem \eqref{regul_probl_3} and terms $\{\Pi_{3,n}^{(i)}\}_{i=1}^3$  of the boundary asymptotics (\ref{prim+}) from the problems
\begin{equation}\label{new_prim+probl_3}
 \left\{\begin{array}{rcll}
  -\Delta_{\xi_i^*, \overline{\xi}_i}
   \Pi_{3,n}^{(i)}(\xi_i^*,\overline{\xi}_i) & =
   & 0,
   & \xi_i^*\in(0,+\infty), \quad \overline{\xi}_i\in\Upsilon_i(\ell_i),
   \\[2mm]
  -\partial_{\nu_{\overline{\xi}_i}}
   \Pi_{3,n}^{(i)}(\xi_i^*,\overline{\xi}_i) & =
   & 0,
   & \xi_i^*\in(0,+\infty), \quad \overline{\xi}_i\in\partial\Upsilon_i(\ell_i),
   \\[2mm]
  \Pi_{3,n}^{(i)}(0,\overline{\xi}_i) & =
   & - u_{3,n}^{(i)} (\ell_i,\overline{\xi}_i),
   & \overline{\xi}_i\in\Upsilon_i(\ell_i),
   \\[2mm]
  \Pi_{3,n}^{(i)}(\xi_i^*,\overline{\xi}_i) & \to
   & 0,
   & \xi_i^*\to+\infty, \quad \overline{\xi}_i\in\Upsilon_i(\ell_i),
 \end{array}\right.
\end{equation}
$i=1, 2, 3,$respectively.

The second term  ${N}_{1,n}$ of the inner asymptotic expansion (\ref{junc}) is a  unique solution to the problem~(\ref{junc_probl_n})
that looks now as follows:
\begin{equation}\label{new_junc_probl_1}
 \left\{\begin{array}{rcll}
  -\Delta_{\xi}{N_{1,n}}(\xi) & = &
   0,                         &
   \quad \xi\in\Xi,
 \\[2mm]
   \partial_{\nu_\xi}{N_{1,n}}(\xi) & = &
   0,                               &
   \quad \xi\in \partial \Xi,
 \\[1mm]
   N_{1,n}(\xi)                                                               & \sim &
   w^{(i)}_{1,n}(0) + \xi_i \dfrac{dW_{n}^{(i)}}{dx}(0),                  &
   \quad \xi_i \to +\infty, \ \ {\xi}_i \in \Xi^{(i)}, \quad i=1,2,3.
 \end{array}\right.
\end{equation}
Its solvability condition   is satisfied due to the second  Kirchhoff  condition in the problem \eqref{main} for the eigenfunction $\widetilde{W}_n.$

Thus, we have uniquely determined the first two terms in expansions (\ref{regul}), (\ref{prim+}), (\ref{junc}) and \eqref{exp-EVl}.

\medskip

{\it Inductive step.}
Assume that we have determined the coefficients $w_{1,n}^{(i)},\ldots,w_{k-1,n}^{(i)},$
$u_{2,n}^{(i)},\ldots,u_{k+1,n}^{(i)}$ of the series (\ref{regul}),
coefficients
$\Pi^{(i)}_{2,n},\ldots,\Pi^{(i)}_{k+1,n}$ of the series (\ref{prim+}),
coefficients
${N}_{1,n},\ldots,{N}_{k-1,n}$ of the series (\ref{junc}), constants $\boldsymbol{\delta_{1,n}^{(i)}},\ldots,\boldsymbol{\delta_{k-1,n}^{(i)}}, \ i=2, 3,$
$\boldsymbol{d_{1,n}^*},\ldots,\boldsymbol{d_{k-1,n}^*},$ and coefficients $\mu_{1,n}\ldots,\mu_{k-1,n}$ of  \eqref{exp-EVl}.

Then we can determine  the constants $\boldsymbol{\delta_{k,n}^{(2)}}, \, \boldsymbol{\delta_{k,n}^{(3)}}$ (see (\ref{const_d_0})) in the first  transmission condition and the constant $\boldsymbol{ d_{k,n}^*}$ (see \eqref{const_d_*}) in the second transmission conditions of the problem (\ref{omega_probl*}).
We reduce the problem (\ref{omega_probl*}) to the problem \eqref{tilde-omega_probl*} and
from \eqref{Fredholm} we define the number
\begin{equation}\label{mu_k}
  \mu_{k,n} = {\bf D_k} \,  \Lambda_n  W_{n}^{(1)}(0)
-  \Lambda_n  \sum_{i=1}^3 \int_0^{\ell_i} \Phi_{k,n}^{(i)}(x_i; \mu_{1,n},\ldots,\mu_{k-1,n}) \, W_{n}^{(i)} (x_i) \, dx_i
\end{equation}
of the series \eqref{exp-EVl}. This means that  there exists a unique weak solution to the problem (\ref{tilde-omega_probl*})
such that
\begin{equation}\label{ortho_solution}
  \langle \widetilde{\phi}_{k,n}, \widetilde{W}_n \rangle_0 =0.
\end{equation}
Then, with the help of  \eqref{sunstitutions} we find the
solution $\{w_{k,n}^{(i)}\}_{i=1}^3$ to the problem (\ref{omega_probl*}).

The coefficients $u_{k+2,n}^{(i)}, \ i=1,2,3,$ are determined as solutions to the problems
\eqref{regul_probl_k}. The  coefficients $\Pi_{k+2,n}^{(i)}, \ i=1,2,3,$ of the boundary asymptotic expansions (\ref{prim+}) are solutions to the problems (\ref{prim+probl}) with the boundary condition
$$
\Pi_{k+2,n}^{(i)}(0,\overline{\xi}_i)  =  - u_{k+2,n}^{(i)} (\ell_i,\overline{\xi}_i),
   \quad  \overline{\xi}_i\in\Upsilon_i(\ell_i).
$$

Finally, we find the coefficient ${N}_{k,n}$ of the inner asymptotic expansion (\ref{junc}),
which is a unique solution to the problem (\ref{junc_probl_n}). The corresponding  solvability condition  is  the second transmission  condition
 for the functions $\{w_{k-1,n}^{(i)}\}_{i=1}^3.$

Thus, we  can successively determine all coefficients of series  (\ref{regul}), (\ref{prim+}), (\ref{junc}) and  \eqref{exp-EVl}.

\subsection{Justification}\label{Sec_justification}

Using  (\ref{regul}), (\ref{prim+}), and (\ref{junc}),  we construct the series
\begin{equation}\label{asymp_expansion}
    \sum\limits_{k=0}^{+\infty} \varepsilon^{k}
    \Big(
    \overline{u  }_{k,n} (x; \, \varepsilon, \, \gamma)
  + \overline{\Pi}_{k,n} (x; \, \varepsilon)
  + \overline{N  }_{k,n} (x; \, \varepsilon, \, \gamma)
    \Big),
\quad x\in\Omega_\varepsilon,
\end{equation}
where
$$
\overline{u}_{k,n} (x; \, \varepsilon, \, \gamma)
 := \sum\limits_{i=1}^3 \chi_{\ell_0}^{(i)} \left(\frac{x_i}{\varepsilon^\gamma}\right)
    \left( u_{k,n}^{(i)} \left( x_i, \frac{\overline{x}_i}{\varepsilon} \right)
 +  w_{k,n}^{(i)} (x_i) \right),
 \quad ( u_{0,n}\equiv u_{1,n} \equiv 0, \ \ w_{0,n}^{(i)} \equiv   W_{n}^{(i)} ),
$$
$$
\overline{\Pi}_{k,n} (x; \, \varepsilon)
 := \sum\limits_{i=1}^3 \chi_\delta^{(i)} (x_i) \,
    \Pi_{k,n}^{(i)} \left( \frac{\ell_i -x_i}{\varepsilon}, \frac{\overline{x}_i}{\varepsilon} \right),
 \quad ( \Pi_{0,n}\equiv \Pi_{1,n} \equiv 0 ),
$$
$$
\overline{N}_{k,n} (x; \, \varepsilon, \, \gamma)
 := \left(1 - \sum\limits_{i=1}^3 \chi_{\ell_0}^{(i)} \left(\frac{x_i}{\varepsilon^\gamma}\right) \right)
    N_{k,n} \left( \frac{x}{\varepsilon} \right),
 \quad \left( N_{0,n} \equiv W_n^{(1)}(0) \right),
$$
$\gamma$ is a fixed number from the interval $(\frac23, 1),$ $\chi_{\ell_0}^{(i)}, \ \chi_\delta^{(i)}$ are smooth cut-off functions defined by
\begin{equation}\label{cut-off-functions}
\chi_{\ell_0}^{(i)} (x_i) =
\left\{\begin{array}{ll}
1, &  \ \ x_i \ge 3 \, \ell_0,
\\
0, &  \ \ x_i \le 2 \, \ell_0,
\end{array}\right.
\quad
\chi_\delta^{(i)}(x_i)=
\left\{\begin{array}{ll}
1, & \ \ x_i \ge 1 -  \delta,
\\
0, &  \ \ x_i \le 1 - 2\delta,
\end{array}\right.
\quad i=1,2,3,
\end{equation}
and $\delta$ is a  fixed sufficiently small positive number.

Denote by
\begin{equation}\label{aaN}
U^{(M)}_{n}(x;\varepsilon)
 :=  \sum\limits_{k=0}^{M} \varepsilon^{k}
    \Big(
    \overline{u  }_k (x; \, \varepsilon, \, \gamma)
  + \overline{\Pi}_k (x; \, \varepsilon)
  + \overline{N  }_k (x; \, \varepsilon, \, \gamma)
    \Big),
\quad x\in\Omega_\varepsilon,
\end{equation}
 the partial sum of $(\ref{asymp_expansion}),$ where $M\in \Bbb N, \, M \ge 3,$  and by
\begin{equation}\label{part_EV}
 \mathcal{L}^{(M)}_{n}(\varepsilon) := \Lambda_n + \sum\limits_{k=1}^{M}\varepsilon^k \mu_{k,n}
\end{equation}
 the partial sum of \eqref{exp-EVl}. Obviously, $ U^{(M)}_{n} \in \mathcal{H}_\varepsilon.$

 It is easy to calculate that for each $i\in \{1, 2, 3\}$ and $k\in \Bbb N_0$ (the index $n$ is omitted)
 $$
  \Delta_x\Bigg(\chi_{\ell_0}^{(i)}\Big(\frac{x_i}{\varepsilon^\gamma}\Big)
    \left(u_{k}^{(i)} \left( x_i, \frac{\overline{x}_i}{\varepsilon} \right)
 +  w_{k}^{(i)} (x_i)\right) +
 \Big(1 - \chi_{\ell_0}^{(i)} \left(\frac{x_i}{\varepsilon^\gamma}\right)\Big)
    N_{k} \left( \frac{x}{\varepsilon} \right)  + \chi_\delta^{(i)} (x_i)
    \Pi_{k}^{(i)} \left( \frac{\ell_i -x_i}{\varepsilon}, \frac{\overline{x}_i}{\varepsilon} \right)\Bigg)
 $$
 $$
= \chi_{\ell_0}^{(i)}\Big(\frac{x_i}{\varepsilon^\gamma}\Big)
 \left( \varepsilon^{-2} \Delta_{\overline{\xi}_i}u_{k}^{(i)}\left( x_i,\overline{\xi}_i\right)\Big|_{\overline{\xi}_i = \frac{\overline{x}_i}{\varepsilon}} + \partial^2_{x_i^2}u_{k}^{(i)} \left( x_i, \frac{\overline{x}_i}{\varepsilon} \right)
 + \frac{d^2w_{k}^{(i)} (x_i)}{dx_i^2} \right)
 $$
 $$
 +\,  \Big(1 - \chi_{\ell_0}^{(i)} \left(\frac{x_i}{\varepsilon^\gamma}\right)\Big) \varepsilon^{-2}
    \Delta_{\xi} N_{k}(\xi)\Big|_{\xi= \frac{x}{\varepsilon}}
 $$
 $$
 - \, 2 \varepsilon^{-1} \left(\chi_\delta^{(i)} (x_i) \right)'
   \partial_{\xi_i^*} \Pi_{k}^{(i)} \left( \xi_i^*, \frac{\overline{x}_i}{\varepsilon} \right)\Big|_{\xi_i^*=\frac{\ell_i -x_i}{\varepsilon}}
   + \left(\chi_\delta^{(i)} (x_i) \right)'' \Pi_{k}^{(i)} \left( \frac{\ell_i -x_i}{\varepsilon}, \frac{\overline{x}_i}{\varepsilon} \right)
 $$
  \begin{equation}\label{calcul1}
  + \varepsilon^{-\gamma} \frac{\partial}{\partial x_i} \left( \Big(\chi_{\ell_0}^{(i)}\Big)'\Big(\frac{x_i}{\varepsilon^\gamma}\Big)
       \,  \Big(  w_{k}^{(i)} (x_i) - N_{k} \Big( \frac{x}{\varepsilon} \Big) \Big) \right)
    + \varepsilon^{-\gamma} \Big(\chi_{\ell_0}^{(i)}\Big)'\Big(\frac{x_i}{\varepsilon^\gamma}\Big)
   \, \frac{\partial}{\partial x_i} \left(w_{k}^{(i)} (x_i) - N_{k} \Big( \frac{x}{\varepsilon} \Big)  \right)
 \end{equation}
and
\begin{equation}\label{calcul2}
\partial_\nu  \overline{u  }_k (x; \, \varepsilon, \, \gamma) =  \chi_{\ell_0}^{(i)}\Big(\frac{x_i}{\varepsilon^\gamma}\Big)
\dfrac{\varepsilon h_i^\prime (x_i)}{\sqrt{1 + \varepsilon^2 |h_i^\prime (x_i)|^2 \,}} \sum_{k=m-1}^{m} \varepsilon^k
 \left( \frac{\partial u_{k}^{(i)}}{\partial x_i}\Big( x_i, \frac{\overline{x}_i}{\varepsilon} \Big)
 +  \frac{d w_{k}^{(i)}}{dx_i} (x_i) \right) \quad \text{on} \ \ \Gamma_\varepsilon^{(i)}.
  \end{equation}

Substituting $ U^{(M)}_{n}$ and $\mathcal{L}^{(M)}_{n}(\varepsilon)$ $(M \ge 3)$ into the problem~(\ref{1.1}) in place of $u^\varepsilon$ and $\lambda(\varepsilon)$
respectively and  taking into account  relations (\ref{main}), \eqref{omega_probl*}, \eqref{regul_probl_2}, \eqref{regul_probl_3}, \eqref{regul_probl_k}, \eqref{prim+probl}, \eqref{junc_probl_n}, \eqref{mu_k}  for coefficients of  the series (\ref{asymp_expansion}) and \eqref{exp-EVl}, we find with the help of \eqref{calcul1} and  \eqref{calcul2}  that for any $\psi \in \mathcal{H}_\varepsilon$
\begin{equation}\label{just1}
  \int_{\Omega_\varepsilon}\left( \nabla U^{(M)}_{n} \cdot \nabla \psi  -   \rho_\varepsilon(x)\, \mathcal{L}^{(M)}_{n}(\varepsilon)\,
  U^{(M)}_{n}\, \psi \right) dx = \mathfrak{F}_\varepsilon (\psi),
\end{equation}
where
 $\mathfrak{F}_\varepsilon$ is a sum of integrals of residuals produced by $U^{(M)}_{n}$ and $\mathcal{L}^{(M)}_{n}(\varepsilon)$ and $\mathfrak{F}_\varepsilon \in \mathcal{H}_\varepsilon^*.$
 Using \eqref{est3},  \eqref{est4},  \eqref{ineq1}, \eqref{inner_asympt}   and Remarks~\ref{asymp-Pi} and \ref{rem_exp-decrease},
we estimate those residuals  (for more detail see Subsection \ref{subsec_just+alpfa}) and obtain that
\begin{equation}\label{just2}
|\mathfrak{F}_\varepsilon (\psi)| \le C_M \, \left( \varepsilon^{M+1} + \varepsilon^{\gamma M +1} \right) \,  \|\psi \|_\varepsilon.
\end{equation}
Remembering  the definition of the operator $A_\varepsilon$ (see \eqref{operator1}) and the Riesz representation theorem,  we
deduce   from \eqref{just1} and \eqref{just2} the inequality
\begin{equation}\label{just3}
\left\| U^{(M)}_{n} - \mathcal{L}^{(M)}_{n}(\varepsilon)\, A_\varepsilon\Big( U^{(M)}_{n}\Big) \right\|_\varepsilon \le  C_M \,
\varepsilon^{\gamma M +1}  .
\end{equation}

Further we will repeatedly use Lemma 12 \cite{Vishik}, which has wide applications for the approximation
of eigenvalues and eigenfunctions of self-adjoint compact operators. Therefore, we recall it here.
\begin{lemma}[\cite{Vishik}]\label{Vishik}
Let $A:H\rightarrow H$  be a linear  self-adjoint positive compact  operator in a Hilbert space $H.$
Suppose that there exist a positive number $\mu\in\Bbb{R}$ and a vector  $u\in H$ such that $\left\|u\right\|_H=1$ and
$$
\left\|Au-\mu u\right\|_H\leq\beta.
$$

Then there exists an eigenvalue $\lambda$ of the operator $A$ such that  $\left|\mu-\lambda\right|\leq\beta.$
Moreover, for any $d_0>\beta$  there exists a vector  $\widetilde{u}\in H,\ \left\|\widetilde{u}\right\|_H=1,$ such that
	$$
		\left\|u-\widetilde{u}\right\|_H\leq 2 d_0^{-1} \beta,
	$$
where $\widetilde{u}$ is a linear combination of eigenvectors corresponding to all eigenvalues of  $A$ from
the segment  $[\mu - d_0, \mu + d_0]$.
\end{lemma}

Since for sufficiently small $\varepsilon$
\begin{gather}\label{additional_est1}
 0 < c_1 \le  \mathcal{L}^{(M)}_{n}(\varepsilon) \le c_2
\\
   0<  c_3 \varepsilon \le \| U^{(M)}_{n} \|_\varepsilon \le c_4 \varepsilon \label{additional_est2}
\end{gather}
(the last estimate thanks to  \eqref{EF_lim_prob} ), it follows from \eqref{just3} that
\begin{equation}\label{just4}
\left\| A_\varepsilon\left( \frac{U^{(M)}_{n}}{\|U^{(M)}_{n}\|_\varepsilon}\right)  - \left(\mathcal{L}^{(M)}_{n}(\varepsilon)\right)^{-1} \, \frac{U^{(M)}_{n}}{\|U^{(M)}_{n}\|_\varepsilon}  \right\|_\varepsilon
 \le \tilde{C}_M  \, \varepsilon^{\gamma M } .
\end{equation}

Now let us take
$$
d_0 := \frac{1}{2} \min \left\{ \frac{1}{\Lambda_{n-1}} - \frac{1}{\Lambda_{n}} \, , \,  \frac{1}{\Lambda_{n}} - \frac{1}{\Lambda_{n+1}} \right\},
$$
where $\{\Lambda_n\}_{n\in \Bbb N}$ are simple eigenvalues of the limit spectral problem \eqref{main}.
Since for each $n \in \Bbb N$ the eigenvalue $\lambda_n(\varepsilon)$ goes to $\Lambda_n$ as $\varepsilon \to 0$ (see Lemma~\ref{Lemma_main_Convergence}) and the corresponding eigenfunction satisfies \eqref{normalized} and \eqref{normalized+},
it follows from Lemma~\ref{Vishik} and \eqref{just4} that
\begin{equation}\label{just5}
  \left| \frac{1}{\lambda_n(\varepsilon)} - \frac{1}{\mathcal{L}^{(M)}_{n}(\varepsilon)} \right| \le \tilde{C}_M  \,
   \varepsilon^{\gamma M}  ,
\end{equation}
\begin{equation}\label{just6}
  \left\| \frac{u_n^\varepsilon}{\varepsilon} - \frac{U^{(M)}_{n}}{\|U^{(M)}_{n}\|_\varepsilon} \right\|_\varepsilon \le 2 d_0^{-1} \tilde{C}_M  \, \varepsilon^{\gamma M} .
\end{equation}

Taking  \eqref{t0.1},  \eqref{additional_est1} and the fact that $\gamma \in (\frac{2}{3}, 1)$   into account, we deduce from \eqref{just5} that
\begin{equation}\label{just7-}
  \left| \lambda_n(\varepsilon)  - \mathcal{L}^{(\lfloor \frac{2}{3} M\rfloor -1)}_{n}(\varepsilon) \right| \le \tilde{\tilde{C}}_M  \,  \varepsilon^{\lfloor \frac{2}{3} M\rfloor},
\end{equation}
where $\lfloor t \rfloor$ is the integer part of a real number $t.$
Since $M$ is arbitrary positive integer,  and  it follows from \eqref{just7-} that
for any $p\in \Bbb N$
\begin{equation}\label{just7}
  \left| \lambda_n(\varepsilon)  - \mathcal{L}^{(p -1)}_{n}(\varepsilon) \right| \le {C}_p  \,
\varepsilon^{p}.
\end{equation}

With regard to \eqref{normalized} and \eqref{additional_est2},  we get  from  \eqref{just6}  that
\begin{equation}\label{just9}
  \left\| u_n^\varepsilon - \tau_M(\varepsilon) \,  U^{(M)}_{n} \right\|_\varepsilon \le
{C  }_M  \, \varepsilon^{\gamma M +1} ,
\end{equation}
where $\tau_M(\varepsilon)= \varepsilon  /  \|U^{(M)}_{n}\|_\varepsilon,$ \  $ 0 < \frac{1}{c_4} \le \tau_M(\varepsilon) \le \frac{1}{c_3},$ and due to
\eqref{EF_lim_prob}
\begin{equation}\label{ta-m}
\tau_M(\varepsilon)= \frac{1}{\sqrt{\pi}} + \mathcal{O}(\varepsilon^{\frac{\gamma}{2}}) \quad \text{as} \quad \varepsilon \to 0.
\end{equation}

Thus, the following theorem is proved.
\begin{theorem}[$\alpha =0$]
  For each $n\in \Bbb N$ the $n$-th eigenvalue $\lambda_n(\varepsilon)$ of the problem \eqref{1.1}  is  decomposed into the asymptotic series
$$
\lambda_n(\varepsilon) \approx \Lambda_n + \sum_{k=1}^{+\infty} \varepsilon^k \, \mu_{k,n} \quad \text{as} \quad \varepsilon \to 0,
$$
where $\Lambda_n$ is the $n$-th eigenvalue of the limit spectral problem \eqref{main}  and the coefficients $\{\mu_{k,n}\}$ are defined by formulas
\eqref{mu_1}  and \eqref{mu_k}, and  the asymptotic estimate \eqref{just7} holds; in particular
\begin{equation}\label{t5}
\lambda_n(\varepsilon) = \Lambda_n + \varepsilon \mu_{1, n} + \mathcal{O}(\varepsilon^2)
\quad \text{as} \quad \varepsilon \to 0 .
\end{equation}

For the corresponding eigenfunction $u_n^\varepsilon$ normalized by \eqref{normalized} and \eqref{normalized+} one can construct  the approximation function $U_n^{(M)} \in \mathcal{H}_\varepsilon$ defined with \eqref{aaN} such that the asymptotic estimate \eqref{just9} holds  for any $M\in \Bbb N, \ M \ge 3,$  where $\gamma$ is a fixed number from the interval $(\frac23, 1).$
\end{theorem}

From  \eqref{just9} with $M= 3,$ similarly as in \cite[Corollary 5.1]{Mel_Klev_M2AS-2018},  it follows the corollary.
\begin{corollary}\label{corollary1}
For the difference between the eigenfunction  $u_n^\varepsilon$ of the problem \eqref{1.1} and the partial sum
$$
U_n^{(1)} = \sum\limits_{i=1}^3 \chi_{\ell_0}^{(i)} \left(\frac{x_i}{\varepsilon^\gamma}\right)
    \left( W_{n}^{(i)}(x_i) + \varepsilon\, w_{1,n}^{(i)} (x_i) \right)
    +
    \left(1 - \sum\limits_{i=1}^3 \chi_{\ell_0}^{(i)} \left(\frac{x_i}{\varepsilon^\gamma}\right) \right)
  \left( W_n^{(1)}(0) + \varepsilon   N_{1,n} \left( \frac{x}{\varepsilon} \right) \right), \quad x\in\Omega_\varepsilon,
$$
of \eqref{asymp_expansion}, where $\widetilde{W}_n=\{W_{n}^{(i)}\}_{i=1}^3$ is the eigenfunction of the limit  problem~$(\ref{main}),$
 the  following asymptotic estimates  hold:
\begin{equation}\label{t6}
 \left\| \, u_n^\varepsilon - \tau_3(\varepsilon) \, U_n^{(1)} \right\|_{H^1(\Omega_\varepsilon)} \leq C_1 \,
 \varepsilon^{2} ;
\end{equation}
\begin{equation}\label{t7}
 \left\| \, u_n^\varepsilon  - \tau_3(\varepsilon) \, \left( W_{n}^{(i)} + \varepsilon\, w_{1,n}^{(i)} \right) \right\|_{H^1(\Omega_{\varepsilon,\gamma}^{(i)})} \leq C_2 \,  \varepsilon^{2}
\end{equation}
in the thin cylinders $\Omega_{\varepsilon,\gamma}^{(i)} :=
 \Omega_\varepsilon^{(i)} \cap \big\{ x\in \Bbb{R}^3 : \
 x_i\in I_{\varepsilon, \gamma}^{(i)}:= (3\ell_0\varepsilon^\gamma, \ell_i) \big\}, \ i\in \{1, 2, 3\};$
and
 \begin{equation}\label{t-joint0}
     \left\| \, \nabla_{x}u_n^\varepsilon   - \tau_3(\varepsilon) \nabla_{\xi} \, N_{1,n} \right\|_{L^2(\Omega^{(0)}_{\varepsilon, \ell_0})}
 \le  C_3 \, \varepsilon^{\frac52}
\end{equation}
in  the  neighbourhood
$\Omega^{(0)}_{\varepsilon, \ell_0} :=
 \Omega_\varepsilon\cap \big\{ x : \ \ x_i<2\ell_0\varepsilon, \ i=1,2,3 \big\}$
of the node  $\Omega^{(0)}_{\varepsilon}.$ Here $\tau_3(\varepsilon)$ has the asymptotics \eqref{ta-m}.
 \end{corollary}

Using the Cauchy-Buniakovskii-Schwarz inequality and  the continuously embedding of
the space $H^1(I_{\varepsilon, \gamma}^{(i)})$  in $C\Big(\overline{I_{\varepsilon, \gamma}^{(i)}}\Big),$
it follows from (\ref{t7}) the following corollary.

\begin{corollary}\label{corollary2}
If $h_i(x_i) \equiv h_i \equiv const, \, (i=1,2,3),$ then
\begin{equation}\label{t9}
 \left\| \, E^{(i)}_\varepsilon(u_n^\varepsilon) - \tau_3(\varepsilon) \,W_{n}^{(i)} \right\|_{H^1(I_{\varepsilon, \gamma}^{(i)})} \leq \widetilde{C}_2 \,
 \varepsilon,
\end{equation}
\begin{equation}\label{t10}
\max\nolimits_{x_i\in \overline{I_{\varepsilon, \gamma}^{(i)}}} \left| \, E^{(i)}_\varepsilon(u_n^\varepsilon)(x_i) - \tau_3(\varepsilon) \,W_{n}^{(i)}(x_i)
\right|  \leq \widetilde{C}_3 \,  \varepsilon,
\end{equation}
where
$$
\big(E^{(i)}_\varepsilon u_n^\varepsilon\big)(x_i)
 = \frac{1}{\pi \varepsilon^2\, h_i^2}
   \int_{\Upsilon^{(i)}_\varepsilon(0)}
  u_n^\varepsilon(x)\, d\overline{x}_i,
\quad i=1,2,3.
$$
\end{corollary}

\begin{remark}\label{remark-answers}
We see that the local geometric irregularity of the node  does not affect the view of  the limit spectral problem \eqref{main}. However,
the second terms $\{w_{1,n}^{(i)}\}_{i=1}^3$ of the regular asymptotics \eqref{regul}
and the second term $\mu_{1,n}$ in the asymptotic expansion \eqref{exp-EVl}  for the eigenvalue $\lambda_n(\varepsilon)$
sensate the node impact through the values
 $\boldsymbol{\delta_{1,n}^{(2)}},$  $\boldsymbol{\delta_{1,n}^{(3)}}$ and $\boldsymbol{\delta_{1,n}^{(i)}}$ (see  \eqref{delta_1},  \eqref{d_1^*}
and \eqref{mu_1}, respectively).

Therefore, the problem \eqref{omega_probl_1}  for the second term of the regular asymptotics  should be considered in parallel with the limit spectral problem \eqref{main} to observe the node influence  on the asymptotic behavior of the spectrum through the asymptotic estimate \eqref{t5}. This conclusion coincides with the main inference of  \cite{Mel_Klev_AA-2019}.

In addition,  one-dimensional problems for two leading regular terms of the asymptotics on the corresponding  graph yield more correct information about  the original physical system in a thin graph-like structure (see \eqref{t6} and \eqref{t7}).

To obtain more information about the eigenfunction $u_n^\varepsilon$  in a neighborhood of the node,
 the problem \eqref{new_junc_probl_1} should be considered  in parallel with the limit spectral problem \eqref{main}
 in virtue of the estimate \eqref{t-joint0}.
\end{remark}

\section{Asymptotic approximations in the case $\alpha \in (0, 1)$}\label{Sec_4}

In this and subsequent sections, more attention will be paid  on the effect of the concentrated mass. Therefore, we assume that the thin cylinders $\{\Omega^{(i)}_\varepsilon\}_{i=1}^3$  are rectilinear, i.e., the functions $h_i \equiv h_i (0), \ i=1, 2, 3.$
As a result (see \S~\ref{regul_asymp}),  all coefficients $\{u_k\}_{k=2}^\infty$ in the regular part \eqref{regul} vanish.
This in turn means that  there is no boundary-layer part \eqref{prim+}  in the asymptotics.

 Since the problem now has two parameters, the asymptotic scale must be changed.  If the parameter $\alpha$ is an irrational number from $(0, 1),$ then the following  ansatzes of series for the approximation of  an eigenfunction $u^\varepsilon$  and the corresponding eigenvalue $\lambda(\varepsilon)$ (the index $n$ is omitted)  of  the problem~\eqref{1.1} are proposed:
\begin{enumerate}
  \item
   \begin{equation}\label{regul+alfa}
\mathcal{U}_\infty^{(i)} := \sum\limits_{k=0}^{+\infty} \sum\limits_{p=0}^{+\infty}  \varepsilon^{k - p\alpha } \, w_{k - p\alpha}^{(i)} (x_i)
\end{equation}
for the regular parts of the asymptotics in  each thin cylinder $\Omega^{(i)}_\varepsilon,  \ i=1,2,3,$ respectively;
 \item
  \begin{equation}\label{junc+alfa}
  \mathcal{N}^{(\infty)}  :=
\sum\limits_{k=0}^{+\infty} \sum\limits_{p=0}^{+\infty}  \varepsilon^{k - p\alpha } \, N_{k - p\alpha}\left(\frac{x}{\varepsilon}\right)
\end{equation}
for the inner part of the asymptotics in a neighborhood of the node $\Omega^{(0)}_\varepsilon;$
\item
\begin{equation}\label{exp-EVl+alfa}
 \mathcal{L}^{(\infty)} :=  \sum\limits_{k=0}^{+\infty} \sum\limits_{p=0}^{+\infty}  \varepsilon^{k - p\alpha } \, \mu_{k - p\alpha } .
\end{equation}
for  the  eigenvalue $\lambda(\varepsilon).$
\end{enumerate}

\begin{remark}\label{r4*1}
Hereinafter coefficients with negative indices are considered to be zero in all sums.
\end{remark}

\begin{definition}\label{def_4*1} The series \eqref{exp-EVl+alfa} is called the asymptotic expansion for an eigenvalue $ \lambda(\varepsilon)$
if for any $M \in \Bbb N$
\begin{equation}\label{def_as_exp}
\left| \lambda(\varepsilon) -  \mathcal{L}^{(M)}\right| = \overline{o} (\varepsilon^{M }) \quad \text{as} \ \
\varepsilon \to 0,
\end{equation}
where
$$
\mathcal{L}^{(M)} := \sum\limits_{k=0}^{M} \sum\limits_{\tiny{
  \begin{array}{c}
 p    \\
 0\le k - p\alpha \le M
  \end{array}
}}
  \varepsilon^{k - p\alpha} \, \mu_{k - p\alpha}
$$
is the partial sum of \eqref{exp-EVl+alfa}.
\end{definition}
This is a non-standard definition of an asymptotic expansion, since the better accuracy of the the approximation  we want, the more terms between integer powers of $\varepsilon$ must be determined. Similar definitions are given for such series in Banach spaces.

\begin{remark}\label{rational}
If  $\alpha$ is a rational number $\frac{m_0}{n_0},$ where ${m_0}, {n_0}$ are relatively prime numbers and ${m_0} <{n_0},$ then
the asymptotic scale $\{\varepsilon^{k - p\alpha}\}_{k, p \in \Bbb N_0},$ as is easy to see,
 becomes $\{\varepsilon^{\frac{k}{n_0}}\}_{k\in \Bbb N_0}.$
Thus, in this case we take  the following asymptotic ansatzes:
\begin{enumerate}
  \item
   \begin{equation}\label{regul+ration}
\mathcal{U}_\infty^{(i)}  :=   \sum\limits_{k=0}^{+\infty}  \varepsilon^{\frac{k}{n_0}} \, w_{\frac{k}{n_0}}^{(i)} (x_i)
\end{equation}
for the regular parts of the asymptotics in  each thin cylinder $\Omega^{(i)}_\varepsilon,  \ i=1,2,3,$ respectively;
 \item
  \begin{equation}\label{junc+ration}
  \mathcal{N}^{(\infty)} :=  \sum\limits_{k=0}^{+\infty} \varepsilon^{\frac{k}{n_0}} \, N_{\frac{k}{n_0}}\left(\frac{x}{\varepsilon}\right)
\end{equation}
for the inner part of the asymptotics in a neighborhood of the node $\Omega^{(0)}_\varepsilon;$
\item  and
\begin{equation}\label{exp-EVl+ration}
  \mathcal{L}^{(\infty)} :=   \sum\limits_{k=0}^{+\infty}  \varepsilon^{\frac{k}{n_0}} \, \mu_{\frac{k}{n_0}}
\end{equation}
for  the corresponding eigenvalue.
\end{enumerate}
\end{remark}

Further we will show in detail how to find the coefficients of the series \eqref{regul+alfa}, \eqref{junc+alfa} and \eqref{exp-EVl+alfa} and
prove the corresponding asymptotic estimates. At the end of this section,  we briefly consider the case~$\alpha =\frac{m_0}{n_0}.$

\subsection{The parameter $\alpha$ is an irrational number}\label{irrational}

Substituting \eqref{regul+alfa} and \eqref{exp-EVl+alfa}  into the differential equation and boundary conditions of the problem~\eqref{1.1}, collecting the coefficients at the same power of $\varepsilon$,  we get the following relations for all $\{k, p\} \subset \Bbb N_0$:
\begin{eqnarray}
    -   \dfrac{d^2 w_{k - p\alpha }^{(i)}(x_i)}{d{x_i}^2} &=&  \sum_{j=0}^{k}
\sum\limits_{m=0}^{p}
\mu_{j - m\alpha}  \,  w_{k -j - (p-m)\alpha} ^{(i)} (x_i), \quad x_i\in I_\varepsilon^{(i)}, \label{omega_probl_alfa_k}
\\
w_{k - p\alpha }^{(i)}(\ell_i)&=&0, \quad i=1, 2, 3. \label{BC_w_alfa_k}
\end{eqnarray}

Doing the same with \eqref{junc+alfa} and \eqref{exp-EVl+alfa}  and   matching  the expansions  \eqref{regul+alfa} and  \eqref{junc+alfa},
we get
\begin{equation}\label{N_probl_k+alfa}
\left\{
\begin{array}{rcl}
    -  \Delta_{\xi}N_{k - p\alpha}(\xi)  &=&  \sum\limits_{j=0}^{k-2} \sum\limits_{m=0}^{p}
\mu_{j - m\alpha}  \,  N_{k-2 -j - (p-m)\alpha}(\xi) , \quad \xi\in \Xi^{(i)},
\\\\
-  \Delta_{\xi}N_{k - p\alpha}(\xi)  &=&  \varrho_0(\xi)\sum\limits_{j=0}^{k-2}
\sum\limits_{m=0}^{p-1}
\mu_{j - m\alpha }  \,  N_{k-2 -j - (p-1-m)\alpha} (\xi) , \quad \xi\in \Xi^{(0)},
\\\\
   \partial_{{\boldsymbol \nu}_\xi}{N_{k - p\alpha }}(\xi) & = &
   0,     \quad \xi\in \partial \Xi,
\\[4mm]
 N_{k - p\alpha}(\xi)      & \sim &
   w^{(i)}_{k - p\alpha }(0) + \Psi^{(i)}_{k - p\alpha}(\xi_i),
   \quad \xi_i \to +\infty, \ \ \xi  \in \Xi^{(i)}, \quad i=1,2,3,
\end{array}
\right.
\end{equation}
where
$$
\Psi_{k - p\alpha}^{(i)}(\xi_i) =   \sum\limits_{j=1}^{k} \dfrac{\xi_i^j}{j!}\,
     \dfrac{d^j w_{k -j - p\alpha}^{(i)}}{dx_i^j} (0)
 \quad  i=1,2,3.
$$
The solvability condition for the corresponding problem for $\widetilde{N}_{k - p\alpha}$ (see \S~\ref{inner_asymp})  looks as follows:
\begin{equation}\label{trans*k-p*alpfa}
    \sum\limits_{i=1}^3 h_i^2 (0) \frac{d w_{k -1 - p\alpha}^{(i)}}{dx_i} (0)
 =  \boldsymbol{d_{k -1 - p\alpha}^*},
\end{equation}
where
\begin{multline}\label{const_d*k-p*alpfa}
\boldsymbol{d_{k  - p\alpha}^*}
 =  \sum\limits_{i=1}^3  h_i^2 (0) \, \sum\limits_{j=1}^k \frac{\ell_0^j}{j!} \sum_{t=0}^{k-j } \sum_{m=0}^{p} \mu_{t -m \alpha} \,
    \dfrac{d^{j-1} w_{k-j- t - (p- m)\alpha}^{(i)}}{d x_i^{j-1}} (0)
\\
-  \frac{1}{\pi}  \sum_{j=0}^{k-1} \sum_{m=0}^{p} \mu_{j - m \alpha} \, \sum\limits_{i=1}^3 \int_{\Xi^{(i)}}\Big(
N_{k-1-j -(p -m)\alpha}(\xi) -  G^{(i)}_{k-1-j -(p-m)\alpha}(\xi_i) \Big)  \,  d\xi
\\
     -  \frac{1}{ \pi}\int_{\Xi^{(0)}} \varrho_0(\xi) \sum_{j=0}^{k-1} \sum_{m=0}^{p-1} \mu_{j - m \alpha} \, N_{k-1-j -(p-1 -m)\alpha}(\xi)  \, d\xi,
\end{multline}
where
\begin{equation}\label{formula2-alfa}
G^{(i)}_{k - p\alpha}(\xi_i)  := w_{k - p\alpha}^{(i)}(0) + \Psi_{k - p\alpha}^{(i)}(\xi_i), \quad \xi\in\Xi^{(i)},
\quad i=1,2,3.
\end{equation}

\medskip

{\bf 1.} Fix some $M\in \Bbb N;$ the index $k\in \{0, 1,\ldots,M\}.$  Then
in the corresponding partial sums of \eqref{regul+alfa},
\eqref{junc+alfa} and \eqref{exp-EVl+alfa} there is a finite  number of terms  (see Definition~\ref{def_4*1}).

For the first term  $N_0$ in \eqref{junc+alfa} we obtain the problem \eqref{junc_probl_general-k=0}, and as a result,
the first transmission condition \eqref{trans0} for $\{w_0^{(i)}\}_{i=1}^3$ and  $N_0 \equiv  \widetilde{N}_0 \equiv w_0^{(1)} (0)$
(see for more detail \S~\ref{sub_3*4}).
The same problem \eqref{junc_probl_general-k=0} is  also  for the coefficients $N_{k - p\alpha}$  if $k - p\alpha < 1,$
 but with the following conditions:
$
N_{k - p\alpha}(\xi) \to    w^{(i)}_{k - p\alpha}(0)
$
as
$  \xi_i \to +\infty, \ \ \xi  \in \Xi^{(i)},$ $i=1,2,3.$
Thus,
\begin{equation}\label{trans_1+alfa}
 w^{(1)}_{k - p\alpha}(0)=w^{(2)}_{k - p\alpha }(0)=w^{(3)}_{k - p\alpha }(0),
\end{equation}
and $N_{k - p\alpha} \equiv  \widetilde{N}_{k - p\alpha} \equiv w_{k - p\alpha}^{(1)} (0)$
if $k - p\alpha < 1;$ in this case $N_{k - p\alpha} - G^{(i)}_{k - p\alpha} \equiv 0$ for $i=1, 2, 3.$

For the term $N_1$ in \eqref{junc+alfa} we get  the problem  \eqref{new_junc_probl_1} and its  solvability condition gives  the second transmission condition \eqref{transmisiont1} for $\{w_0^{(i)}\}_{i=1}^3.$  Thus, we arrive to the  limit spectral problem
\begin{equation}\label{limitSpeProb}
 \left\{\begin{array}{rclr}
  -  \dfrac{d^2 w_0^{(i)}}{d{x_i^2}}(x_i) & = &
    \mu_0\, \, w_0^{(i)}(x_i), &                                x_i\in I_i, \quad i=1,2,3,
 \\[2mm]
    w_0^{(i)}(\ell_i) & = &
    0, &           i=1,2,3,
 \\[2mm]
    w_0^{(1)} (0) \ \, = \ \, w_0^{(2)} (0) & = &
    w_0^{(3)} (0), &
  \\[2mm]
    \sum_{i=1}^3  h_i^2 (0) \dfrac{dw_{0}^{(i)}}{dx_i} (0) & = &
    0. &
 \end{array}\right.
\end{equation}
This is the same problem as in the case $\alpha =0.$ Thus, all eigenvalue of  the  limit spectral problem (\ref{limitSpeProb}) are simple and the corresponding eigenfunctions belong to the space $ \mathcal{H}_0$ (recall that $h_i \equiv h_i (0), \ i=1, 2, 3)$
and  they can be uniquely selected  so that the conditions \eqref{EF_lim_prob} and \eqref{EF_addion_lim_prob} are satisfied.
So, we have found  the first terms $\{w_0^{(i)}\}_{i=1}^3,$ $N_0$  and $\mu_0$ in the series \eqref{regul+alfa},  \eqref{junc+alfa} and \eqref{exp-EVl+alfa},  respectively. Here $w_0^{(i)}=W^{(i)}_n$ for some index $n\in \Bbb N$ (the index $n$ is omitted in further calculations).

\medskip

{\bf 2.} If  $p=0,$ then the problem \eqref{N_probl_k+alfa} looks as follows:
\begin{equation}\label{N_k}
\left\{
\begin{array}{rcll}
    -  \Delta_{\xi}N_{k }(\xi)  &=&  \sum_{j=0}^{k-2}
\mu_{j}  \,  N_{k-2 -j}(\xi) , & \xi\in \Xi^{(i)},
\\
-  \Delta_{\xi}N_{k }(\xi)  &=& 0 , & \xi\in \Xi^{(0)},
\\
   \partial_{{\boldsymbol \nu}_\xi}{N_{k}}(\xi) & = &
   0,     & \xi\in \partial \Xi,
\\
 N_{k}(\xi)      & \sim &
   w^{(i)}_{k}(0) + \Psi^{(i)}_{k}(\xi), & \xi_i \to +\infty, \ \ \xi  \in \Xi^{(i)}, \quad i=1,2,3.
\end{array}
\right.
\end{equation}

\begin{remark}\label{rem_to_N_k}
 From \eqref{N_k} it follows  that coefficients of the series \eqref{regul+alfa},  \eqref{junc+alfa} and \eqref{exp-EVl+alfa} with integer indices  are defined independently of the other coefficients.  To define those coefficients we can use all formulas from \S~\ref{inner_asymp}--\ref{justification}  with regard that the right-hand sides vanish on the node $\Xi^{(0)}$ in the corresponding problems for $N_k.$  For example, to define $\{w_1^{(i)}\}_{i=1}^3$  we get the same problem \eqref{omega_probl_1}, but now
\begin{gather}
\boldsymbol{d_{1}^*}
 =  \mu_0 \, w_{0}^{(1)}(0)  \sum\limits_{i=1}^3   \ell_0 h_i^2 (0) , \label{d_1^*alfa}
     \\
 \mu_{1}   =  \left(\mu_0 \, w_{0}^{(1)}(0)\right)^2  \sum\limits_{i=1}^3  \ell_0 h_i^2 (0)
      + \left(\dfrac{\boldsymbol{\delta_1^{(2)}} h^2_2(0)}{\ell_2}
 + \dfrac{ \boldsymbol{\delta_1^{(3)}} h^2_3(0)}{\ell_3}\right) \,   \mu_0 \, w_{0}^{(1)}(0); \label{mu1}
\end{gather}
please compare with  \eqref{d_1^*} and \eqref{mu_1}, respectively. The constants $\boldsymbol{\delta_1^{(2)}}$ and $\boldsymbol{\delta_1^{(3)}}$ are defined with \eqref{delta_1}.
\end{remark}


{\bf 3.}
To understand how to find coefficients with indices $k - p \alpha$, first we take  partial sums of the series \eqref{regul+alfa},  \eqref{junc+alfa} and \eqref{exp-EVl+alfa}  with $M=2.$ For the series \eqref{junc+alfa}
it is the  sum
$$
 \mathcal{N}^{(2)}(\xi) = \sum\limits_{k=0}^{2} \sum\limits_{\tiny{
  \begin{array}{c}
 p   \\
 0\le k - p\alpha  \le 2
  \end{array}
}}
  \varepsilon^{k - p\alpha} \, N_{k - p\alpha}(\xi) .
$$
 A number of terms in this sum is finite and depends on $\alpha.$
 For definiteness and explanation, we assume that $\alpha \in (\frac{2}{3}, 1).$ Then
\begin{equation}\label{part-sum2-inner}
 \mathcal{N}^{(2)} = N_0  + \varepsilon^{1 -\alpha} N_{1 -\alpha} + \varepsilon^{2 -2\alpha} N_{2 -2\alpha} + \varepsilon N_{1} +
 \varepsilon^{2 -\alpha} N_{2 -\alpha} + \varepsilon^{2} N_{2}.
\end{equation}
The coefficients $N_0, N_1, N_2$ are already determined, and we know that $N_{1 - \alpha } \equiv  w_{1 - \alpha}^{(1)} (0)$ and
$N_{2 - 2\alpha } \equiv  w_{2 - 2\alpha}^{(1)} (0)$ (see the first item).

To find the second transmission condition for $\{w_{1  -\alpha}^{(i)}\}_{i=1}^3$ we should consider the problem for  $N_{2 - \alpha}$
\begin{equation}\label{junc2-alfa}
 \left\{\begin{array}{rcll}
  -\Delta_{\xi}{N_{2 - \alpha}}(\xi) & = &
   0,                         &
   \quad \xi\in\Xi^{(i)}, \quad i=1, 2, 3,
 \\[1mm]
  -\Delta_{\xi}{N_{2 - \alpha}}(\xi) & = & \mu_0 \, w_0^{(1)} (0)\, \varrho_0(\xi),   &
   \quad \xi\in\Xi^{(0)},
 \\[1mm]
   \partial_{\nu_\xi}{N_{2 - \alpha}}(\xi) & = &
   0,                               &
   \quad \xi\in \partial \Xi,
 \\
   N_{2 - \alpha}(\xi)                                                               & \sim &
   w^{(i)}_{2 - \alpha}(0) + \xi_i \dfrac{dw_{1 - \alpha}^{(i)}}{dx_i}(0),                  &
   \quad \xi_i \to +\infty, \ \ {\xi}_i \in \Xi^{(i)}, \ i=1,2,3.
 \end{array}\right.
\end{equation}
The solvability condition for the corresponding problem for $\widetilde{N}_{2 - \alpha}$ looks as follows:
\begin{equation}\label{w_1-alfa}
   \sum\limits_{i=1}^3 h_i^2 (0) \,\dfrac{d w_{1-\alpha}^{(i)}}{dx_i} (0) =
- \frac{\mu_0}{\pi} \, w_{0}^{(1)}(0) \, \int_{\Xi^{(0)}} \varrho_0(\xi)  \, d\xi  .
\end{equation}
Thus, for $\widetilde{w}_{1-\alpha}=\{w_{1 - \alpha}^{(i)}\}_{i=1}^3$ we have the problem
\begin{equation}\label{w1-alpfa}
 \left\{\begin{array}{rclr}
  -  \dfrac{d^2 w_{1 - \alpha}^{(i)}}{d{x_i^2}}(x_i) & = &
    \mu_0\, \, w_{1- \alpha}^{(i)}(x_i) + \mu_{1- \alpha}\,  w_{0}^{(i)}(x_i), &   x_i\in I_i,
 \\[4mm]
    w_{1- \alpha}^{(i)}(\ell_i) & = &
    0, &           i=1,2,3,
 \\[2mm]
    w_{1-  \alpha}^{(1)} (0) \ \, = \ \, w_{1 - \alpha}^{(2)} (0) & = &
    w_{1 - \alpha}^{(3)} (0), &
  \\[2mm]
    \sum_{i=1}^3  h_i^2 (0) \dfrac{dw_{1 - \alpha}^{(i)}}{dx_i} (0) & = & - \displaystyle{ \frac{\mu_0}{\pi} \, w_{0}^{(1)}(0) \, \int_{\Xi^{(0)}} \varrho_0(\xi)  \, d\xi } .&
 \end{array}\right.
\end{equation}
Applying the Fredholm theory (see the end of \S~\ref{sub_3*4}),  we find that
\begin{equation} \label{mu1-alfa}
  \mu_{1-\alpha}   = -  \frac{\left(\mu_0 \, w_{0}^{(1)}(0)\right)^2}{\pi}  \, \int_{\Xi^{(0)}} \varrho_0(\xi)  \, d\xi
 \end{equation}
and there exists a unique solution to \eqref{w1-alpfa} such that $ \langle \widetilde{w}_{1 - \alpha}, \widetilde{w}_0 \rangle_0 =0.$
In addition, $N_{1 - \alpha} \equiv  w_{1 - \alpha}^{(1)} (0).$
Since $\widetilde{w}_{1 - \alpha} \in \mathcal{H}_0,$ it follows from \eqref{limit_identity} that $\big( \widetilde{w}_{1 - \alpha}, \widetilde{w}_0 \big)_{\mathcal{V}_0} =0$ as well.

Thanks to Proposition~\ref{tverd1} we can state that the problem \eqref{junc2-alfa} has a unique solution if and only if
\begin{equation}\label{trans1}
w_{2-\alpha}^{(1)}(0) =  w_{2-\alpha}^{(2)}(0) - \boldsymbol{\delta_{2-\alpha}^{(2)}} = w_{2-\alpha}^{(3)}(0) - \boldsymbol{\delta_{2-\alpha}^{(3)}},
\end{equation}
 where $\boldsymbol{\delta_{2-\alpha}^{(2)}}$ and $\boldsymbol{\delta_{2-\alpha}^{(3)}}$ are defined by formulas \eqref{const_d_0}.

To find the second transmission condition for $\{w_{2  -\alpha}^{(i)}\}_{i=1}^3,$ we should consider the problem
\begin{equation}\label{junc3-alfa}
 \left\{\begin{array}{rcll}
  -\Delta_{\xi}{N_{3 - \alpha}}(\xi) & = &
   \mu_0 N_{1-\alpha}(\xi) + \mu_{1-\alpha} w_0^{(1)} (0),                         &
   \quad \xi\in\Xi^{(i)}, \quad i=1, 2, 3,
 \\[1mm]
  -\Delta_{\xi}{N_{3 - \alpha}}(\xi) & = & \varrho_0(\xi)\left(\mu_0 N_1(\xi) +  \mu_1 \, w_0^{(1)} (0)\right) ,   &
   \quad \xi\in\Xi^{(0)},
 \\[1mm]
   \partial_{\nu_\xi}{N_{3 - \alpha}}(\xi) & = &
   0,                               &
   \quad \xi\in \partial \Xi,
 \\
   N_{3 - \alpha}(\xi)                                                               & \sim &
   w^{(i)}_{3 - \alpha}(0) + \xi_i \dfrac{dw_{2 - \alpha}^{(i)}}{dx_i}(0) + \dfrac{\xi^2_i}{2}  \dfrac{d^2w_{1 - \alpha}^{(i)}}{dx_i^2}(0) ,                  &
   \quad \xi_i \to +\infty,  \ i=1,2,3.
 \end{array}\right.
\end{equation}
\begin{remark}
 From the second equation in \eqref{junc3-alfa} we see  that coefficients with integer indices begin to mix with others and affect their determination.
\end{remark}
The solvability condition of the corresponding problem for $\widetilde{N}_{3 - \alpha}$ gives the second transmission condition
for $\widetilde{w}_{2  -\alpha}$ and as a result, we get the problem
\begin{equation}\label{w2-alpfa}
 \left\{\begin{array}{rclr}
  -  \dfrac{d^2 w_{2 - \alpha}^{(i)}}{d{x_i^2}}(x_i)-  \mu_0\, \, w_{2- \alpha}^{(i)}(x_i)  & = &
     \mu_{1- \alpha}\,  w_{1}^{(i)}(x_i) + \mu_{1}\,  w_{1- \alpha}^{(i)}(x_i)
  + \mu_{2- \alpha}\,  w_{0}^{(i)}(x_i)    , &   x_i\in I_i,
 \\[4mm]
    w_{2- \alpha}^{(i)}(\ell_i) & = &
    0, &           i=1,2,3,
 \\[2mm]
    w_{2-  \alpha}^{(1)} (0) \ \, = \ \, w_{2 - \alpha}^{(2)} (0) - \boldsymbol{\delta_{2-\alpha}^{(2)}} & = &
    w_{2 - \alpha}^{(3)} (0) - \boldsymbol{\delta_{2-\alpha}^{(3)}}, &
  \\[2mm]
    \sum_{i=1}^3  h_i^2 (0) \dfrac{dw_{2 - \alpha}^{(i)}}{dx_i} (0) & = &\boldsymbol{d_{2  - \alpha}^*} ,&
 \end{array}\right.
\end{equation}
where $\boldsymbol{d_{2  - \alpha}^*}$ is defined with the help of \eqref{const_d*k-p*alpfa} and
$$
\boldsymbol{d_{2  - \alpha}^*} =
\sum_{i=1}^{3} \ell_0 h_i^2(0)  \left(
    \mu_0 w_{1- \alpha}^{(i)}(0) + \mu_{1-\alpha} w_{0}^{(i)}(0) \right)
-   \displaystyle{ \frac{1}{\pi} \,  \int_{\Xi^{(0)}} \varrho_0(\xi)
    \left(  \mu_0 N_{1}(\xi) + \mu_{1} w_{0}^{(1)}(0) \right)   d\xi } .
$$
Similar as in  \S~\ref{sub_3*4},  we find
\begin{equation}\label{mu_2-alfa}
  \mu_{2  - \alpha} =  \mu_0 \, w_{0}^{(1)}(0) \, \boldsymbol{ D_{2  - \alpha}}
-  \mu_0  \sum_{i=1}^3 \int_0^{\ell_i} \Phi_{2  - \alpha}^{(i)}(x_i; \mu_{1-\alpha}, \mu_{1}) \, w_{0}^{(i)} (x_i) \, dx_i
\end{equation}
and a unique solution to the problem \eqref{w2-alpfa}. Here  $\boldsymbol{ D_{2  - \alpha}}$ is defined with \eqref{D_k} and
$\Phi_{2  - \alpha}^{(i)}(x_i; \mu_{1-\alpha}, \mu_{1}) = h_i(0) \left(\mu_{1- \alpha}\,  w_{1}^{(i)}(x_i) + \mu_{1}\,  w_{1- \alpha}^{(i)}(x_i)\right).$

It remains to find $\{w_{2 - 2\alpha}^{(i)}\}_{i=1}^3.$ For this we should consider the problem
\begin{equation}\label{junc3-2alfa}
 \left\{\begin{array}{rcll}
  -\Delta_{\xi}{N_{3 - 2\alpha}}(\xi) & = & 0,                         &
   \quad \xi\in\Xi^{(i)}, \quad i=1, 2, 3,
 \\[1mm]
  -\Delta_{\xi}{N_{3 - 2\alpha}}(\xi) & = & \varrho_0(\xi)\left(\mu_0 w_{1-\alpha}^{(1)}(0)  +  \mu_{1-\alpha} \, w_0^{(1)} (0)\right) ,   &
   \quad \xi\in\Xi^{(0)},
 \\[1mm]
   \partial_{\nu_\xi}{N_{3 - 2\alpha}}(\xi) & = &
   0,                               &
   \quad \xi\in \partial \Xi,
 \\
   N_{3 - 2\alpha}(\xi)                                                               & \sim &
   w^{(i)}_{3 - 2\alpha}(0) + \xi_i \dfrac{dw_{2 - 2\alpha}^{(i)}}{dx_i}(0),                  &
   \quad \xi_i \to +\infty,  \ i=1,2,3.
 \end{array}\right.
\end{equation}
The solvability condition for the corresponding problem for $\widetilde{N}_{3 - 2\alpha}$ gives the second transmission condition
for $\widetilde{w}_{2  -2\alpha}$ and as a result, we get the problem
\begin{equation}\label{w2-2alpfa}
 \left\{\begin{array}{rclr}
  -  \dfrac{d^2 w_{2 - 2\alpha}^{(i)}}{d{x_i^2}}(x_i)-  \mu_0\, \, w_{2- 2\alpha}^{(i)}(x_i)  & = &
     \mu_{1- \alpha}\,  w_{1-\alpha}^{(i)}(x_i)   + \mu_{2- 2\alpha}\,  w_{0}^{(i)}(x_i)    , &   x_i\in I_i,
 \\[4mm]
    w_{2- 2\alpha}^{(i)}(\ell_i) & = &
    0, &           i=1,2,3,
 \\[2mm]
    w_{2-  2\alpha}^{(1)} (0) \ \, = \ \, w_{2 - 2\alpha}^{(2)} (0) & = &
    w_{2 - 2 \alpha}^{(3)} (0) , &
  \\[2mm]
    \sum_{i=1}^3  h_i^2 (0) \dfrac{dw_{2 - 2\alpha}^{(i)}}{dx_i} (0) & = &\boldsymbol{d_{2  - 2\alpha}^*} ,&
 \end{array}\right.
\end{equation}
where
$$
\boldsymbol{d_{2  - 2\alpha}^*} = -   \displaystyle{ \frac{1}{\pi} \, \left( \mu_0 w_{1-\alpha}^{(1)}(0) + \mu_{1-\alpha} w_{0}^{(1)}(0) \right) \int_{\Xi^{(0)}} \varrho_0(\xi)\, d\xi } .
$$
Similar as in  \S~\ref{sub_3*4} we find
\begin{equation}\label{mu_2-2alfa}
  \mu_{2  - 2\alpha} = \mu_0 \, w_{0}^{(1)}(0)\, \boldsymbol{d_{2  - 2\alpha}^*}
-  \mu_0  \sum_{i=1}^3 \int_0^{\ell_i} \Phi_{2  - \alpha}^{(i)}(x_i; \mu_{1-\alpha}) \, w_{0}^{(i)} (x_i) \, dx_i
\end{equation}
and a unique solution to the problem \eqref{w2-alpfa}. Here
$\Phi_{2  - 2\alpha}^{(i)}(x_i; \mu_{1-\alpha}) = h_i(0)\,  \mu_{1- \alpha}\,  w_{1- \alpha}^{(i)}(x_i).$ In addition,
$N_{2 - 2\alpha} \equiv  w_{2 - 2\alpha}^{(1)} (0).$

Thus, for $\alpha \in \left( \frac{2}{3}, 1\right)$ we have defined all coefficients in the partial sums  \eqref{part-sum2-inner},
\begin{gather*}
    \mathcal{U}_2^{(i)} = w_{0}^{(i)}   + \varepsilon^{1 -\alpha} w^{(i)}_{1 -\alpha} + \varepsilon^{2 -2\alpha} w^{(i)}_{2 -2\alpha} + \varepsilon w^{(i)}_{1} +
 \varepsilon^{2 -\alpha} w^{(i)}_{2 -\alpha} + \varepsilon^{2} w^{(i)}_{2}, \quad i=1, 2, 3;
   \\
   \mathcal{L}^{(2)} = \mu_0  + \varepsilon^{1 -\alpha} \mu_{1 -\alpha} + \varepsilon^{2 -2\alpha} \mu_{2 -2\alpha} + \varepsilon \mu_{1} +
 \varepsilon^{2 -\alpha} \mu_{2 -\alpha} + \varepsilon^{2} \mu_{2}.
\end{gather*}

\medskip

{\bf 4.}
Now let us take $M=3$ and for definiteness additionally  assume  that $\alpha \in \left( \frac{2}{3}, \frac{3}{4}\right)$
(the closer the parameter $\alpha$ to $1,$  the more terms between $\varepsilon^0$ and $\varepsilon^1).$
 Then the partial sum for the series~\eqref{junc+alfa} is
\begin{multline}\label{part-sum3-inner}
 \mathcal{N}^{(3)} = N_0  + \varepsilon^{3 -4\alpha} N_{3 -4\alpha} +
 \varepsilon^{1 -\alpha} N_{1 -\alpha} + \varepsilon^{2 -2\alpha} N_{2 -2\alpha} + \varepsilon^{3 -3\alpha} N_{3 -3\alpha} + \varepsilon N_{1}
 \\
  + \varepsilon^{2 -\alpha} N_{2 -\alpha} + \varepsilon^{3 -2\alpha} N_{3 -2\alpha} +  \varepsilon^{2} N_{2} + \varepsilon^{3 -\alpha}
  N_{3 -\alpha}
 +  \varepsilon^{3} N_{3}.
\end{multline}
We see that new terms $N_{3 -4\alpha}, \, N_{3 -3\alpha}$ and
$N_{3 -2\alpha}, \, N_{3 -\alpha}$ appear between the terms with integer powers of $\varepsilon.$
By virtue of Remark~\ref{rem_to_N_k} we can regard that all terms $N_{3},$ $\{w^{(i)}_3\}_{i=1}^3,$ $\mu_3$
are uniquely defined.  The coefficients $N_{3 -2\alpha}$ and $N_{3 -\alpha}$ are solutions to the problems  \eqref{junc3-2alfa} and \eqref{junc3-alfa}, respectively (see the items 2,  3 and \S\ref{inner_asymp}).
 Only the coefficients $N_{3 -4\alpha}$ and $N_{3 -3\alpha}$ have not yet been determined.
From the first item it follows that $N_{3 - 4\alpha } \equiv  w_{3 - 4\alpha}^{(1)} (0)$ and
$N_{3 - 3\alpha } \equiv  w_{3 - 3\alpha}^{(1)} (0).$

 To find $\{w^{(i)}_{3 - 4\alpha}\}_{i=1}^3,$ we consider the problem
 \begin{equation}\label{junc_probl4-4alfa}
 \left\{\begin{array}{rcll}
  -\Delta_{\xi}{N_{4 -4 \alpha}}(\xi) & = &
   0,                         &
   \quad \xi\in\Xi,
 \\[1mm]
   \partial_{\nu_\xi}{N_{4-4 \alpha}}(\xi) & = &
   0,                               &
   \quad \xi\in \partial \Xi,
 \\
   N_{4-4 \alpha}(\xi)                                                               & \sim &
   w^{(i)}_{4-4 \alpha}(0) + \xi_i \dfrac{d w_{3-4 \alpha}^{(i)}}{dx_i}(0),                  &
   \quad \xi_i \to +\infty, \ \ {\xi}_i \in \Xi^{(i)}, \ i=1,2,3.
 \end{array}\right.
\end{equation}
With the  substitution \eqref{new-solution} the problem \eqref{junc_probl4-4alfa} is reduced to the corresponding problem for
$\widetilde{N}_{4 -4 \alpha}$ and its solvability condition (see Proposition~\ref{tverd1}) gives the second Kirchhoff condition for
$\{w_{3 - 4\alpha}^{(i)}\}_{i=1}^3.$  As a result, we get the problem
 \begin{equation}\label{main3-4alpfa<1}
 \left\{\begin{array}{rclr}
  -  \dfrac{d^2 w_{3-4\alpha}^{(i)}}{d{x_i^2}}(x_i) - \mu_0\,  w_{3-4\alpha}^{(i)}(x_i)  & = &
     \mu_{3-4\alpha}\,  w_{0}^{(i)}(x_i), &   x_i\in I_i,
 \\[2mm]
    w_{3-4\alpha}^{(i)}(\ell_i) & = &
    0, &           i=1,2,3,
 \\[2mm]
    w_{3-4\alpha}^{(1)} (0) \ \, = \ \, w_{3-4\alpha}^{(2)} (0) & = &
    w_{3-4\alpha}^{(3)} (0), &
  \\[2mm]
    \sum_{i=1}^3  h_i^2 (0) \dfrac{dw_{3-4\alpha}^{(i)}}{dx_i} (0) & = &
    0.&
 \end{array}\right.
\end{equation}
Since $\mu_0$ is an eigenvalue and $\widetilde{w}_0$ is the corresponding eigenfunction of the limit spectral problem
\eqref{limitSpeProb}, we deduce that $\mu_{3-4\alpha} =0$ and $w_{3-4\alpha}^{(i)}\equiv 0, \ i=1, 2, 3$
(see the end of \S~\ref{sub_3*4}). This means that $N_{3-4\alpha}\equiv 0$ and $N_{4 - 4\alpha } \equiv  w_{4 - 4\alpha}^{(1)} (0).$
Recall that these relations are obtained under the assumption that $\alpha \in \left( \frac{2}{3}, \frac{3}{4}\right).$

For all $\alpha \in (0,1 )$  we can similarly prove that  $N_{k_0 - p_0 \alpha}\equiv w_{k_0 - p_0\alpha}^{(i)}\equiv 0, \ i=1, 2, 3,$ and $\mu_{k_0 - p_0\alpha}= 0$ if  $k_0 - p_0\, \alpha < 1 - \alpha.$ Thus, the following statement holds.

\begin{proposition}\label{prop_4_1}
The second terms in the asymptotics both for an eigenvalue and the corresponding eigenfunction of the problem \eqref{1.1} have the order $\varepsilon^{1-\alpha}.$
\end{proposition}

The problem \eqref{N_probl_k+alfa} with $k=4$ and $p=3$ looks as follows:
\begin{equation}\label{junc4-3alfa}
 \left\{\begin{array}{rcll}
  -\Delta_{\xi}{N_{4 - 3\alpha}}(\xi) & = & 0,                         &
    \xi\in\Xi^{(i)}, \ i=1, 2, 3,
 \\[1mm]
  -\Delta_{\xi}{N_{4 - 3\alpha}}(\xi) & = & \varrho_0(\xi)\left(\mu_0 w_{2-2\alpha}^{(1)}(0)  +  \mu_{1-\alpha} w_{1-\alpha}^{(1)}(0)
   +  \mu_{2-2\alpha}  w_0^{(1)} (0)\right),   &
    \xi\in\Xi^{(0)},
 \\[1mm]
   \partial_{\nu_\xi}{N_{4 - 3\alpha}}(\xi) & = &
   0,                               &
    \xi\in \partial \Xi,
 \\
   N_{4 - 3\alpha}(\xi)                                                               & \sim &
   w^{(i)}_{4 - 3\alpha}(0) + \xi_i \dfrac{dw_{3 - 3\alpha}^{(i)}}{dx_i}(0),                  &
   \xi_i \to +\infty,  \ i=1,2,3.
 \end{array}\right.
\end{equation}
The solvability condition of the corresponding problem for $\widetilde{N}_{4 - 3\alpha}$ yields the second transmission condition
for $\widetilde{w}_{3  -3\alpha}$ and as a result, we get the problem
\begin{equation}\label{w3-3alpfa}
 \left\{\begin{array}{rclr}
  -  \dfrac{d^2 w_{3 - 3\alpha}^{(i)}}{d{x_i^2}} -  \mu_0 w_{3- 3\alpha}^{(i)}& = &
     \mu_{1- \alpha} w_{2-2\alpha}^{(i)}(x_i)   + \mu_{2- 2\alpha} w_{1-\alpha}^{(i)}(x_i) + \mu_{3- 3\alpha}  w_{0}^{(i)}(x_i)    , &   x_i\in I_i,
 \\[2mm]
    w_{3- 3\alpha}^{(i)}(\ell_i) & = &
    0, &           i=1,2,3,
 \\[2mm]
    w_{3-  3\alpha}^{(1)} (0) \ \, = \ \, w_{3 - 3\alpha}^{(2)} (0) & = &
    w_{3 - 3 \alpha}^{(3)} (0) , &
  \\[2mm]
    \sum_{i=1}^3  h_i^2 (0) \dfrac{dw_{3 - 3\alpha}^{(i)}}{dx_i} (0) & = &\boldsymbol{d_{3  - 3\alpha}^*} ,&
 \end{array}\right.
\end{equation}
where
$$
\boldsymbol{d_{3  - 3\alpha}^*} = -   \displaystyle{ \frac{1}{\pi} \,  \left(\mu_0\, w_{2-2\alpha}^{(1)}(0)  +  \mu_{1-\alpha} \,w_{1-\alpha}^{(1)}(0)
   +  \mu_{2-2\alpha} \, w_0^{(1)} (0)\right) \int_{\Xi^{(0)}} \varrho_0(\xi)\,   d\xi } .
$$
By the same way as at the end of \S~\ref{sub_3*4} we get
\begin{equation}\label{mu_2-2alfa}
  \mu_{3  - 3\alpha} = \mu_0 \, w_{0}^{(1)}(0)\, \boldsymbol{d_{3  - 3\alpha}^*}
-  \mu_0  \sum_{i=1}^3 \int_0^{\ell_i} \Phi_{3  - 3\alpha}^{(i)}(x_i; \mu_{1-\alpha}, \mu_{2-2\alpha}) \, w_{0}^{(i)} (x_i) \, dx_i
\end{equation}
and a unique solution to the problem \eqref{w3-3alpfa}. Here
$$
\Phi_{3  - 3\alpha}^{(i)}(x_i; \mu_{1-\alpha},\mu_{2-2\alpha}) = h_i(0)\left( \mu_{1- \alpha} w_{2- 2\alpha}^{(i)}(x_i) + \mu_{2- 2\alpha} w_{1-\alpha}^{(i)}(x_i)\right).
$$

Similar as before,  we find $\mu_{3  -2\alpha}, \ \mu_{3  - \alpha}$ and  $\{w^{(i)}_{3 -2\alpha}\}_{i=1}^3, \
\{w^{(i)}_{3 -\alpha}\}_{i=1}^3$ from the problems
\begin{equation}\label{w3-2alpfa}
 \left\{\begin{array}{rclr}
  -  \dfrac{d^2 w_{3 - 2\alpha}^{(i)}}{d{x_i^2}} -  \mu_0 w_{3- 2\alpha}^{(i)}  & = & \mu_{1}  w_{2 -2\alpha}^{(i)}(x_i) +
    \mu_{1- \alpha}  w_{2 -\alpha}^{(i)}(x_i) + \mu_{2-\alpha} w_{1- \alpha}^{(i)}(x_i) &
    \\
   &  &  + \, \mu_{2- 2\alpha}  w_{1}^{(i)}(x_i)   + \mu_{3- 2\alpha}  w_{0}^{(i)}(x_i), &   x_i\in I_i,
 \\[2mm]
    w_{3- 2\alpha}^{(i)}(\ell_i) & = &
    0, &           i=1,2,3,
 \\[2mm]
    w_{3-  2\alpha}^{(1)} (0) & = &  w_{3 - 2\alpha}^{(2)} (0) - \boldsymbol{\delta_{3-2\alpha}^{(2)}} \ \, = \ \,
    w_{3 - 2\alpha}^{(3)} (0) - \boldsymbol{\delta_{3-2\alpha}^{(3)}}, &
  \\[2mm]
    \sum_{i=1}^3  h_i^2 (0) \dfrac{dw_{3 - 2\alpha}^{(i)}}{dx_i} (0) & = &\boldsymbol{d_{3  - 2\alpha}^*} ,&
 \end{array}\right.
\end{equation}
and
\begin{equation}\label{w3-alpfa}
 \left\{\begin{array}{rclr}
  -  \dfrac{d^2 w_{3 - \alpha}^{(i)}}{d{x_i^2}} -  \mu_0 w_{3- \alpha}^{(i)}  & = & \mu_{1}  w_{2 -\alpha}^{(i)}(x_i) +
    \mu_{1- \alpha}  w_{2}^{(i)}(x_i) + \mu_{2} w_{1- \alpha}^{(i)}(x_i) &
    \\
   &  &  + \, \mu_{2- \alpha}  w_{1}^{(i)}(x_i)   + \mu_{3- \alpha}  w_{0}^{(i)}(x_i), &   x_i\in I_i,
 \\[2mm]
    w_{3- \alpha}^{(i)}(\ell_i) & = &
    0, &           i=1,2,3,
 \\[2mm]
    w_{3-  \alpha}^{(1)} (0) & = &  w_{3 - \alpha}^{(2)} (0) - \boldsymbol{\delta_{3-\alpha}^{(2)}} \ \, = \ \,
    w_{3 - \alpha}^{(3)} (0) - \boldsymbol{\delta_{3-\alpha}^{(3)}}, &
  \\[2mm]
    \sum_{i=1}^3  h_i^2 (0) \dfrac{dw_{3 - \alpha}^{(i)}}{dx_i} (0) & = &\boldsymbol{d_{3  - \alpha}^*} ,&
 \end{array}\right.
\end{equation}
where
\begin{multline*}
\boldsymbol{d_{3  - 2\alpha}^*} =
\sum_{i=1}^{3} \ell_0 h_i^2(0)  \left(
    \mu_0 w_{2- 2\alpha}^{(i)}(0) + \mu_{1-\alpha} w_{1-\alpha}^{(i)}(0) + \mu_{2-2\alpha} w_{0}^{(i)}(0) \right)
 \\
   -  \displaystyle{ \frac{1}{\pi} \,  \int_{\Xi^{(0)}} \varrho_0(\xi)
    \left(  \mu_0 N_{2-\alpha }(\xi) + \mu_{1} w_{1-\alpha}^{(1)}(0)  + \mu_{1-\alpha} N_{1 }(\xi) + \mu_{2-\alpha} w_{0}^{(1)}(0) \right)   d\xi},
 \end{multline*}
\begin{multline*}
\boldsymbol{d_{3  - \alpha}^*} =
\sum_{i=1}^{3} \ell_0 h_i^2(0)   \left(  \mu_0 w_{2- \alpha}^{(i)}(0) + \sum_{m=0}^{1} \mu_{1- m\alpha} w_{1- (1-m)\alpha}^{(i)}(0) + \mu_{2-\alpha} w_{0}^{(i)}(0) \right)
\\
+ \sum_{i=1}^{3}  h_i^2(0)
     \frac{\ell_0^2}{2} \left(\mu_0 \frac{dw_{1- \alpha}^{(i)}}{dx_i}(0) + \mu_{1-\alpha}\frac{dw_{1}^{(i)}}{dx_i}(0) \right)
- \frac{\mu_0}{\pi} \sum\limits_{i=1}^3 \int_{\Xi^{(i)}}
\left( N_{2-\alpha}(\xi) - G^{(i)}_{2-\alpha}(\xi_i)
\right) d\xi
\\
- \frac{\mu_{1-\alpha}}{\pi} \sum\limits_{i=1}^3 \int_{\Xi^{(i)}}
\left( N_{1}(\xi) - G^{(i)}_{1}(\xi_i) \right) d\xi
    -  \displaystyle{ \frac{1}{\pi} \,  \int_{\Xi^{(0)}} \varrho_0(\xi)
    \sum_{j=0}^{2}  \mu_j N_{2-j}(\xi)\, d\xi},
 \end{multline*}
and the constants $\boldsymbol{\delta_{3-2\alpha}^{(i)}}, \ \boldsymbol{\delta_{3-\alpha}^{(i)}}, \ i=2,3,$ are defined
by formulas \eqref{const_d_0}.

Thus, for $\alpha \in \left( \frac{2}{3}, \frac{3}{4}\right)$ we have defined all coefficients in the partial sums  \eqref{part-sum3-inner},
\begin{gather*}
    \mathcal{U}_3^{(i)} = w_{0}^{(i)}   + \varepsilon^{1 -\alpha} w^{(i)}_{1 -\alpha} + \varepsilon^{2 -2\alpha} w^{(i)}_{2 -2\alpha} +
    \varepsilon^{3 -3\alpha} w^{(i)}_{3 -3\alpha} + \varepsilon w^{(i)}_{1}
     +
\varepsilon^{2 -\alpha} w^{(i)}_{2 -\alpha}
\\
 + \varepsilon^{3 -2\alpha} w^{(i)}_{3 -2\alpha}
+
 \varepsilon^{2} w^{(i)}_{2} +
 \varepsilon^{3 -\alpha} w^{(i)}_{3 -\alpha}  + \varepsilon^{3} w^{(i)}_{3}, \quad i=1, 2, 3;
    \\[2mm]
   \mathcal{L}^{(3)} = \mu_0  + \varepsilon^{1 -\alpha} \mu_{1 -\alpha} + \varepsilon^{2 -2\alpha} \mu_{2 -2\alpha} +
   \varepsilon^{3 -3\alpha} \mu_{3 -3\alpha} + \varepsilon \mu_{1} +
 \varepsilon^{2 -\alpha} \mu_{2 -\alpha}
  \\
  + \varepsilon^{3 -2\alpha} \mu_{3 -2\alpha} + \varepsilon^{2} \mu_{2} + \varepsilon^{3 -\alpha} \mu_{3 -\alpha} + \varepsilon^{3} \mu_{3}.
\end{gather*}

\subsection{Justification  $(\alpha$ is an irrational number)}\label{subsec_just+alpfa}

From \S~\ref{irrational} we see that all terms of series \eqref{regul+alfa}, \eqref{junc+alfa}, \eqref{exp-EVl+alfa}  can be uniquely defined.
With the help of the cut-off functions \eqref{cut-off-functions} we construct the series
\begin{equation}\label{asymp_expansion+alpha}
\sum\limits_{k=0}^{+\infty} \sum\limits_{p=0}^{+\infty}  \varepsilon^{k - p\alpha }
    \Big(
    \overline{u  }_{k - p\alpha} (x; \, \varepsilon, \, \gamma)
  + \overline{N  }_{k - p\alpha} (x; \, \varepsilon, \, \gamma)
    \Big),
\quad x\in\Omega_\varepsilon,
\end{equation}
where
$$
\overline{u}_{k - p\alpha} (x; \, \varepsilon, \, \gamma)
 := \sum\limits_{i=1}^3 \chi_{\ell_0}^{(i)} \left(\frac{x_i}{\varepsilon^\gamma}\right)
     w_{k - p\alpha}^{(i)} (x_i),
\qquad
\overline{N}_{k - p\alpha} (x; \, \varepsilon, \, \gamma)  := \left(1 - \sum\limits_{i=1}^3 \chi_{\ell_0}^{(i)} \left(\frac{x_i}{\varepsilon^\gamma}\right) \right)
    N_{k - p\alpha} \left( \frac{x}{\varepsilon} \right),
$$
$\gamma$ is a fixed number from the interval $(\frac23, 1).$
Denote by
\begin{equation}\label{aaN+alpha}
\mathfrak{U}^{(M)}(x;\varepsilon)
 :=  \sum\limits_{k=0}^{M} \sum\limits_{\tiny{
  \begin{array}{c}
 p    \\
 0\le k - p\alpha \le M
  \end{array}
}}  \varepsilon^{k - p\alpha}
    \Big(
    \overline{u  }_{k - p\alpha} (x; \, \varepsilon, \, \gamma)
    + \overline{N  }_{k - p\alpha} (x; \, \varepsilon, \, \gamma)
    \Big),
\quad x\in\Omega_\varepsilon,
\end{equation}
 the partial sum of $(\ref{asymp_expansion+alpha}),$ where $M \in \Bbb N$ (see Definition~\ref{def_4*1}).
 Obviously, $ \mathfrak{U}^{(M)}  \in \mathcal{H}_\varepsilon.$

Substituting $\mathfrak{U}^{(M)}$ and $\mathcal{L}^{(M)}$  into the problem~(\ref{1.1}) in place of $u^\varepsilon$ and $\lambda(\varepsilon),$ with the help of \eqref{calcul1}  we find that
\begin{itemize}
  \item
  for each $i\in \{1, 2, 3\}$ thanks to \eqref{omega_probl_alfa_k}
\begin{equation}\label{res_cylinder}
  - \Delta_x \mathfrak{U}^{(M)} -  \mathcal{L}^{(K)} \mathfrak{U}^{(M)} = \overline{o}(\varepsilon^M)\, \mathfrak{R}^{(i)}_\varepsilon(x_i) \quad \text{in} \ \
  \Omega_{\varepsilon,\gamma}^{(i)},
\end{equation}
  where  $\Omega_{\varepsilon,\gamma}^{(i)} :=  \Omega_\varepsilon^{(i)} \cap \big\{ x\in \Bbb{R}^3 : \
 x_i\in I_{\varepsilon, \gamma}^{(i)}:= (3\ell_0\varepsilon^\gamma, \ell_i) \big\};$
   \item
thanks to \eqref{N_probl_k+alfa}
\begin{equation}\label{res_node}
  - \Delta_x \mathfrak{U}^{(M)} -  \rho_\varepsilon\, \mathcal{L}^{(M)} \mathfrak{U}^{(M)} =
  \left\{
    \begin{array}{ll}
     \overline{o}(\varepsilon^{M-2})\, \mathfrak{R}^{(0,i)}_\varepsilon(\frac{x}{\varepsilon}) & \hbox{in} \
     \Omega_\varepsilon^{(i)} \cap \big\{ x: \,  x_i\in  (\ell_0\varepsilon, 2\ell_0\varepsilon^\gamma) \big\}, \ \ i\in \{1, 2, 3\},
     \\[2mm]
     \overline{o}(\varepsilon^{M-2-\alpha})\,  \mathfrak{R}^{(0)}_\varepsilon(\frac{x}{\varepsilon}) & \hbox{in} \ \Omega^{(0)}_{\varepsilon};
    \end{array}
  \right.
\end{equation}
  \item
  and for each $i\in \{1, 2, 3\}$  thanks to \eqref{omega_probl_alfa_k} and the first  relation in \eqref{N_probl_k+alfa}
  \begin{multline}\label{res_cylinder_layer}
  - \Delta_x \mathfrak{U}^{(M)} -  \mathcal{L}^{(M)} \mathfrak{U}^{(M)}
    = \overline{o}(\varepsilon^M)\, \mathfrak{R}^{(i)}_\varepsilon(x_i) + \overline{o}(\varepsilon^{M-2})\, \mathfrak{R}^{(0,i)}_\varepsilon \big(\frac{x}{\varepsilon}\big)
 \\
  + \,  \mathfrak{G}_\varepsilon^{(i)}\big(x_i,\frac{x}{\varepsilon}\big) \quad \text{in} \ \
   \Omega_\varepsilon^{(i)} \cap \big\{ x: \  x_i\in  (2\ell_0\varepsilon^\gamma, 3\ell_0\varepsilon^\gamma) \big\}.
\end{multline}
\end{itemize}
In \eqref{res_cylinder} --  \eqref{res_cylinder_layer},  the residuals  $\overline{o}(\varepsilon^M)\, \mathfrak{R}^{(i)}_\varepsilon,$
$\overline{o}(\varepsilon^{M-2})\, \mathfrak{R}^{(0,i)}_\varepsilon,$ $\overline{o}(\varepsilon^{M-2-\alpha})\,  \mathfrak{R}^{(0)}_\varepsilon$ consist of such sums
$$
\varepsilon^{k - p \alpha}\sum_{j=0}^{k} \sum\limits_{m=0}^{p} \mu_{j - m\alpha}  \,  w_{k -j - (p-m)\alpha} ^{(i)} (x_i),
$$
where $ \ k\in\{M+1,\ldots,2M\}, \ M < k - p \alpha,$
$$
\varepsilon^{k -2 - p \alpha} \sum\limits_{j=0}^{k-2} \sum\limits_{m=0}^{p}
\mu_{j - m\alpha}  \,  N_{k-2 -j - (p-m)\alpha}(\tfrac{x}{\varepsilon}\big) ,
$$
where $ \ k\in\{M+1,\ldots,2M+2\}, \ M  < k - p \alpha,$
$$
\varepsilon^{k -2 - (p-1) \alpha} \varrho_0\big(\tfrac{x}{\varepsilon}\big)\sum\limits_{j=0}^{k-2}
 \sum\limits_{m=0}^{p-1}
\mu_{j - m\alpha }  \,  N_{k-2 -j - (p-1-m)\alpha}(\tfrac{x}{\varepsilon}\big),
$$
where $ \ k\in\{M+1,\ldots,2M+2\}, \ M -\alpha < k - (p-1)\alpha,$
respectively; and
\begin{multline*}
\mathfrak{G}_\varepsilon^{(i)}\big(x_i,\frac{x}{\varepsilon}\big) =  \sum\limits_{k=0}^{M} \sum\limits_{\tiny{
  \begin{array}{c}
 p    \\
 0\le k - p\alpha \le M
  \end{array}
}}  \varepsilon^{k - p\alpha -\gamma} \Bigg(
\frac{\partial}{\partial x_i} \left( \Big(\chi_{\ell_0}^{(i)}\Big)'\Big(\frac{x_i}{\varepsilon^\gamma}\Big)
       \,  \Big(  w_{k - p\alpha}^{(i)} (x_i) - N_{k - p\alpha} \Big( \frac{x}{\varepsilon} \Big) \Big) \right)
  \\
    +  \Big(\chi_{\ell_0}^{(i)}\Big)'\Big(\frac{x_i}{\varepsilon^\gamma}\Big)
   \, \frac{\partial}{\partial x_i} \left(w_{k - p\alpha}^{(i)} (x_i) - N_{k - p\alpha} \Big( \frac{x}{\varepsilon} \Big)  \right) \Bigg).
\end{multline*}
With the help of the the Taylor formula with the integral remainder term  for functions $\{w_{k - p\alpha}^{(i)}\}$ at the point $x_i=0$
and taking into account the fourth relation in \eqref{N_probl_k+alfa}, Remark~\ref{rem_exp-decrease} and the
fact that the support of  $\mathfrak{G}_\varepsilon^{(i)}$ belongs to    $ \Omega_\varepsilon^{(i)} \cap \big\{ x: \  x_i\in  [2\ell_0\varepsilon^\gamma, 3\ell_0\varepsilon^\gamma] \big\} $ $(i\in \{1, 2, 3\}),$ we derive
\begin{equation}\label{reminder1}
  \left| \int_{\Omega_\varepsilon^{(i)}}  \mathfrak{G}_\varepsilon^{(i)} \, \psi \, dx \right|
  \le \overline{o}\left( \varepsilon^{\gamma M - \frac{\gamma}{2} +1}\right) \|\psi\|_\varepsilon \quad \forall \psi \in \mathcal{H}_\varepsilon
\end{equation}

Using \eqref{res_cylinder} -- \eqref{reminder1}, with the help of
\eqref{est3},  \eqref{est4},  \eqref{ineq1}, \eqref{inner_asympt}   and Remark~\ref{rem_exp-decrease} we can estimate the right-hand side in the integral identity
\begin{equation}\label{just+alfa}
  \int_{\Omega_\varepsilon}\left( \nabla \mathfrak{U}^{(M)} \cdot \nabla \psi  -   \rho_\varepsilon(x)\, \mathcal{L}^{(M)}(\varepsilon)\,
  \mathfrak{U}^{(M)}\, \psi \right) dx = \mathfrak{F}_\varepsilon (\psi) \quad \forall \psi \in \mathcal{H}_\varepsilon,
\end{equation}
namely
\begin{align}\label{just+alfa+1}
|\mathfrak{F}_\varepsilon (\psi)| & \le \left( \overline{o}\left(\varepsilon^{M+1}\right)  + \overline{o}\left(\varepsilon^{\gamma M}\right)
+ \overline{o}\left(\varepsilon^{M -\alpha}\right) + \overline{o}\left( \varepsilon^{\gamma M - \frac{\gamma}{2} +1}\right) \right) \|\psi\|_\varepsilon \notag
 \\[2mm]
    & =
\left( \overline{o}\left(\varepsilon^{M -\alpha}\right) + \overline{o}\left( \varepsilon^{\gamma M - \frac{\gamma}{2} +1}\right) \right) \|\psi\|_\varepsilon \quad \text{as} \ \ \varepsilon \to 0.
\end{align}

Similar as in Subsection~\ref{Sec_justification}, with the help of the operator $A_\varepsilon$ (see \eqref{operator1}) we get
\begin{equation}\label{just4+alfa}
\left\| A_\varepsilon\left( \frac{\mathfrak{U}^{(M)}}{\|\mathfrak{U}^{(M)}\|_\varepsilon}\right)  - \left(\mathcal{L}^{(M)}(\varepsilon)\right)^{-1} \, \frac{\mathfrak{U}^{(M)}}{\|\mathfrak{U}^{(M)}\|_\varepsilon}  \right\|_\varepsilon
 \le \left( \overline{o}\left(\varepsilon^{M -\alpha-1}\right) + \overline{o}\left( \varepsilon^{\gamma M - \frac{\gamma}{2} }\right) \right)  .
\end{equation}
Due to Lemma~\ref{Vishik} the inequality \eqref{just4+alfa} means that the following statement holds.

\begin{proposition}\label{Prop_4_2}
For each  $n\in \Bbb N$ in any neighbourhood of the eigenvalue $\Lambda_n$ of the limit spectral problem~\eqref{limitSpeProb}  there is an eigenvalue of the problem \eqref{1.1} for $\varepsilon$ small enough.
\end{proposition}

Next to identify  an eigenvalue that lies in a neighbourhood of $\Lambda_n,$ we need the following fact.

\begin{lemma}[$\alpha <1$]\label{Lemma_main_Convergence} For each $n\in \Bbb N$
\begin{equation}\label{convergence}
  \lambda_n(\varepsilon) \rightarrow \Lambda_n \quad \text{as}  \ \ \varepsilon \to 0.
 \end{equation}
 \end{lemma}

\begin{proof} {\bf 1.} In this item the index $n$ is omitted.
Thanks to \eqref{t0.1} and \eqref{lower-est},  we can regard that
$\lambda(\varepsilon) \rightarrow  \Lambda \neq 0$ as $\varepsilon \to 0$ (select a subsequence if necessary).

With the help of the Cauchy--Bunyakovsky--Schwarz inequality and \eqref{normalized} it is easy to show  that the sequence
$\big\{E^{(i)}_\varepsilon u^\varepsilon\big\}_{\varepsilon>0}$ is bounded in $H^1(I_\varepsilon^{(i)}),$
where $ \ I_\varepsilon^{(i)} =\{x: \ x_i\in  (\varepsilon \ell_0, \ell_i), \ \overline{x_i}=(0,0)\}$ and
  $$
\big(E^{(i)}_\varepsilon u^\varepsilon\big)(x_i)
 = \frac{1}{\pi \varepsilon^2\, h_i^2}
   \int_{\Upsilon^{(i)}_\varepsilon(0)}
  u^\varepsilon(x)\, d\overline{x}_i,
\quad i=1,2,3.
$$
Extending with the constant $\big(E^{(i)}_\varepsilon u^\varepsilon\big)(\varepsilon \ell_0)$ the function
$E^{(i)}_\varepsilon u^\varepsilon$ on the segment $[0,  \varepsilon \ell_0],$ we get
\begin{equation}\label{EF_est_2}
  \left\| E^{(i)}_\varepsilon u^\varepsilon \right\|_{H^1(I_i)} \le C_2
  \end{equation}
since $\big|\big(E^{(i)}_\varepsilon u^\varepsilon\big)(\varepsilon \ell_0)\big| \le c_1 \big\|E^{(i)}_\varepsilon u^\varepsilon\big\|_{H^1(I^{(i)}_\varepsilon)} \le c_2.$ In addition, it is easy to show that
\begin{equation}\label{EF_est_3}
  \left| \big(E^{(i)}_\varepsilon u^\varepsilon\big)(\varepsilon \ell_0) - \big(E^{(j)}_\varepsilon u^\varepsilon\big)(\varepsilon \ell_0) \right| \le C_3 \, \varepsilon^{\frac{1}{2}}, \quad i\neq j \in \{1,2,3\}.
  \end{equation}
Owing to \eqref{EF_est_2} and \eqref{EF_est_3},  there exist a subsequence of $\{\varepsilon\}$
(again denoted by $\{\varepsilon\})$ and a function $\widetilde{v}=\{v^{(i)}\}_{i=1}^3 \in \mathcal{H}_0$
such that for each $i\in \{1, 2, 3\}$
\begin{equation}\label{weak_limit}
  E^{(i)}_\varepsilon u^\varepsilon \rightarrow v^{(i)} \quad \text{uniformly on} \ \ I_i \quad \text{and weakly in} \ \ H^1(I_i) \quad \text{as} \ \ \varepsilon \to 0.
\end{equation}

Now we re-write the integral identity \eqref{identity} for the problem \eqref{1.1} in the form
\begin{equation}\label{identity+graph}
\sum_{i=1}^{3} {\pi h_i^2} \int_{I_\varepsilon^{(i)}} \frac{d \big(E^{(i)}_\varepsilon u^\varepsilon\big)}{dx_i} \, \frac{d \Phi_\varepsilon}{dx_i} \, dx_i =
\lambda(\varepsilon) \left(  \sum_{i=1}^{3} {\pi h_i^2}\int_{I^{(i)}_\varepsilon} E^{(i)}_\varepsilon u^\varepsilon \, \Phi_\varepsilon \, dx_i + c_0 \varepsilon^{-\alpha -2}  \int_{\Omega^{(0)}_\varepsilon}  \varrho_0\big(\tfrac{x}{\varepsilon}\big)\, u^\varepsilon  \, dx\right).
\end{equation}
Here a test-function $\Phi_\varepsilon$ is taken  as follows :
$$
\Phi_\varepsilon(x)= \left\{
        \begin{array}{ll}
          c_0=\widetilde{\phi}(0), & x \in \ \Omega_\varepsilon^{(0)},
\\
\phi_i(0) + \big(\phi_i(\varepsilon 2 \ell_0) - \phi_i(0)\big) \big( \frac{x_i}{\varepsilon \ell_0} -1 \big) , & x \in \ \Omega_\varepsilon^{(i)} \cap \{x: \ x_i\in  [\varepsilon \ell_0, \varepsilon 2 \ell_0], \ \overline{x_i}=(0,0)\},
\\
          \phi_i(x_i), & x \in \ \Omega_\varepsilon^{(i)} \cap \{x: \ x_i\in  (\varepsilon 2 \ell_0, \ell_i), \ \overline{x_i}=(0,0)\}, \ \ i\in\{1, 2, 3\},
        \end{array}
      \right.
$$
where $\widetilde{\phi}=\{\phi_i\}_{i=1}^3$ is arbitrary smooth function from the space $\mathcal{ H}_0$ (see\eqref{space_H_0}).
Clearly, that $\Phi_\varepsilon \in \mathcal{H}_\varepsilon$ and $\Phi_\varepsilon \in \mathcal{H}_0,$ and in addition,
\begin{equation}\label{test_limit}
\Phi_\varepsilon \rightarrow \widetilde{\phi} \quad \text{in} \ \mathcal{ H}_0 \ \ \text{as} \ \ \varepsilon \to 0.
\end{equation}

In virtue of \eqref{normalized} and \eqref{est3},
  \begin{equation}\label{EF_est_1}
  \int_{\Omega^{(0)}_\varepsilon}   \big(u^\varepsilon\big)^2  \, dx \le C_1 \varepsilon^3.
  \end{equation}

Taking \eqref{weak_limit}, \eqref{test_limit} and  \eqref{EF_est_1}  into account, we can pass to the limit (as $\varepsilon \to 0)$ in \eqref{identity+graph} and get
$$
\sum_{i=1}^{3} {\pi h_i^2} \int_{I_i} \frac{d v^{(i)}}{dx_i} \, \frac{d \phi^{(i)}}{dx_i} \, dx_i =
\Lambda  \sum_{i=1}^{3} {\pi h_i^2} \int_{I_i} v^{(i)} \, \phi^{(i)} \, dx_i .
$$
This identity means  that $\widetilde{v}=\{v^{(i)}\}_{i=1}^3$  is an eigenfunction corresponding the eigenvalue $\Lambda$ if we show that the function $\widetilde{v} \neq \{0\}$ (see Definition~\ref{def3_1}).

From \eqref{normalized} and \eqref{identity} it follows that the eigenfunction $u^\varepsilon$ corresponding to the eigenvalue
$\lambda(\varepsilon)$ of the problem \eqref{1.1}  satisfies the equality
  \begin{equation}\label{norm+}
 \frac{1}{\varepsilon^2}   \sum_{i=1}^{3} \int_{\Omega^{(i)}_\varepsilon}  \big(u^\varepsilon\big)^2 \, dx + \varepsilon^{-\alpha -2}  \int_{\Omega^{(0)}_\varepsilon}  \varrho_0\big(\tfrac{x}{\varepsilon}\big)\, \big(u^\varepsilon\big)^2  \, dx = \frac{1}{\lambda(\varepsilon)}.
  \end{equation}

Using  \eqref{EF_est_1} and the inequality
$$
\left| \int_{I_\varepsilon^{(i)}} \Big( \big(E^{(i)}_\varepsilon u^\varepsilon\big)^2 - E^{(i)}_\varepsilon\big( (u^\varepsilon)^2 \big) \Big)  \, dx \right|\le C_3 \|u^\varepsilon\|^2_\varepsilon \le C_3 \varepsilon^2
$$
(it  was proved in a general form in \cite[Lemma 2.2]{Mel_Pop_2012}), we can pass to the limit in \eqref{norm+} and find that
  \begin{equation}\label{norm_limit}
 \sum_{i=1}^{3} \pi h_i^2 \int_{I_i}  \big(v^{(i)}\big)^2 \, dx_i  = \frac{1}{\Lambda} \neq 0.
  \end{equation}

{\bf 2.} Making use the diagonal process, we can extract a subsequence from
$\{\varepsilon\}$ (again denoted by $\{\varepsilon\})$ for which the following statements hold:
 \begin{itemize}
   \item for any $n\in \Bbb N$:  \  $ \lambda_n(\varepsilon) \rightarrow  \Lambda_n^*$  \ and \  $\Lambda_n^* \le \Lambda_{n+1}^*;$
\item for any $n\in \Bbb N$:  \
$E^{(i)}_\varepsilon u^\varepsilon_n \rightarrow v^{(i)}_n$
uniformly on $I_i$ \ and weakly in  $H^1(I_i)$ \ as  \ $\varepsilon \to 0$;
   \item $\widetilde{v}_n=\{v^{(i)}_n\}_{i=1}^3$ is an eigenfunction corresponding to the eigenvalue $\Lambda_n^*$ of the limit problem \eqref{limitSpeProb}.
 \end{itemize}

From \eqref{normalized} and \eqref{identity} it follows that for $n\neq m$
  \begin{equation}\label{orthog+}
 \frac{1}{\varepsilon^2}   \sum_{i=1}^{3} \int_{\Omega^{(i)}_\varepsilon}  u^\varepsilon_n \, u^\varepsilon_m \, dx + \varepsilon^{-\alpha -2}  \int_{\Omega^{(0)}_\varepsilon}  \varrho_0\big(\tfrac{x}{\varepsilon}\big)\, u^\varepsilon_n \, u^\varepsilon_m \, dx = 0.
  \end{equation}
With the help of \eqref{EF_est_1} and \eqref{weak_limit} we find the limit of the equality  \eqref{orthog+} as $\varepsilon \to 0,$ namely
 \begin{equation}\label{limit_orthog+}
  \sum_{i=1}^{3} h_i^2 \int_{I_i}  v^{(i)}_n \, v^{(i)}_m \, dx_i  = 0 \quad \text{for any} \ \ n\neq m .
  \end{equation}
In addition, because of \eqref{normalized+}  we can regard that $\left( v^{(i)}_n(\ell_1) \right)' > 0.$

Since the spectrum of the limit spectral problem  is simple, the equalities \eqref{limit_orthog+} and Proposition~\ref{Prop_4_2} mean
that for any $n\in \Bbb N$ \ $\Lambda_n^* = \Lambda_{n}, \ n\in \Bbb N.$

To complete the proof, it suffices to observe that similar arguments and results are valid for any subsequence of the sequence $\{\varepsilon\}$ chosen  in the  beginning of the proof. \end{proof}

Taking into account that the eigenfunctions $\{u^\varepsilon_n\}_{n\in \Bbb N}$ satisfy \eqref{normalized} and \eqref{normalized+}
and the eigenfunctions $ \{\widetilde{W}_n \}_{n \in \Bbb N}$ of the limit spectral problem \eqref{limitSpeProb} are
selected uniquely so that the conditions \eqref{EF_lim_prob} and \eqref{EF_addion_lim_prob} hold,  from the proof of Lemma~\ref{Lemma_main_Convergence} it follows the following statement.

\begin{corollary}
For any $n\in \Bbb N$ and $i\in \{1, 2, 3\}$
$$
E^{(i)}_\varepsilon u^\varepsilon_n \rightarrow \tfrac{1}{\sqrt{\pi}}\, W^{(i)}_n \quad \text{uniformly on} \ I_i \ \text{and weakly in}\  H^1(I_i)
\quad \text{as}  \ \varepsilon \to 0,
$$
where $\widetilde{W}_n=\{W^{(i)}_n\}_{i=1}^3$ is the eigenfunction corresponding to the eigenvalue $\Lambda_n$ of  the limit spectral problem~\eqref{limitSpeProb}, which satisfies \eqref{EF_lim_prob} and \eqref{EF_addion_lim_prob}.
\end{corollary}

Based on \eqref{just4+alfa},  Lemmas~\ref{Lemma_main_Convergence},~\ref{Vishik} and Proposition~\ref{prop_4_1}, similarly as in  Subsection~\ref{Sec_justification} we prove the theorem.

\begin{theorem}[$\alpha \in (0, 1) \cap (\Bbb R \setminus \Bbb Q)$ ]
  For each $n\in \Bbb N$ the $n$-th eigenvalue $\lambda_n(\varepsilon)$ of the problem \eqref{1.1}  is  decomposed into the asymptotic series
$$
\lambda_n(\varepsilon) \approx \Lambda_n +
\sum\limits_{k=1}^{+\infty} \sum\limits_{p=0}^{+\infty}  \varepsilon^{k - p\alpha } \, \mu_{k - p\alpha, n}
 \quad \text{as} \quad \varepsilon \to 0,
$$
where $\Lambda_n$ is the $n$-th eigenvalue of the limit spectral problem \eqref{limitSpeProb},
\ $\mu_{k - p\alpha, n} =0$ if $k - p\alpha < 1 - \alpha;$ in addition the asymptotic estimate \eqref{def_as_exp} holds,  in particular
\begin{equation}\label{t5+}
\lambda_n(\varepsilon) = \Lambda_n + \varepsilon^{1 - \alpha} \mu_{1 - \alpha,n} + \overline{o}(\varepsilon^{1 - \alpha})
\quad \text{as} \quad \varepsilon \to 0,
\end{equation}
where
\begin{equation}\label{mu_1-alfa_n}
\mu_{1 - \alpha,n} = -  \frac{\big(\Lambda_n \, W_{n}^{(1)}(0)\big)^2}{\pi}  \, \int_{\Xi^{(0)}} \varrho_0(\xi)  \, d\xi.
\end{equation}

For the corresponding eigenfunction $u_n^\varepsilon$ normalized by \eqref{normalized} and \eqref{normalized+} one can construct  the approximation function $\mathfrak{U}_n^{(M)} \in \mathcal{H}_\varepsilon$ defined with \eqref{aaN+alpha} such that the asymptotic estimate
\begin{equation}\label{just9+alfa}
  \left\| u_n^\varepsilon - \tau_M(\varepsilon) \,  \mathfrak{U}^{(M)}_{n} \right\|_\varepsilon =
 \overline{o}\left(\varepsilon^{M -\alpha}\right) + \overline{o}\big( \varepsilon^{\gamma M +1 - \frac{\gamma}{2} }\big)
\quad \text{as} \ \ \varepsilon \to 0,
\end{equation}
 holds  for any $M\in \Bbb N, \ M \ge 3.$ Here  $\tau_M(\varepsilon) = \frac{\varepsilon}{\|\mathfrak{U}^{(M)}_{n}\|_\varepsilon}$
and it has the asymptotics \eqref{ta-m}, $\gamma$ is a fixed number from the interval $(\frac23, 1).$
\end{theorem}

From  \eqref{just9+alfa} with $M= 3$ it follows the corollary.
\begin{corollary}\label{corollary1+alfa}
For the difference between the eigenfunction  $u_n^\varepsilon$ of the problem \eqref{1.1} and the partial sum
$$
\mathfrak{U}_n^{(1-\alpha)} = \sum\limits_{i=1}^3 \chi_{\ell_0}^{(i)} \left(\frac{x_i}{\varepsilon^\gamma}\right)
    \left( W_{n}^{(i)}(x_i) + \varepsilon\, w_{1-\alpha,n}^{(i)} (x_i) \right)
    +
    \Big(1 - \sum\limits_{i=1}^3 \chi_{\ell_0}^{(i)} \left(\frac{x_i}{\varepsilon^\gamma}\right)\Big) \left(W_n^{(1)}(0) + \varepsilon^{1-\alpha} w_{1-\alpha,n}^{(1)}(0)\right), \quad x\in\Omega_\varepsilon,
$$
of \eqref{asymp_expansion+alpha},  the asymptotic estimate
\begin{equation}\label{t6+alfa}
\frac{1}{\sqrt{|\Omega_\varepsilon|}}\, \left\| \, u_n^\varepsilon - \tau_3(\varepsilon) \, \mathfrak{U}_n^{(1-\alpha)} \right\|_{H^1(\Omega_\varepsilon)} =
 \overline{o}\big(\varepsilon^{1 - \alpha}\big) \quad (\text{as} \ \ \varepsilon \to 0)
\end{equation}
holds; \ in addition,
\begin{equation}\label{t9+alfa}
 \left\| \, E^{(i)}_\varepsilon(u_n^\varepsilon) - \tau_3(\varepsilon) \,W_{n}^{(i)} \right\|_{H^1(I_{\varepsilon, \gamma}^{(i)})} \leq C_1 \,
 \varepsilon^{1-\alpha},
\end{equation}
\begin{equation}\label{t10+alfa}
\max\nolimits_{x_i\in \overline{I_{\varepsilon, \gamma}^{(i)}}} \left| \, E^{(i)}_\varepsilon(u_n^\varepsilon)(x_i) - \tau_3(\varepsilon) \,W_{n}^{(i)}(x_i)
\right|  \leq C_2 \,  \varepsilon^{1-\alpha},
\end{equation}
where $I_{\varepsilon, \gamma}^{(i)}:= \{x: \ x_i\in (3\ell_0\varepsilon^\gamma, \ell_i),  \ \overline{x}_i=0\},  \ \ i\in \{1, 2, 3\}.$
\end{corollary}

\subsection{The case $\alpha = \frac{m_0}{n_0}$}\label{rational case}

Recall that  ${m_0}, {n_0}$ are relatively prime numbers and ${m_0} <{n_0}.$
Substituting \eqref{regul+ration} and \eqref{exp-EVl+ration}  into the differential equation and boundary conditions of the problem~\eqref{1.1}, collecting the coefficients at the same power of $\varepsilon$,  we get the following relations:
\begin{eqnarray}
    -   \frac{d^2}{d{x_i}^2}\left(w_{\frac{k}{n_0}}^{(i)}(x_i)\right) &=&  \mu_{0}  \,  w_{\frac{k}{n_0}} ^{(i)} (x_i) + \sum_{q=1}^{k}
\mu_{\frac{q}{n_0}}  \,  w_{\frac{k-q}{n_0}} ^{(i)} (x_i), \quad x_i\in I_\varepsilon^{(i)}, \label{omega_probl_ration_alfa_k}
\\
w_{\frac{k}{n_0}}^{(i)}(\ell_i)&=&0, \quad i=1, 2, 3. \label{BC_w_ration_alfa_k}
\end{eqnarray}

Doing the same with \eqref{junc+ration} and \eqref{exp-EVl+ration}  and   matching  the expansions  \eqref{regul+ration} and  \eqref{junc+ration},
we get
\begin{equation}\label{N_probl_k+alfa_ration}
\left\{
\begin{array}{rcl}
    -  \Delta_{\xi}N_{\frac{k}{n_0}}(\xi)  &=&  \displaystyle{ \sum\limits_{q=0}^{k- 2n_0}} \mu_{\frac{q}{n_0}}  \,  N_{\frac{k-q}{n_0} -2}(\xi) , \quad \xi\in \Xi^{(i)},
\\\\
-  \Delta_{\xi}N_{\frac{k}{n_0}}(\xi)  &=&  \varrho_0(\xi)\displaystyle{ \sum\limits_{q=0}^{k+ m_0 - 2 n_0}}
\mu_{\frac{q}{n_0}}  \,  N_{\frac{k+ m_0 -q}{n_0} -2} (\xi) , \quad \xi\in \Xi^{(0)},
\\\\
   \partial_{{\boldsymbol \nu}_\xi}{N_{\frac{k}{n_0}}}(\xi) & = &
   0,     \quad \xi\in \partial \Xi,
\\[4mm]
 N_{\frac{k}{n_0}}(\xi)      & \sim &
   w^{(i)}_{\frac{k}{n_0}}(0) + \Psi^{(i)}_{\frac{k}{n_0}}(\xi_i),
   \quad \xi_i \to +\infty, \ \ \xi  \in \Xi^{(i)}, \quad i=1,2,3,
\end{array}
\right.
\end{equation}
where
$$
\Psi_{\frac{k}{n_0}}^{(i)}(\xi_i) =   \sum\limits_{j=1}^{k} \dfrac{\xi_i^j}{j!}\,
     \dfrac{d^j w_{\frac{k}{n_0} -j }^{(i)}}{dx_i^j} (0),
 \quad  i=1,2,3.
$$
Recall that coefficients with negative indices are considered to be zero in all sums.
The solvability condition for the corresponding problem for $\widetilde{N}_{\frac{k}{n_0}}$ (see \S~\ref{inner_asymp})  looks as follows:
\begin{equation}\label{trans*k-p*alpfa_ration}
    \sum\limits_{i=1}^3 h_i^2 (0) \frac{d w_{\frac{k}{n_0} -1}^{(i)}}{dx_i} (0)
 =  \boldsymbol{d_{\frac{k}{n_0} -1 }^*},
\end{equation}
where
\begin{multline}\label{const_d*k-p*alpfa_ration}
\boldsymbol{d_{\frac{k}{n_0} -1}^*}
 =  \sum\limits_{i=1}^3  h_i^2 (0) \, \sum\limits_{j=1}^{k-1} \frac{\ell_0^j}{j!} \, \sum_{q=0}^{k-n_0(j+1)}  \mu_{\frac{q}{n_0}} \,
    \dfrac{d^{j-1} w_{\frac{k - q}{n_0} -  j-1}^{(i)}}{d x_i^{j-1}} (0)
\\
-  \frac{1}{\pi}  \sum_{q=0}^{k-2 n_0} \mu_{\frac{q}{n_0}} \, \sum\limits_{i=1}^3 \int_{\Xi^{(i)}}\Big( N_{\frac{k-q}{n_0} -2}(\xi) - G^{(i)}_{\frac{k-q}{n_0} -2}(\xi_i)\Big)  \,  d\xi
   \\     -  \frac{1}{ \pi}\int_{\Xi^{(0)}} \varrho_0(\xi) \sum_{q=0}^{k+ m_0 -2n_0}  \mu_{\frac{q}{n_0}} \, N_{\frac{k+m_0 -q}{n_0}-2}(\xi)  \, d\xi.
\end{multline}

Let us show how to determine coefficients of the expansions \eqref{regul+ration}, \eqref{junc+ration} and \eqref{exp-EVl+ration}.
The first and second terms of these series are defined in the same way as for an irrational number $\alpha$ (see Subsection~\ref{irrational}). So, $\mu_0$ is an eigenvalue and $\{w^{(i)}_0\}_{i=1}^3$ is the corresponding eigenfunction of the limit
spectral problem~\eqref{limitSpeProb} (the eigenfunctions satisfy the conditions \eqref{EF_lim_prob} and \eqref{EF_addion_lim_prob});
$N_0 \equiv w^{(1)}(0).$  Similarly we conclude that Proposition~\ref{prop_4_1} holds; this means that
\begin{equation}\label{first_terms_0}
\mu_{\frac{k}{n_0}} = 0, \quad w^{(i)}_{\frac{k}{n_0}} \equiv 0, \quad N_{\frac{k}{n_0}} \equiv 0 \quad \text{for} \ \ k\in \{1,\ldots,n_0-m_0 -1\}, \quad  \ n_0-m_0 \ge 2.
\end{equation}
The problem for  second nonzero terms $\{w^{(i)}_{\frac{n_0 -m_0}{n_0}}\}_{i=1}^3$ and
$\mu_{\frac{n_0 -m_0}{n_0}}$ coincides with the problem \eqref{w1-alpfa} and $\mu_{\frac{n_0 -m_0}{n_0}}$ is determined by the formula \eqref{mu1-alfa}. In addition, $N_{\frac{n_0 -m_0}{n_0}} \equiv w^{(1)}_{\frac{n_0 -m_0}{n_0}}(0).$

The next terms are defined as follows. From \eqref{N_probl_k+alfa_ration} with $k=n_0 -m_0+1$ we get the problem
\begin{equation}\label{junc_probl_n_0 -m_0+1}
 \left\{\begin{array}{rcll}
  -\Delta_{\xi}{N_{\frac{n_0 -m_0+1}{n_0}}}(\xi) & = &
   0,                         &
   \quad \xi\in\Xi,
 \\[1mm]
   \partial_{\nu_\xi}{N_{\frac{n_0 -m_0+1}{n_0}}}(\xi) & = &
   0,                               &
   \quad \xi\in \partial \Xi,
 \\
   N_{\frac{n_0 -m_0+1}{n_0}}(\xi)                                                               & \sim &
   w^{(i)}_{\frac{n_0 -m_0+1}{n_0}}(0) + \xi_i \dfrac{d w_{\frac{1 -m_0}{n_0}}^{(i)}}{dx_i}(0),                  &
   \quad \xi_i \to +\infty, \ \ {\xi}_i \in \Xi^{(i)}, \ i=1,2,3.
 \end{array}\right.
\end{equation}
If  $m_0=1,$ then $N_{\frac{n_0 -m_0+1}{n_0}}=N_1$  and this problem coincides with \eqref{new_junc_probl_1}.
The solvability condition for the corresponding problem for $\widetilde{N}_{1}$ is satisfied since it is  the second Kirchhoff condition in the limit problem~\eqref{limitSpeProb}. To define $\{w^{(i)}_{1}\}_{i=1}^3$ and
$\mu_{1}$ we get the problem
\begin{equation}\label{Prob_w_1}
 \left\{\begin{array}{rclr}
  -  \dfrac{d^2 w_1^{(i)}}{d{x_i^2}}(x_i) - \mu_0\, w_1^{(i)}(x_i) & = &
  \mu_{\frac{n_0 -1}{n_0}} \, w_{\frac{1}{n_0}}^{(i)}(x_i) + \mu_1\, w_0^{(i)}(x_i) , &                                x_i\in I_i, \quad i=1,2,3,
 \\[2mm]
    w_1^{(i)}(\ell_i) & = &
    0, &           i=1,2,3,
 \\[2mm]
   w_{1}^{(1)} (0) \ \, = \ \, w_{1}^{(2)} (0) - \boldsymbol{\delta_{1}^{(2)}} & = &
                                   w_{1}^{(3)} (0) - \boldsymbol{\delta_{1}^{(3)}}, &
 \\[3mm]
    \sum\limits_{i=1}^3 h_i^2 (0) \,\dfrac{d w_{1}^{(i)}}{dx_i} (0) & = & \boldsymbol{d_{1}^*} , &
 \end{array}\right.
\end{equation}
where $\boldsymbol{\delta_{1}^{(2)}}$ and $\boldsymbol{\delta_{1}^{(2)}}$ are defined with \eqref{delta_1}, $\boldsymbol{d_{1}^*}= \ell_0\, \mu_0 \, w_{0}^{(1)} (0) \sum\limits_{i=1}^3 h_i^2 (0)$ if $n_0>2$ and
$$
 \boldsymbol{d_{1}^*}= \ell_0\, \mu_0 \, w_{0}^{(1)} (0) \sum\limits_{i=1}^3 h_i^2 (0)
 -  \frac{1}{ \pi}\int_{\Xi^{(0)}} \varrho_0(\xi) \sum_{q=0}^{1}  \mu_{\frac{q}{2}} \, N_{\frac{1-q}{2}}(\xi)  \, d\xi \quad \text{if} \ \ n_0=2
$$
 thanks to  \eqref{first_terms_0} $(\mu_{\frac{1}{n_0}} =0,$ $w_{\frac{1}{n_0}}^{(i)} \equiv 0,$ and $N_{\frac{1}{n_0}} \equiv 0$ if $n_0 > 2).$ Then,  as in Subsection~\ref{justification}, we can uniquely determine $\{w^{(i)}_{1}\}_{i=1}^3$ and
$\mu_{1}.$

\medskip

 If  $m_0 >1,$ then $w_{\frac{1 -m_0}{n_0}}^{(i)}\equiv 0$ in \eqref{junc_probl_n_0 -m_0+1} and
as a result we get that $N_{\frac{n_0 -m_0+1}{n_0}} \equiv w^{(1)}_{\frac{n_0 -m_0+1}{n_0}}(0)$ and
\begin{equation}\label{first_trans_n_0+m_0+1}
w^{(1)}_{\frac{n_0 -m_0+1}{n_0}}(0)=w^{(2)}_{\frac{n_0 -m_0+1}{n_0}}(0)=w^{(3)}_{\frac{n_0 -m_0+1}{n_0}}(0).
\end{equation}
From \eqref{trans*k-p*alpfa_ration} and
\eqref{const_d*k-p*alpfa_ration}  with $k=2n_0 -m_0+1$ it follows that
\begin{equation}\label{second_trans_n_0+m_0+1}
    \sum\limits_{i=1}^3 h_i^2 (0) \frac{d w_{\frac{n_0 -m_0+1}{n_0}}^{(i)}}{dx_i} (0)
 =
-  \frac{1}{ \pi}\int_{\Xi^{(0)}} \varrho_0(\xi)
\big(\mu_{0} \, N_{\frac{1}{n_0}}(\xi) +
\mu_{\frac{1}{n_0}} \, N_{0} \big)  \, d\xi.
\end{equation}
Relations \eqref{first_trans_n_0+m_0+1} and  \eqref{second_trans_n_0+m_0+1} are transmission conditions for
the differential equations
\begin{equation}
    -   \frac{d^2}{d{x_i}^2}\left(w_{\frac{n_0 -m_0+1}{n_0}}^{(i)}(x_i)\right) -  \mu_{0}  \,  w_{\frac{n_0 -m_0+1}{n_0}} ^{(i)} (x_i) =
\mu_{\frac{n_0 -m_0}{n_0}}  \,  w_{\frac{1}{n_0}} ^{(i)} (x_i) +
\mu_{\frac{n_0 -m_0+1}{n_0}}  \,  w_{0} ^{(i)} (x_i), \quad x_i\in I_\varepsilon^{(i)},
\label{omega_probl__n_0+m_0+1}
\end{equation}
$ i=1,2,3.$ Thus, if $\mu_{\frac{1}{n_0}} =0,$ $w_{\frac{1}{n_0}}^{(i)} \equiv N_{\frac{1}{n_0}} \equiv 0$
$(n_0-m_0 \ge 2,$ see  \eqref{first_terms_0}),
then $\mu_{\frac{n_0 -m_0+1}{n_0}} =0,$ $w^{(i)}_{\frac{n_0 -m_0+1}{n_0}}\equiv N_{\frac{n_0 -m_0+1}{n_0}} \equiv 0.$
If the values $\mu_{\frac{1}{n_0}},$ $w_{\frac{1}{n_0}}^{(i)},$ $N_{\frac{1}{n_0}}$ don't vanish $(n_0-m_0 = 1, \ m_0 >1,$ for example $n_0=4$ and $m_0=3),$ then similarly as in Subsection~\ref{justification}  we uniquely determine $\{w^{(i)}_{\frac{2}{n_0}}\}_{i=1}^3$ and
$$
\mu_{\frac{2}{n_0}}= -  \frac{\mu_0\,  w^{(1)}_{0}(0) }{ \pi} \Big(\mu_{0} \, w^{(1)}_{\frac{1}{n_0}}(0)  +
\mu_{\frac{1}{n_0}} \, w^{(1)}_{0}(0) \Big) \int_{\Xi^{(0)}} \varrho_0(\xi)   \, d\xi.
$$

Doing the inductive step as in Subsection~\ref{justification},  we determine all  coefficients of the expansions \eqref{regul+ration}, \eqref{junc+ration} and \eqref{exp-EVl+ration}.

\subsection{Justification  $(\alpha=\frac{m_0}{n_0}$)}\label{subsec_just+alfa_rational}

With the help of the cut-off functions \eqref{cut-off-functions} we construct the series
\begin{equation}\label{asymp_expansion+alfa_rational}
    \sum\limits_{k=0}^{+\infty} \varepsilon^{\frac{k}{n_0}}
    \Big(
    \overline{u  }_{\frac{k}{n_0}} (x; \, \varepsilon, \, \gamma)
  + \overline{N}_{\frac{k}{n_0}}(x; \, \varepsilon, \, \gamma)
    \Big),
\quad x\in\Omega_\varepsilon,
\end{equation}
where
$$
\overline{u}_{\frac{k}{n_0}} (x; \, \varepsilon, \, \gamma)
 := \sum\limits_{i=1}^3 \chi_{\ell_0}^{(i)} \left(\frac{x_i}{\varepsilon^\gamma}\right)
     w_{\frac{k}{n_0}}^{(i)} (x_i),
\quad
\overline{N}_{\frac{k}{n_0}} (x; \, \varepsilon, \, \gamma)
 := \left(1 - \sum\limits_{i=1}^3 \chi_{\ell_0}^{(i)} \left(\frac{x_i}{\varepsilon^\gamma}\right) \right)
    N_{\frac{k}{n_0}} \left( \frac{x}{\varepsilon} \right),
$$
$\gamma$ is a fixed number from the interval $(\frac23, 1).$
Denote by
\begin{equation}\label{aaN+alpha+ration}
\mathfrak{U}^{(M)}(x;\varepsilon)
 :=  \sum\limits_{k=0}^{M} \varepsilon^{\frac{k}{n_0}}
    \Big(
    \overline{u  }_{\frac{k}{n_0}} (x; \, \varepsilon, \, \gamma)
  + \overline{N}_{\frac{k}{n_0}}(x; \, \varepsilon, \, \gamma)
    \Big),
\quad x\in\Omega_\varepsilon,
\end{equation}
the partial sum of \eqref{asymp_expansion+alfa_rational}, where $M\in \Bbb N, \, M \ge 3,$
and by
\begin{equation}\label{part_EV+ration}
 \mathcal{L}^{(M)}(\varepsilon) :=  \sum\limits_{k=0}^{M }\varepsilon^{\frac{k}{n_0}} \mu_{\frac{k}{n_0}}
\end{equation}
 the partial sum of \eqref{exp-EVl+ration}.  Obviously, $ \mathfrak{U}^{(M)}  \in \mathcal{H}_\varepsilon.$

We substitute $\mathfrak{U}^{(M)}$ and $\mathcal{L}^{(M)}$  into the problem~(\ref{1.1}) in place of $u^\varepsilon$ and $\lambda(\varepsilon)$ and similarly as in Subsection~\ref{subsec_just+alpfa}  find
residuals. But now we can more accurately indicate the main order of these residues, namely
\begin{equation}\label{just4+alfa_rational}
\left\| A_\varepsilon\left( \frac{\mathfrak{U}^{(M)}}{\|\mathfrak{U}^{(M)}\|_\varepsilon}\right)  - \left(\mathcal{L}^{(M)}(\varepsilon)\right)^{-1} \, \frac{\mathfrak{U}^{(M)}}{\|\mathfrak{U}^{(M)}\|_\varepsilon}  \right\|_\varepsilon
 \le C_M \left(  \varepsilon^{\frac{M+1}{n_0} - \frac{m_0}{n_0} -1} + \varepsilon^{\gamma \frac{M+1}{n_0} - \frac{\gamma}{2}} \right),
\end{equation}
where the constant $C_M$ is independent of $\varepsilon.$

Then, using Lemmas~\ref{Vishik} and \ref{Lemma_main_Convergence}, we prove, as before, the following theorem.

\begin{theorem}[$\alpha \in (0, 1) \cap \Bbb Q$ ]
  For each $n\in \Bbb N$ the $n$-th eigenvalue $\lambda_n(\varepsilon)$ of the problem \eqref{1.1}  is  decomposed into the asymptotic series
$$
\lambda_n(\varepsilon) \approx \Lambda_n + \varepsilon^{\frac{n_0-m_0}{n_0}} \mu_{\frac{n_0-m_0}{n_0},n} +
\sum\limits_{k=n_0-m_0+1}^{+\infty}\varepsilon^{\frac{k}{n_0}} \, \mu_{\frac{k}{n_0}, n}
 \quad \text{as} \quad \varepsilon \to 0,
$$
where $\Lambda_n$ is the $n$-th eigenvalue of the limit spectral problem \eqref{limitSpeProb}
and
\begin{equation}\label{mu_1-alfa_n+ration}
\mu_{\frac{n_0 -m_0}{n_0}, n} = -  \frac{\big(\Lambda_n \, W_{n}^{(1)}(0)\big)^2}{\pi}  \, \int_{\Xi^{(0)}} \varrho_0(\xi)  \, d\xi.
\end{equation}

For the corresponding eigenfunction $u_n^\varepsilon$ normalized by \eqref{normalized} and \eqref{normalized+} one can construct  the approximation function $\mathfrak{U}_n^{(M)} \in \mathcal{H}_\varepsilon$ defined with \eqref{aaN+alpha+ration} $(M\in \Bbb N, \ M \ge n_0)$ such that the asymptotic estimate
\begin{equation}\label{just9+alfa+ration}
  \left\| u_n^\varepsilon - \tau_M(\varepsilon) \,  \mathfrak{U}^{(M)}_{n} \right\|_\varepsilon \le
C_M \left(  \varepsilon^{\frac{M+1}{n_0} - \frac{m_0}{n_0}} + \varepsilon^{\gamma \frac{M+1}{n_0} +1 - \frac{\gamma}{2}} \right)
\end{equation}
 holds  for any $\varepsilon$ small enough.
Here  $\tau_M(\varepsilon) = \frac{\varepsilon}{\|\mathfrak{U}^{(M)}_{n}\|_\varepsilon}$
and it has the asymptotics \eqref{ta-m}, $\gamma$ is a fixed number from the interval $(\frac23, 1).$
\end{theorem}

From  \eqref{just9+alfa+ration} with $M= n_0$ it follows the corollary.
\begin{corollary}\label{corollary1+alfa+ration}
For the difference between the eigenfunction  $u_n^\varepsilon$ of the problem \eqref{1.1} and the partial sum
$$
\mathfrak{U}_n^{({\frac{n_0 -m_0}{n_0}})}(x) = \sum\limits_{i=1}^3 \chi_{\ell_0}^{(i)} \left(\frac{x_i}{\varepsilon^\gamma}\right)
    \left( W_{n}^{(i)}(x_i) + \varepsilon\, w_{\frac{n_0 -m_0}{n_0}, n}^{(i)} (x_i) \right)
$$
$$
    +
    \Big(1 - \sum\limits_{i=1}^3 \chi_{\ell_0}^{(i)} \left(\frac{x_i}{\varepsilon^\gamma}\right)\Big) \left(W_n^{(1)}(0) + \varepsilon^{\frac{n_0 -m_0}{n_0}} w_{\frac{n_0 -m_0}{n_0}, n}^{(1)}(0)\right), \quad x\in\Omega_\varepsilon,
$$
of \eqref{asymp_expansion+alfa_rational},  the asymptotic estimate
\begin{equation}\label{t6+alfa+ration}
\frac{1}{\sqrt{|\Omega_\varepsilon|}}\, \left\| \, u_n^\varepsilon - \tau_3(\varepsilon) \, \mathfrak{U}_n^{({\frac{n_0 -m_0}{n_0}})} \right\|_{H^1(\Omega_\varepsilon)} \leq C_1 \, \varepsilon^{\frac{n_0 - m_0+1}{n_0}}
\end{equation}
holds; \ in addition, the asymptotic estimates  \eqref{t9+alfa} and \eqref{t10+alfa}  are satisfied as well.
\end{corollary}

\section{Asymptotic approximations in the case $\alpha = 1$}\label{alfa=1}

For this case we will use  the ansatzes of series \eqref{regul}, \eqref{junc} and  \eqref{exp-EVl}  for the approximation of  an eigenfunction $u^\varepsilon$  and the corresponding eigenvalue $\lambda(\varepsilon)$ (the index $n$ is omitted)  of  the problem~\eqref{1.1}.
Recall that, as in Section~\ref{Sec_4}, we assume that $h_i \equiv h_i (0), \ i=1, 2, 3.$ Therefore, all coefficients $\{u_k\}_{k=2}^\infty$ in the regular part \eqref{regul}  vanish.

Substituting \eqref{regul} and  \eqref{exp-EVl}   into the differential equation and boundary conditions of the problem~\eqref{1.1}, collecting the coefficients at the same power of $\varepsilon$,  we get the following relations:
\begin{eqnarray}
    -   \dfrac{d^2 w_{k }^{(i)}(x_i)}{d{x_i}^2} - \mu_{0}  \,  w_{k} ^{(i)} (x_i) &=&   \sum_{m=1}^{k}
\mu_{m}  \,  w_{k-m} ^{(i)} (x_i) , \quad x_i\in I_\varepsilon^{(i)}, \label{w_probl_k}
\\
w_{k}^{(i)}(\ell_i)&=&0, \quad i=1, 2, 3; \label{BC_w_1_k}
\end{eqnarray}
doing the same with \eqref{junc} and  \eqref{exp-EVl}, we obtain
\begin{equation}\label{junc_probl_k_1}
 \left\{\begin{array}{rcll}
  -\Delta_{\xi}{N_k}(\xi) & = & \sum_{m=0}^{k-2} \mu_m \, N_{k-2-m}(\xi),      &
   \quad \xi\in\Xi^{(i)}, \quad i=1,2,3,
\\[2mm]
  -\Delta_{\xi}{N_k}(\xi) & = &  \varrho_0(\xi) \sum_{m=0}^{k-1} \mu_m \, N_{k-1-m}(\xi),      &
   \quad \xi\in\Xi^{(0)},
\\[2mm]
   \partial_{{\boldsymbol \nu}_\xi}{N_k}(\xi) & = &
   0,                               &
   \quad \xi\in \partial \Xi
\\[2mm]
   N_k(\xi)                                                               & \sim &
   w^{(i)}_{k}(0) + \Psi^{(i)}_{k}(\xi),                                    &
   \quad \xi_i \to +\infty, \ \ \xi  \in \Xi^{(i)}, \quad i=1,2,3,
  \end{array}\right.
\end{equation}
where $k\in \Bbb N_0,$ $\{\Psi^{(i)}_{k}\}$ are defined in \eqref{Psi_k}.

As in Subsection~\ref{sub_3*4}, we deduce  that $N_0 \equiv  \widetilde{N}_0 \equiv w_0^{(1)} (0)$
and the first transmission condition \eqref{trans0} for $\{w_0^{(i)}\}_{i=1}^3.$
Writing down the solvability condition for the corresponding problem for $\widetilde{N}_{k},$ similarly as in Subsection~\ref{inner_asymp} we get
\begin{equation}\label{transmisiont1+1}
    \sum\limits_{i=1}^3  h_i^2 (0) \frac{d w_{k-1}^{(i)}}{dx_i} (0)
 =  \boldsymbol{{d}^\star_{k-1}},
\end{equation}
where
\begin{equation}\label{const_d_*+1}
\boldsymbol{d_k^\star}
 =  \boldsymbol{\widehat{d}_k}
    \, - \, \mu_k \, \frac{ w_0^{(1)}(0)}{\pi} \int_{\Xi^{(0)}} \varrho_0(\xi)  \, d\xi,
\quad k\in\Bbb N;
\end{equation}
\begin{multline*}
\boldsymbol{\widehat{d}_k}
= \sum\limits_{i=1}^3 h_i^2 (0) \, \sum\limits_{j=1}^k \frac{\ell_0^j}{j!} \sum_{m=0}^{k-j } \mu_m \,
    \dfrac{d^{j-1} w_{k-j-m}^{(i)}}{d x_i^{j-1}} (0)
\\
-  \frac{1}{\pi}  \sum_{m=0}^{k-1} \mu_m \, \sum\limits_{i=1}^3 \int_{\Xi^{(i)}}\Big( N_{k-1-m}(\xi) - G^{(i)}_{k-1-m}(\xi_i)\Big)  \,  d\xi
     -  \frac{1}{\pi} \int_{\Xi^{(0)}} \varrho_0(\xi) \sum_{m=0}^{k-1} \mu_m \, N_{k-m}(\xi)  \, d\xi.
\end{multline*}

Thus, in this case the limit spectral problem looks as follows:
\begin{equation}\label{limitSpeProb+1}
 \left\{\begin{array}{rclr}
  -  \dfrac{d^2 w_0^{(i)}}{d{x_i^2}}(x_i) & = &
    \mu_0\, \, w_0^{(i)}(x_i), &                                x_i\in I_i, \quad i=1,2,3,
 \\[3mm]
    w_0^{(i)}(\ell_i) & = &
    0, &           i=1,2,3,
 \\[2mm]
    w_0^{(1)} (0) \ \, = \ \, w_0^{(2)} (0) & = &
    w_0^{(3)} (0), &
  \\[2mm]
    \sum_{i=1}^3  h_i^2 (0) \dfrac{dw_{0}^{(i)}}{dx_i} (0) & = &
    -    \, \displaystyle{ \frac{\mu_0 \, w^{(1)}_0(0)}{\pi} \, \int_{\Xi^{(0)}} \varrho_0(\xi) \, d\xi} . &
 \end{array}\right.
\end{equation}
Now we see a spectral parameter $\mu_0$ both in the differential equations, and in the second  Kirchhoff condition of the problem \eqref{limitSpeProb+1}.

\begin{definition}\label{def5_1}
A function $\widetilde{w}\in \mathcal{ H}_0 \setminus \{0\}$  is called  an eigenfunction  corresponding to the eigenvalue
$\mu$ of the  problem \eqref{limitSpeProb+1} if
\begin{equation}\label{limit_identity+1}
\big\langle \widetilde{w}, \widetilde{\phi} \big\rangle_0 = \mu
\left( \big( \widetilde{w}, \widetilde{\phi} \big)_{\mathcal{ V}_0} +  \displaystyle{ \frac{w^{(1)}(0) \, \phi^{(1)}(0)}{\pi} \, \int_{\Xi^{(0)}} \varrho_0(\xi) \, d\xi} \right)
\qquad \forall \, \widetilde{\phi}  \in \mathcal{ H}_0,
  \end{equation}
where the scalar products $\langle \cdot, \cdot \rangle_0, \ \left( \cdot, \cdot \right)_{\mathcal{ V}_0}$ in the spaces $\mathcal{ H}_0, \
\mathcal{ V}_0,$ respectively, are defined in Subsection~\ref{sub_3*4}.
\end{definition}

Let us define an operator $ A_1: \mathcal{H}_{0}\mapsto
\mathcal{H}_{0}$  by the  equality
\begin{equation}\label{limit-operator2}
\big\langle A_{1} \widetilde{u} , \widetilde{v}\big\rangle_0 =
\left( \big( \widetilde{u}, \widetilde{v} \big)_{\mathcal{ V}_0} +  \displaystyle{ \frac{u^{(1)}(0) \, v^{(1)}(0)}{\pi} \, \int_{\Xi^{(0)}} \varrho_0(\xi) \, d\xi} \right)
 \quad \forall \ \widetilde{u}, \widetilde{v} \in \mathcal{ H}_{0}.
\end{equation}
It is easy to verify (see e.g. \cite{Oleinik_1988}),  that  the operator $A_1$ is self-adjoint, positive, and compact.
Then the problem (\ref{limitSpeProb+1}) is equivalent to the spectral problem
\ $A_{1} \widetilde{w} = \mu^{-1} \,\widetilde{w}$ \  in $\mathcal{ H}_{0}.$
Similarly as in Subsection~\ref{sub_3*4} we prove that each eigenvalue of the problem \eqref{limitSpeProb+1} is simple.
Therefore, the eigenvalues of the problem (\ref{limitSpeProb+1})  form the sequence
\begin{equation}\label{EV_lim_prob_alfa=1}
0  <  \Lambda_1   < \Lambda_2 < \ldots < \Lambda_n  < \dots \to + \infty \quad {\rm as}
\quad   n \to +\infty,
\end{equation}
and the corresponding eigenfunctions $ \{\widetilde{W}_n \}_{n \in \Bbb N} \subset  \mathcal{ H}_{0}$ can be uniquely
ortho\-nor\-ma\-liz\-ed with relations \eqref{EF_lim_prob} and \eqref{EF_addion_lim_prob}.

Let $\mu_0$ be an eigenvalue of the problem \eqref{limitSpeProb+1}, and $\widetilde{w}_0=\{w_0^{(i)}\}_{i=1}^3$ is the corresponding eigenfunction  satisfying  conditions \eqref{EF_lim_prob} and \eqref{EF_addion_lim_prob}.
Then, there exists a unique  weak solution $\widehat{N}_1$ to the problem
\begin{equation}\label{new_junc_probl_alfa=1}
 \left\{\begin{array}{rcll}
  -\Delta_{\xi}\widetilde{N}_{1}(\xi) & = & 0,   &  \quad \xi\in\Xi^{(i)}, \quad i=1,2,3,
 \\[2mm]
-\Delta_{\xi}{\widetilde{N}_{1}}(\xi) & = & \varrho_0(\xi)  \mu_0 \, w_0^{(1)}(0),   &  \quad \xi\in\Xi^{(0)},
\\[2mm]
   \partial_{\nu_\xi}\widetilde{N}_{1}(\xi) & = &
   0,                               &
   \quad \xi\in \partial \Xi,
\\[2mm]
  \widetilde{N}_{1}(\xi)& \rightarrow  & w^{(i)}_{1}(0),   \quad \xi_i \to +\infty, \ \ {\xi}_i \in \Xi^{(i)}, \quad i=1,2,3,
 \end{array}\right.
\end{equation}
with the asymptotics \eqref{inner_asympt_general},  but now
\begin{equation}\label{delta_1+1}
\boldsymbol{\widehat{\delta}_{1}^{(i)}}:=
  \sum\limits_{j=1}^3 \dfrac{d w_0^{(j)}}{dx_j}(0) \int_{\Xi^{(j)}} \mathfrak{N}_i(\xi)
        \Big( \xi_j \chi_j^{\prime\prime}(\xi_j) + 2 \chi_j^{\prime}(\xi_j) \Big) \, d\xi + \mu_0 \, w_0^{(1)}(0) \int_{\Xi^{(0)}}\varrho_0(\xi)  \, d\xi, \quad i=2,3;
\end{equation}
please compare with $ \boldsymbol{{\delta}_{1}^{(i)}}$ in \eqref{delta_1}  for $\alpha=0.$

As follows from Subsection~\ref{sub_3*4}, to satisfy relations
\begin{equation}\label{relations_infinity}
   N_{1}(\xi)\sim  w^{(i)}_{1}(0) + \xi_i \dfrac{dw_{0}^{(i)}}{dx}(0),   \quad \xi_i \to +\infty, \ \ {\xi}_i \in \Xi^{(i)}, \quad i=1,2,3,
\end{equation}
in the problem \eqref{junc_probl_k_1} with $k=1,$
we should solve the problem
\begin{equation}\label{w_1_probl_1}
 \left\{\begin{array}{rclr}
-   \dfrac{d^2 w_{1 }^{(i)}(x_i)}{d{x_i}^2} - \mu_{0}  \,  w_{1} ^{(i)} (x_i) &=&  \mu_{1}  \,  w_{0} ^{(i)} (x_i),
 &   x_i\in I_i, \  i=1,2,3,
 \\[3mm]
    w_{1}^{(i)}(\ell_i) & = &
    0, \quad i=1,2,3, &
 \\[3mm]
    w_{1}^{(1)} (0) \ \, = \ \, w_{1}^{(2)} (0) - \boldsymbol{\widehat{\delta}_{1}^{(2)}} & = &
                                   w_{1}^{(3)} (0) - \boldsymbol{\widehat{\delta}_{1}^{(3)}}, &
 \\[3mm]
    \sum\limits_{i=1}^3 h_i^2 (0) \,\dfrac{d w_{1}^{(i)}}{dx_i} (0) & = &
   \boldsymbol{d_{1}^\star},&
 \end{array}\right.
\end{equation}
 where
\begin{equation}\label{const_d_*1+1}
\boldsymbol{d_1^\star} = \mu_0 \, w^{(1)}_0(0) \sum\limits_{i=1}^3 \ell_0 \, h_i^2 (0)
-  \frac{\mu_0}{\pi} \int_{\Xi^{(0)}} \varrho_0(\xi) \, N_1(\xi)\, d\xi -  \frac{\mu_1}{\pi}\, w^{(1)}_0(0) \int_{\Xi^{(0)}} \varrho_0(\xi) \, d\xi;
\end{equation}
please compare with $\boldsymbol{d_1^*}$ in \eqref{const_d_*1}  for $\alpha=0.$

With the help of the substitutions \eqref{sunstitutions},  the problem \eqref{w_1_probl_1} is  reduced to the corresponding problem~
\eqref{tilde-omega_probl*}, and in the same way as in Subsection~\ref{sub_3*4} we define $\mu_1$ from \eqref{Fredholm}, namely,
\begin{equation}\label{mu_1+1}
 \mu_{1}   =  \left(\mu_0 \, w_{0}^{(1)}(0) \right)^2   \sum\limits_{i=1}^3  \ell_0 h_i^2 (0)
     -  \frac{\mu_0^2\, w_{0}^{(1)}(0) }{\pi} \int\limits_{\Xi^{(0)}} \varrho_0(\xi) N_1(\xi) \, d\xi
 + \left(\dfrac{\boldsymbol{\widehat{\delta}_{1}^{(2)}} h^2_2(0)}{\ell_2}
 + \dfrac{ \boldsymbol{\widehat{\delta}_{1}^{(3)}} h^2_3(0)}{\ell_3}\right) \,   \mu_0 w_{0}^{(1)}(0).
\end{equation}
To compare with $\mu_1$ in \eqref{mu_1}  for $\alpha=0,$ we re-write \eqref{mu_1+1} in the following way
\begin{multline}\label{mu_1+1_re}
 \mu_{1}   =  \left(\mu_0 \, w_{0}^{(1)}(0) \right)^2  \left( \sum\limits_{i=1}^3  \ell_0 h_i^2 (0)
+ \Big(\frac{h^2_2(0)}{\ell_2} + \frac{h^2_3(0)}{\ell_3}\Big) \int\limits_{\Xi^{(0)}} \varrho_0(\xi)\, d\xi
\right)
     -  \frac{\mu_0^2\, w_{0}^{(1)}(0) }{\pi} \int\limits_{\Xi^{(0)}} \varrho_0(\xi) N_1(\xi) \, d\xi
\\
+ \left(\dfrac{\boldsymbol{{\delta}_{1}^{(2)}} h^2_2(0)}{\ell_2}
 + \dfrac{ \boldsymbol{{\delta}_{1}^{(3)}} h^2_3(0)}{\ell_3}\right) \,   \mu_0 w_{0}^{(1)}(0),
\end{multline}
where $\boldsymbol{{\delta}_{1}^{(2)}}$ and $\boldsymbol{{\delta}_{1}^{(3)}}$ are defined in \eqref{delta_1}.

 Thus, there exists a unique weak solution to the corresponding problem (\ref{tilde-omega_probl*}) with $k=1$
such that
\begin{equation}\label{ortho_solution}
  \langle \widetilde{\phi}_{1}, \widetilde{w}_0 \rangle_0 =0.
\end{equation}
Then, with the help of \eqref{sunstitutions} we find the
unique solution $\widetilde{w}_1=\{w_1^{(i)}\}_{i=1}^3$ to the problem \eqref{w_1_probl_1}. In addition,
we uniquely determine the coefficient $N_1.$

To find the coefficients $\widetilde{w}_k=\{w_k^{(i)}\}_{i=1}^3$ and $\mu_k$ for any $k\in \Bbb N, \ k\geq 2,$ we should consider
a problem consisting with the differential equations \eqref{w_probl_k}, the boundary conditions \eqref{BC_w_1_k}, and
the Kirchhoff transmission conditions
\begin{gather}\label{Kir_cond_k_1}
   w_{k}^{(1)} (0)  =  w_{k}^{(2)} (0) - \boldsymbol{\widehat{\delta}_{k}^{(2)}}  =  w_{k}^{(3)} (0) - \boldsymbol{\widehat{\delta}_{k}^{(3)}}, \\
    \sum\limits_{i=1}^3 h_i^2 (0) \,\dfrac{d w_{k}^{(i)}}{dx_i} (0)  =   \boldsymbol{\widehat{d}_k}
    \, - \, \mu_k \, \frac{ w_0^{(1)}(0)}{\pi} \int_{\Xi^{(0)}} \varrho_0(\xi)  \, d\xi. \label{Kir_cond_k_2}
 \end{gather}
Recall that the solution $\widehat{N}_k$ with the asymptotics
\begin{equation}\label{inner_asympt_1}
\widehat{N}_k(\xi)=\left\{
\begin{array}{rl}
    \mathcal{O}(\exp( - \gamma_1 \xi_1)) & \mbox{as} \ \ \xi_1\to+\infty, \ \ \xi  \in \Xi^{(1)},
\\[2mm]
    \boldsymbol{\widehat{\delta}_k^{(2)}}
 +  \mathcal{ O}(\exp( - \gamma_2 \xi_2)) & \mbox{as} \ \ \xi_2\to+\infty, \ \ \xi  \in \Xi^{(2)},
\\[2mm]
   \boldsymbol{\widehat{\delta}_k^{(3)}}
 +  \mathcal{ O}(\exp( - \gamma_3 \xi_3)) & \mbox{as} \ \ \xi_3\to+\infty, \ \ \xi  \in \Xi^{(3)},
\end{array}
\right.
\end{equation}
is uniquely determined, since the solvability condition for the corresponding problem for $\widetilde{N}_k$ is
the second Kirchhoff condition for $\widetilde{w}_{k-1}.$

As in Subsection~\ref{sub_3*4} we reduce the problem \eqref{w_probl_k}, \eqref{BC_w_1_k}, \eqref{Kir_cond_k_1}, \eqref{Kir_cond_k_2} to the corresponding problem
\eqref{tilde-omega_probl*}, and from its solvability condition  we define
\begin{equation}\label{Fredholm+1}
\mu_k= \mu_0  \sum_{i=1}^3 \int_0^{\ell_i} \Phi_k^{(i)}(x_i; \mu_1,\ldots,\mu_{k-1}) \, w_{0}^{(i)}(x_i) \, dx_i
+ \mu_0  w_{0}^{(1)}(0) \left( \boldsymbol{\widehat{d}_k} +
+ \frac{\boldsymbol{\widehat{\delta}_k^{(2)}} h^2_2(0)}{\ell_2}
 + \frac{\boldsymbol{\widehat{\delta}_k^{(3)}} h^2_3(0)}{\ell_3} \right),
\end{equation}
where
\begin{gather*}
{\Phi}_k^{(1)}(x_1; \mu_1,\ldots,\mu_{k-1}) =  \displaystyle{\sum_{m=1}^{k-1}} \mu_m \, h_1^2(0) \,  w_{k-m}^{(1)} (x_1),
 \\
  {\Phi}_k^{(i)}(x_i; \mu_1,\ldots,\mu_{k-1}) = \displaystyle{\sum_{m=1}^{k-1}} \mu_m \, h_i^2(0) \,  w_{k-m}^{(i)} (x_i)  +  \mu_0 \, h_i^2(0) \, \boldsymbol{\delta_k^{(i)}} \, \frac{\ell_i - x_i}{\ell_i}, \ \ i=2, 3.
 \end{gather*}
Then $\widetilde{w}_{k}$ is uniquely determined, as and $\widetilde{N}_k= w_{k}^{(1)} (0) +  \widehat{N}_k$ and $N_k$ (see \eqref{new-solution}).

\subsection{Justification  $(\alpha=1)$}\label{subsec_just+alfa=1}

With the help of the cut-off functions \eqref{cut-off-functions} we construct the series
\begin{equation}\label{asymp_expansion_alfa=1}
    \sum\limits_{k=0}^{+\infty} \varepsilon^{k}
    \Big(
    \overline{u  }_{k} (x; \, \varepsilon, \, \gamma)
  + \overline{N  }_{k} (x; \, \varepsilon, \, \gamma)
    \Big),
\quad x\in\Omega_\varepsilon,
\end{equation}
where
$$
\overline{u}_{k} (x; \, \varepsilon, \, \gamma)
 := \sum\limits_{i=1}^3 \chi_{\ell_0}^{(i)} \left(\frac{x_i}{\varepsilon^\gamma}\right)
     w_{k}^{(i)} (x_i), \quad
\overline{N}_{k} (x; \, \varepsilon, \, \gamma)
 := \left(1 - \sum\limits_{i=1}^3 \chi_{\ell_0}^{(i)} \left(\frac{x_i}{\varepsilon^\gamma}\right) \right)
    N_{k} \left( \frac{x}{\varepsilon} \right),
$$
$\gamma$ is a fixed number from the interval $(\frac23, 1),$

Denote by
\begin{equation}\label{aaN+alpha=1}
\mathfrak{U}^{(M)}(x;\varepsilon)
 :=  \sum\limits_{k=0}^{M} \varepsilon^k
    \Big(
    \overline{u  }_k(x; \, \varepsilon, \, \gamma)
  + \overline{N}_k(x; \, \varepsilon, \, \gamma)
    \Big),
\quad x\in\Omega_\varepsilon,
\end{equation}
the partial sum of \eqref{asymp_expansion_alfa=1}, where $M\in \Bbb N, \, M \ge 3,$
and by
\begin{equation}\label{part_EV_alfa=1}
 \mathcal{L}^{(M)}(\varepsilon) :=  \sum\limits_{k=0}^{M }\varepsilon^k \mu_k
\end{equation}
 the partial sum of \eqref{exp-EVl}.  Obviously, $ \mathfrak{U}^{(M)}  \in \mathcal{H}_\varepsilon.$

We substitute $\mathfrak{U}^{(M)}$ and $\mathcal{L}^{(M)}$  into the problem~(\ref{1.1}) instead  of $u^\varepsilon$ and $\lambda(\varepsilon),$ respectively,  and similarly as in Subsection~\ref{subsec_just+alpfa}  find that
\begin{equation}\label{just4+alfa=1}
\left\| A_\varepsilon\left( \frac{\mathfrak{U}^{(M)}}{\|\mathfrak{U}^{(M)}\|_\varepsilon}\right)  - \left(\mathcal{L}^{(M)}(\varepsilon)\right)^{-1} \, \frac{\mathfrak{U}^{(M)}}{\|\mathfrak{U}^{(M)}\|_\varepsilon}  \right\|_\varepsilon
 \le C_M \left(  \varepsilon^{M-1} + \varepsilon^{\gamma M + \frac{\gamma}{2}} \right),
\end{equation}
where the constant $C_M$ is independent of $\varepsilon.$

Further, in order to apply the justification scheme from Subsection~\ref{subsec_just+alpfa}, we need the following lemmas.

\begin{lemma}\label{lemma_5_1} For each function $u \in \mathcal{H}_\varepsilon$ the following inequality holds for
sufficiently small values of the parameter $\varepsilon:$
\begin{equation}\label{estimate+1}
\left|\varepsilon^{-3} \int_{\Omega^{(0)}_\varepsilon} \varrho_0\big(\frac{x}{\varepsilon}\big) \, u(x) \,dx  -
\varepsilon^{-3} \langle \varrho_0 \rangle \,\int_{\Omega^{(0)}_\varepsilon} u(x) \,dx \right| \leq C\, \varepsilon^{-\frac{1}{2}} \, \|u\|_{\varepsilon},
\end{equation}
where the constant $C$ is independent of $u$ and $\varepsilon,$ \
$$
\langle \varrho_0 \rangle_0 := \frac{1}{|\Xi^{(0)}|} \int_{\Xi^{(0)}} \varrho_0(\xi)\, d\xi
$$
is the middle value of $\varrho_0$ over $\Xi^{(0)},$ $|\Xi^{(0)}|$  is the Lebesgue measure of the domain $\Xi^{(0)}.$
\end{lemma}
\begin{proof} We use the approach of Lemma 2.1 in \cite{Mel_Pop_2012}. Obviously, the Neumann problem
\begin{equation}\label{T_problem}
 \left\{\begin{array}{rcll}
  -\Delta_{\xi}{T}(\xi) & = & \varrho_0(\xi) - \langle \varrho_0 \rangle,  &
   \quad \xi\in\Xi^{(0)},
\\[2mm]
   \partial_{{\boldsymbol \nu}_\xi}{T}(\xi) & = &
   0,                               &  \quad \xi\in \partial\Xi^{(0)},
\\[2mm]
\langle T \rangle_0 &=& 0, &
 \end{array}\right.
\end{equation}
has the unique weak solution $T \in H^1(\Xi^{(0)})$ and
\begin{equation}\label{est_T}
  \int_{\Omega^{(0)}_{\varepsilon}} |\nabla_\xi T(\xi)\big|_{\xi=\frac{x}{\varepsilon}} \, dx \le C_1 \, \varepsilon^3.
\end{equation}

We can express the relations in \eqref{T_problem}
in terms of the $x$-variables:
 \begin{equation}\label{T_problem_x}
 \left\{\begin{array}{rcll}
  - \varepsilon \sum_{i=1}^3 \frac{\partial}{\partial x_i} \left( \frac{\partial T}{\partial\xi_i} \big|_{\xi= \frac{x}{\varepsilon}}  \right) & = & \varrho_0\big(\frac{x}{\varepsilon}\big) - \langle \varrho_0 \rangle_0,  &
   \quad x \in \Omega^{(0)}_{\varepsilon},
\\[2mm]
\sum_{i=1}^3 \frac{\partial T}{\partial\xi_i} \big|_{\xi= \frac{x}{\varepsilon}} \, \nu_i \big(\frac{x}{\varepsilon}\big)  & = &
   0,                               &  \quad x \in \partial\Omega^{(0)}_{\varepsilon}.
 \end{array}\right.
\end{equation}
Multiplying the first equation in \eqref{T_problem_x} by an arbitrary function $u \in \mathcal{H}_\varepsilon$
and integrating by parts, we get
$$
\varepsilon \int_{\Omega^{(0)}_{\varepsilon}} \nabla_\xi T(\xi)\big|_{\xi=\frac{x}{\varepsilon}} \cdot \nabla_x u \, dx = \int_{\Omega^{(0)}_{\varepsilon}} \Big(\varrho_0\big(\frac{x}{\varepsilon}\big) - \langle \varrho_0 \rangle_0\Big)\, u(x)\, dx,
$$
wherefrom, in view of \eqref{est_T}, it follows  \eqref{estimate+1}.
\end{proof}

\begin{lemma}[$\alpha =1$]\label{Lemma_Convergence_alfa=1} For each $n\in \Bbb N$
\begin{equation}\label{convergence=1}
  \lambda_n(\varepsilon) \rightarrow \Lambda_n \quad \text{as}  \ \ \varepsilon \to 0,
 \end{equation}
where $\lambda_n(\varepsilon)$ is the $n$-th eigenvalue of the problem \eqref{1.1} from the sequence \eqref{1.2}, \
$\Lambda_n$ is the $n$-th eigenvalue of the limit spectral problem \eqref{limitSpeProb+1}  from the sequence  \eqref{EV_lim_prob_alfa=1}.
 \end{lemma}

\begin{proof} We should repeat the proof of Lemma~\ref{Lemma_main_Convergence}, but only one convergence should be specified,
namely,  the limit of the last summand in the right-hand side of \eqref{identity+graph}. Using  \eqref{just4+alfa=1},
\eqref{normalized}  and the integral mean-value theorem, we have
\begin{equation}\label{limit+1}
  \varepsilon^{-3} c_0 \int_{\Omega^{(0)}_\varepsilon} \varrho_0\big(\frac{x}{\varepsilon}\big) \, u^\varepsilon(x) \,dx \longrightarrow
c_0 \, v^{(1)}(0)\, \int_{\Xi^{(0)}} \varrho_0(\xi) \, d\xi \quad \text{as} \ \ \varepsilon \to 0.
\end{equation}
The same argumentations give us
\begin{equation}\label{limit+2}
  \varepsilon^{-3}  \int_{\Omega^{(0)}_\varepsilon} \varrho_0\big(\frac{x}{\varepsilon}\big) \, \big(u^\varepsilon(x)\big)^2 \,dx \longrightarrow
 \big(v^{(1)}(0)\big)^2 \, \int_{\Xi^{(0)}} \varrho_0(\xi) \, d\xi \quad \text{as} \ \ \varepsilon \to 0.
\end{equation}
\end{proof}

Thus, the following statement holds.

\begin{theorem}[$\alpha =1$ ]
  For each $n\in \Bbb N$ the $n$-th eigenvalue $\lambda_n(\varepsilon)$ of the problem \eqref{1.1}  is  decomposed into the asymptotic series
$$
\lambda_n(\varepsilon) \approx \Lambda_n + \varepsilon \, \mu_{1,n} +
\sum\limits_{k=2}^{+\infty}\varepsilon^k \, \mu_{k, n}
 \quad \text{as} \quad \varepsilon \to 0,
$$
where $\Lambda_n$ is the $n$-th eigenvalue of the limit spectral problem
(\ref{limitSpeProb+1})
and $\mu_{1, n}$ is determined by the corresponding formula \eqref{mu_1+1_re}.

For the corresponding eigenfunction $u_n^\varepsilon$ normalized by \eqref{normalized} and \eqref{normalized+} one can construct  the approximation function $\mathfrak{U}_n^{(M)} \in \mathcal{H}_\varepsilon$ defined with \eqref{aaN+alpha=1} $(M\in \Bbb N, \ M \ge 3)$ such that the asymptotic estimate
\begin{equation}\label{just9+alfa=1}
  \left\| u_n^\varepsilon - \tau_M(\varepsilon) \,  \mathfrak{U}^{(M)}_{n} \right\|_\varepsilon \le
C_M \left(  \varepsilon^{M} + \varepsilon^{\gamma M + 1+ \frac{\gamma}{2}} \right)
\end{equation}
 holds  for any $\varepsilon$ small enough.  Here  $\tau_M(\varepsilon) = \frac{\varepsilon}{\|\mathfrak{U}^{(M)}_{n}\|_\varepsilon}$
and it has the asymptotics \eqref{ta-m}, $\gamma$ is a fixed number from the interval $(\frac23, 1).$
\end{theorem}

In this case, the analogues of  Corollaries~\ref{corollary1} and \ref{corollary2} hold with the corresponding replacements of coefficients.

\section*{Acknowledgements}
This research was begun at the University of Stuttgart under support of the Alexander
von Humboldt Foundation in the summer of 2019. The author is very grateful to
Prof. Christian Rohde for the hospitality and wonderful working conditions. 

The author is a member (and his scientific activity is partially supported by) of the research project
``Development of analytical-geometric, asymptotic and qualitative methods for investigation of invariant differential equation sets''
(number 19BF038-02) of the Taras Shevchenko National University of Kyiv.



\begin{thebibliography}{199}
\bibitem{Boyd}
J.P. Boyd,  The Devil's Invention: Asymptotic, superasymptotic and hyperasymptotic series. \emph{Acta Applicandae Mathematicae}, 
{\bf 56} (1999) 1-98. 

\bibitem{Bun_Gau_Leo_2019}
R. Bunoiu,  A. Gaudiello and  A. Leopardi, Asymptotic analysis of a Bingham fluid in a thin T-like shaped structure,  \emph{J. Math. Pures  Appl.} {\bf 123} (2019)  148-166.

\bibitem{IZV}
G.A. Chechkin,  Asymptotic expansion of eigenvalues and eigenfunctions of an elliptic operator in a domain with many
``light'' concentrated masses situated on the boundary. Two-dimensional case. \emph{Izvestia: Mathematics} {\bf 69} (2005) 805-846.

\bibitem{MMO}
G.A. Chechkin, M.E. P\'erez and  E.I. Yablokova, Non-periodic boundary homogenization and “light” concentrated masses,  \emph{Indiana Univ. Math. J.} {\bf 54} (2005) 321-348.

\bibitem{Mel_Che_2014}
G.A. Chechkin and T.A. Mel'nyk, Spatial-skin effect for eigenvibrations of a thick cascade junction with ``heavy'' concentrated masses.
{\em Math. Meth. Appl. Sci.} {\bf 37} (2014)  56-74. 

\bibitem{Antonio_Perez-2019}
A. Gaudiello, D. G\'omez and  M.-E. P\'erez-Martinez, Asymptotic analysis of the high frequencies for the Laplace operator in a thin T-like shaped structure, \emph{J. Math. Pures  Appl.}  (in press) https://doi.org/10.1016/j.matpur.2019.06.005

\bibitem{GauPanPia2016}
A. Gaudiello,  G. Panasenko and A. Piatnitski,  Asymptotic analysis and domain decomposition
for a biharmonic problem in a thin multi-structure, \emph{Communications in Contemporary Mathematics.} {\bf 18} (2016)
1550057. 


\bibitem{Gol_Naz_Ole_Sob_1989}
Yu.D. Golovatyi, S.A. Nazarov, O.A. Oleinik  and T.S. Soboleva, Eigenoscillations of a string with an additional mass,
\emph{Sib Math J.} {\bf 29} (1988) 744-760. 

\bibitem{Gol_1992}
Yu. D. Golovatyi, Spectral properties of oscillatory systems with added masses, \emph{Tr. Mosk. Mat. Obs.} {\bf 54} (1992) 29-72.

\bibitem{Gol_Gra_07}
Yu.  Golovatyi and H. Hrabchak,
Asymptotics of the spectrum of the Sturm-Liouville problem  on a geometrical graph  with perturbed density in
neighborhoods of vertices,  \emph{Visnyk Lviv Univ. Ser. Mech-Math.} {\bf 67} (2007)  66-83 (in Ukrainian).

\bibitem{Klev_2019}
A.V. Klevtsovskiy, Asymptotic expansion of the solution of a linear parabolic boundary-value problem in a thin starlike joint,  \emph{J. Math. Sci.} {\bf 238} (2019) 271-291.

\bibitem{Mel_Klev_AA-2016}
A.V. Klevtsovskiy and T.A. Mel'nyk,  Asymptotic expansion for the solution to a boundary-value problem in a thin cascade domain with a local joint, \emph{Asymptotic Analysis}, \textbf{97} (2016) 265-290.


\bibitem{Mel_Klev_M2AS-2018}
A.V . Klevtsovskiy and T.A. Mel’nyk, Asymptotic approximation for the solution to a semilinear parabolic problem in a thin star-shaped junction, \emph{Math. Meth. Appl. Sci.} \textbf{41} (2018) 159-191. 

\bibitem{Mel_Klev_AA-2019}
A.V. Klevtsovskiy and T.A. Mel'nyk,  Influence of the node on the asymptotic behaviour of the solution
to a semilinear parabolic problem in a thin graph-like junction, \emph{Asymptotic Analysis}, {\bf 113}  (2019) 87-121.

\bibitem{Ko-Ol}
V. A. Kondratiev and O. A. Oleinik, Boundary-value problems for partial differential equations in non-smooth domains,
\emph{Russian Mathematical Surveys}, \textbf{38}:2 (1983) 1-86.

\bibitem{Kuchment2002}
 P. Kuchment, Graph models for waves in thin structures, \emph{Waves in Random Media}, {\bf 12}:4 (2002) 1-24.

\bibitem{La-Pa}
E.M. Landis and G.P. Panasenko, A variant of a theorem of Phragmen-Lindelof type for elliptic equations with coefficients that are periodic in all variables but one, \emph{Topics in modern mathematics, Petrovskii Semin.}  {\bf 5}  (1985) 133-172. 

\bibitem{M3AS1993}
 M. Lobo and  E. P\'erez,  On vibrations of a body with many concentrated
masses near the boundary, \emph{Math. Mod. Meth. Appl. Sci.} {\bf 3} (1993) 249-273.

\bibitem{CRAS1999}
M. Lobo and  E. P\'erez,   A skin effect for systems with many concentrated
masses. {\it C.R. Acad. Sci. Paris. S\'erie II b.} {\bf 327} (1999) 771-776.

\bibitem{Perez_Review}
M. Lobo and  E. P\'erez,   Local problems for vibrating systems with concentrated masses:
a review. {\it C.R. Acad. Sci. Paris. Mecanique,} {\bf 331} (2003) 303-317.

\bibitem{ZAA99}
T.A. Mel'nyk, Homogenization of the Poisson equation in a thick periodic junction,
 Zeitschrift f\"ur Analysis und ihre Anwendungen, \textbf{18} (1999) 953--975.

\bibitem{M3AS}
T.A. Mel'nyk,  Vibrations of a thick periodic junction with concentrated masses, \emph{Math. Mod. Meth. Appl. Sci.} {\bf 11} (2001) 1001-1029.

\bibitem{M-MMAS-2008}
T.A. Mel'nyk, Homogenezation of a boundary-value problem with a nonlinear boundary condition in a thick junction of type 3:2:1,  \emph{Math. Meth. Appl. Sci.} {\bf 31} (2008) 1005-1027.

\bibitem{Mel_Pop_2012}
T.A. Mel'nyk and A.V. Popov,
Asymptotic analysis of boundary-value and spectral problems in thin perforated regions with rapidly changing thickness and different limiting dimensions,  \emph{Sbornik: Mathematics} {\bf 203}  (2012) 1169-1195.

\bibitem{Miller}
P.D. Miller,  \emph{Applied Asymptotic Analysis.  Graduate Studies in Mathematics}, {\bf 75}, American Mathematical Society, (Providence, RI, 2006).

\bibitem{Naz96}
S.A. Nazarov, Junctions of singularly degenerating domains with different limit dimensions,
 J. Math. Sci., \textbf{80}:6 (1996) 1989--2034.

\bibitem{Na-Pla}
S. A. Nazarov and B. A. Plamenevskii,
\emph{Elliptic problems in domains with piecewise smooth boundaries,}
(Walter de Gruyter, Berlin, 1994).

\bibitem{Naz-Ruo-Uus-2016}
S.A. Nazarov, K. Ruotsalainen and P. Uusitalo, Multifarious transmission conditions in the graph models of carbon nano-structures,
\emph{Materials Physics and Mechanics}, {\bf 29} (2016) 107-115.

\bibitem{Oleinik_1988}
O.A. Oleinik, On the frequencies of the eigenoscilations of bodies with concentrated masses, in book
\emph{Functional and numerical mathods of mathematical physics}, (Naukova Dumka, Kiev, 1988), pp.
101-128. (in Russian)

\bibitem{OYS}
 O.A. Oleinik, G.A. Yosifian and A.S. Shamaev, \emph{Mathematical Problems in Elasticity and Homogenization,}
(Amsterdam: North-Holland, 1992) 

\bibitem{P-P-Stokes-1-2015}
G. Panasenko and K. Pileckas,  Asymptotic analysis of the non-steady Navier-Stokes equations in
a tube structure. I. The case without boundary-layer in time. \emph{Nonlinear Analysis.} \textbf{122} (2015) 125-168.


\bibitem{P-P-Stokes-2-2015}
G. Panasenko and K. Pileckas,  Asymptotic analysis of the non-steady Navier-Stokes equations in
a tube structure. II. General case. \emph{Nonlinear Analysis.} \textbf{125} (2015) 582-607.

\bibitem{Perez-Nazarov}
E. P\'erez and S.A. Nazarov,  {New asymptotic effects for the spectrum of problems
on concentrated masses near the boundary,} 
\emph{C.R. Acad. Sci. Paris. Mecanique}, {\bf 337} (2009) 585-590.


\bibitem{Post-2012}
 O. Post,   \emph{Spectral Analysis on Graph-Like Spaces}, (Lecture Notes,  Springer, 2012).

\bibitem{S-P}
 E. S\'anchez-Palencia,  Perturbation of eigenvalues in
thermo-elasticity and vibration of systems with concentrated
masses. \emph{Trends and Applications of Pure Mathematics to
Mechanics. Lecture Notes in Phys.} {\bf 195 } (1984) 346-368.


\bibitem{Vishik}
M. I. Vishik and L. A. Lyusternik, Regular degeneration and boundary layer for linear differential equations with parameter,
\emph{Amer. Math. Soc. Transl.} {\bf 20}:2  (1962)  239-364.

\end{thebibliography}
\end{document}